\documentclass{article}

\usepackage{a4wide}
\usepackage{graphicx}
\usepackage{subfigure}
\usepackage{url}
\usepackage{amsmath,amsthm,amssymb}
\usepackage{multirow}
\usepackage{booktabs}
\usepackage{fancyhdr}
\usepackage{color,soul}
\allowdisplaybreaks

\newtheorem{Theorem}{Theorem}
\newtheorem{Lemma}[Theorem]{Lemma}
\newtheorem{Definition}[Theorem]{Definition}
\newtheorem{Remark}[Theorem]{Remark}

\newcommand\blfootnote[1]{%
	\begingroup
	\renewcommand\thefootnote{}\footnote{#1}%
	\addtocounter{footnote}{-1}%
	\endgroup
}

\begin{document}

\pagestyle{fancy}

\fancyhf{}
\fancyhead[L]{Dynamics of a Model of Polluted Lakes}
\fancyfoot[C]{\thepage}

\title{Dynamics of a Model of Polluted Lakes via Fractal-Fractional Operators 
with Two Different Numerical Algorithms\blfootnote{This is a preprint 
of a paper whose final and definite form is published Open Access 
in \emph{Chaos Solitons Fractals} at [https://doi.org/10.1016/j.chaos.2024.114653].}}


\author{%
Tanzeela Kanwal$^{1,*}$
\and Azhar Hussain$^{1,2}$ 
\and \.{I}brahim Avc{\i}$^{3}$ 
\and Sina Etemad$^{4,5}$ 
\and Shahram Rezapour$^{4,6,7,*}$ 
\and Delfim F. M. Torres$^{8,}$\thanks{Correspondence: \texttt{tanzeelakanwal16@gmail.com} (T.K.);
\texttt{sh.rezapour@azaruniv.ac.ir} (S.R.); \texttt{delfim@ua.pt} (D.F.M.T.)} 
}

\date{%
$^{1}$Department of Mathematics, University of Sargodha, Sargodha 40100, Pakistan; \texttt{tanzeelakanwal16@gmail.com} (T.K.)\\
$^{2}$Department of Mathematics, University of Chakwal, Chakwal 48800, Pakistan; \texttt{azhar.hussain@uoc.edu.pk} (A.H.)\\
$^{3}$Department of Computer Engineering, Faculty of Engineering, Final International University, Kyrenia,
Northern Cyprus, via Mersin 10, Turkey; \texttt{ibrahim.avci@final.edu.tr} (\.{I}.A.)\\
$^{4}$Department of Mathematics, Azarbaijan Shahid Madani University, Tabriz, Iran;
\texttt{sina.etemad@azaruniv.ac.ir} (S.E.); \texttt{sh.rezapour@azaruniv.ac.ir} (S.R.)\\
$^{5}$Mathematics in Applied Sciences and Engineering Research Group,\\ 
Scientific Research Center, Al-Ayen University, Nasiriyah 64001, Iraq\\
$^{6}$Department of Mathematics, Kyung Hee University, 26 Kyungheedae-ro, Dongdaemun-gu, Seoul, Republic of Korea\\
$^{7}$Department of Medical Research, China Medical University Hospital, China Medical University, Taichung, Taiwan\\
$^{8}$Center for Research and Development in Mathematics and Applications (CIDMA),
Department of Mathematics, University of Aveiro, 3810-193 Aveiro, Portugal; \texttt{delfim@ua.pt} (D.F.M.T.)}

\maketitle


\begin{abstract}
We employ Mittag--Leffler type kernels to solve a system of fractional
differential equations using fractal-fractional (FF) operators with two fractal
and fractional orders. Using the notion of FF-derivatives with nonsingular and
nonlocal fading memory, a model of three polluted lakes with one source of pollution
is investigated. The properties of a non-decreasing and compact mapping are used
in order to prove the existence of a solution for the FF-model of polluted lake system.
For this purpose, the Leray--Schauder theorem is used. After exploring stability
requirements in four versions, the proposed model of polluted lakes system is then
simulated using two new numerical techniques based on Adams--Bashforth
and Newton polynomials methods. The effect of fractal-fractional differentiation
is illustrated numerically. Moreover, the effect of the FF-derivatives
is shown under three specific input models of the pollutant:
linear, exponentially decaying, and periodic.

\medskip

\noindent {\bf Keywords:} Pollution of waters; Fractal-fractional derivatives model;
Existence, unicity and stability; Adams--Bashforth and Newton polynomials methods.

\noindent {\bf MSC:} 34A08; 65P99.
\end{abstract}


\section{Introduction}

In the last century, pollution of waters has become a severe danger to the world we live in.
The first step in preparing to conserve the natural environment is to monitor pollution levels.
Monitoring pollution is possible to achieve with the use of mathematical analysis.
Differential equations may be used to simulate environmental contamination,
just as they can be used in many other fields. For example, Biazar {\it et al.}
utilized in 2006 a set of differential equations to predict the pollution level
in a series of lakes \cite{baz}. In concrete, they have proposed a model of triple
lakes connected by channels through compartment modeling. Some other scholars
have investigated this concept using various methodologies. Y\"{u}zba\c{s}i {\it et al.} \cite{wv1}
analyzed such levels of pollution under the collocation method in 2012. Later, Benhammouda {\it et al.}
\cite{wv2} utilized another method to solve the pollution model via a modified differential transform.
Khader {\it et al.} \cite{khad} have also created a fractional case model and used the matrix properties
in 2013. Recently, in 2019, Bildik and Deniz \cite{pp1} considered an Atangana--Baleanu based model
for approximating the solutions of a polluted lake system. After that, Ahmed and Khan turned
to a similar model of lake pollution via different fractional methods \cite{pp2}. In 2020,
Prakasha and Veeresha \cite{pp3} solved such a system of polluted lakes via the so-called q-HATM method.
More recently, in 2022, Shiri and Baleanu have done a research on the amount of pollution
in a three-compartmental model and derived some analytical results \cite{pp4}.
During these years, fractional models of real-world processes have been studied
by many other researchers, showing the applicability of fractional operators
in mathematical modeling: see, e.g., \cite{z1,MR4376325,z4,z5,z7,z9,MR4536746}.
Here we propose and study a mathematical model via a generalized family
of derivatives equipped with two parameters.

Atangana introduced a new class of fractal-fractional notions,
which brings together the two applicable areas of fractal and fractional calculi \cite{ab}.
The structure of these operators is a convolution of the power-law, exponential law,
and modified Mittag--Leffler law with fractal derivatives, which establishes
a connection between fractional and fractal mathematics. The fractal
dimension and order are the two components of these operators and
differential equations with fractal-fractional derivatives convert the putative system's
order and dimension into a rational system. Because of this characteristic,
conventional differential equations are naturally extended to systems with
any order of derivatives and dimensions. The goal of these coupled operators
is to look at distinct nonlocal boundary value problems (BVPs) or
initial value problems (IVPs) that have fractal tendencies in nature.
Many scholars provided results and discoveries in this area, demonstrating
that fractal-fractional operators are more effective at describing
real-world data and for mathematical modeling. Examples of such mathematical models include:
fractal-fractional structures of dynamics of corona viruses \cite{FF01},
malaria transmission \cite{FF02}, dynamics of COVID-19 in Wuhan \cite{FF03},
transmission of AH1N1/09 virus \cite{z11b}, dynamics of Q fever \cite{FF04},
HIV \cite{FF}, dynamics of CD4$^{+}$ cells \cite{z12}, tuberculosis disease \cite{z13}, etc.

The incorporation of fractal-fractional (FF) operators with dual fractal 
and fractional orders in scientific research presents a promising avenue 
with multifaceted advantages. By leveraging two orders simultaneously, 
this approach allows for a more nuanced and refined representation of 
complex systems, capturing intricate patterns and irregularities that 
traditional methods might overlook. The synergy of fractal geometry and 
fractional calculus enhances the modeling and analysis of real-world phenomena, 
providing a more accurate reflection of the inherent self-similar structures 
and non-integer order dynamics. This not only refines our understanding of 
intricate processes but also facilitates the development of more robust 
mathematical models that can be applied across various disciplines. 
The utilization of FF operators holds the potential to revolutionize fields 
ranging from signal processing to image analysis, offering a versatile 
toolkit to address challenges that demand a deeper comprehension of intricate, 
multifractal behaviors. Embracing this innovative paradigm contributes 
to a more holistic and precise approach in scientific investigations, 
opening new frontiers for exploration and discovery.
Here we conduct an analysis of a fractal-fractional model of polluted lakes
in terms of various different characteristics.

The paper is organized as follows. In Section~\ref{sec:2},
we introduce a fractal-fractional system to model polluted lakes.
Existence of a solution to the proposed system is proved in Section~\ref{sec:3}
by using the Leray--Schauder theorem. In Section~\ref{sec:4},
we employ the Banach principle for contractions
to demonstrate uniqueness of solution. Furthermore,
using functional analysis, numerous requirements
for different types of stability for the solution to the polluted lakes system model
are explored in Section~\ref{sec:5}. To simulate our model, we use two different techniques:
a fractional Adams--Bashforth approach (Section~\ref{sec:6}) and a second one
based on Newton's polynomials (Section~\ref{sec:7}). The obtained theoretical results
are then tested in Section~\ref{sec:8} by applying our algorithms with some concrete
data under various fractal and fractional order values in three different cases:
linear, exponentially decaying and periodic input real models.
We end with Section~\ref{sec:9} of conclusion.


\section{The FF-model for a polluted system of three lakes}
\label{sec:2}

We model three lakes. Using three lakes in a system might 
be a good practical choice based on various factors such as land 
availability, cost, and efficiency. The decision of considering 
here three lakes is not purely mathematical, but involves environmental, 
economic, and logistical considerations. Mathematical modeling could 
help optimize the distribution and size of the lakes, but it is essential 
to balance these factors for a sustainable and effective solution. Therefore, 
we restrict ourselves to three lakes and their channels with a pollutant source. 
One can generalize our results to a finite number of lakes.

Each lake is treated as a compartment,
a linking channel between two lakes being viewed as a pipe
connecting the compartments. The direction of the flow across
each channel or pipeline is shown by arrows. A contaminant $c$
is considered in the first lake. By $c(s)$ we denote the rate
at which the contaminant/pollutant
enters Lake~1 at time $s$. The major purpose is to determine
the pollution levels in each lake at any given moment.
To do so, we regard the concentration $C_i(s)$ of the pollutant
in the lake $i$ at time $s$, $s\geq0$, by
\begin{equation}
C_i(s) = \frac{L_i(s)}{V_i},
\end{equation}
where $V_i$ denotes the water volume at lake $i$, $i \in \{1, 2, 3\}$,
assumed to be constant, and $L_i(s)$ specifies the quantity of pollution
that is equally distributed over each lake at time $s$. We are interested to model
the situation shown in Figure~\ref{Fig0}, where we use the symbol $F_{ji}$
to represent the flow rate entering the $j$th lake from the $i$th.
\begin{figure}[ht!]
\centering
\includegraphics[width=0.8\linewidth]{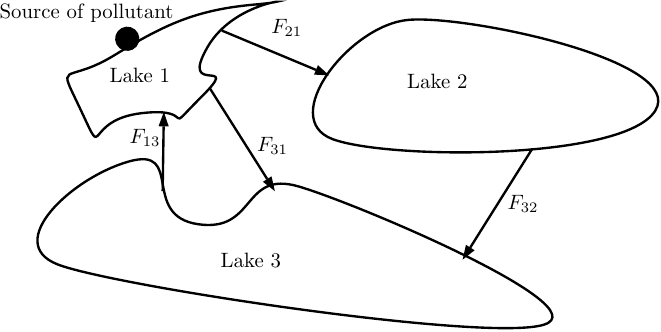}
\caption{Schematic of channels interconnecting the three lakes being modeled.}
\label{Fig0}
\end{figure}
Based on Figure~\ref{Fig0}, we derive the following conditions:
\begin{equation}
\begin{split}
Lake~1&: F_{13} = F_{31} + F_{21},\\
Lake~2&: F_{21} = F_{32},\\
Lake~3&: F_{32} + F_{31} = F_{13}.
\end{split}
\end{equation}
Note that $F_{12}= 0$ since there exists no pipe between the second and the first lakes.
The flux $\mathcal{F}_{ji}(s)$ of pollutant flowing from the $i$th lake to the $j$th lake
at an arbitrary time $s$ measures the flow rate of the concentration of pollutant.
This index equals
\begin{equation}
\mathcal{F}_{ji}(s) = F_{ji} C_i(s) = F_{ji} \frac{L_i(s)}{V_i}.
\end{equation}
Based on the principle that the rate of change of the pollutant is given
by the difference between the input rate and the output rate, we propose
here the following fractal-fractional model for the dynamic behavior
of the polluted lake system of three lakes
via the generalized Mittag--Leffler kernel:
\begin{equation}
\label{model2}
\begin{cases}
{}^{\mathbf{FFML}}\mathcal{D}_{0,s}^{(\theta,\sigma)}{L_1}(s)
= \frac{F_{13}}{V_3}L_3(s)+c(s)-\frac{F_{31}}{V_1}L_1(s)
-\frac{F_{21}}{V_1}L_1(s),\\[0.3cm]
{}^{\mathbf{FFML}}\mathcal{D}_{0,s}^{(\theta,\sigma)}{L_2}(s)
= \frac{F_{21}}{V_1}L_1(s)-\frac{F_{32}}{V_2}L_2(s),\\[0.3cm]
{}^{\mathbf{FFML}}\mathcal{D}_{0,s}^{(\theta,\sigma)}{L_3}(s)
=\frac{F_{31}}{V_1}L_1(s)+\frac{F_{32}}{V_2}L_2(s)
-\frac{F_{13}}{V_3}L_3(s),
\end{cases}
\end{equation}
subject to
\begin{equation}
\label{model2:IC}
{L_1}(0)=L_{1,0} \geq 0,
\quad {L_2}(0)= L_{2,0} \geq 0,
\quad {L_3}(0)=L_{3,0} \geq 0,
\end{equation}
where ${}^{\mathbf{FFML}}\mathcal{D}_{0,s}^{(\theta,\sigma)}$
is the $(\theta,\sigma)$-fractal-fractional
derivative with Mittag--Leffler type kernel
of fractional and fractal orders
$\theta \in (0,1]$ and $\sigma \in (0,1]$,
respectively, as introduced by Atangana in \cite{ab}.

\begin{Definition}[See \cite{ab}]
Let $f : (a,b) \to [0,\infty)$ be a continuous map that is fractal
differentiable of dimension $\sigma$. In this case, the Riemann--Liouville
$(\theta,\sigma)$-fractal-fractional derivative of $f$ with the generalized
Mittag--Leffler type kernel of order $\theta$ is given by
\begin{equation}
\label{eq21}
{}^{\mathbf{FFML}}\mathcal{D}_{a,s}^{(\theta,\sigma)}f(s)
=\frac{\mathcal{AB}(\theta)}{1-\theta} \frac{\mathrm{d}}{\mathrm{d}s^\sigma}
\int_a^{s} \mathbb{E}_\theta \left[ - \frac{\theta}{1-\theta }
(s-\mathfrak{w})^{\theta} \right] f ( \mathfrak{w})\,
\mathrm{d}\mathfrak{w},~~~0< \theta, \sigma \leq 1,
\end{equation}
where
$$
\dfrac{\mathrm{d}f(\mathfrak{w})}{\mathrm{d}\mathfrak{w}^\sigma}
= \lim_{s \to \mathfrak{w}} \dfrac{f(s)
-f(\mathfrak{w})}{s^\sigma - \mathfrak{w}^\sigma}
$$
is the fractal derivative and $\mathcal{AB}(\theta)
= 1-\theta + \dfrac{\theta }{\Gamma(\theta)}$
with $ \mathcal{AB}(0)= \mathcal{AB}(1) = 1$.
\end{Definition}

In what follows, we also use the corresponding notion of fractal-fractional integral.

\begin{Definition}[See \cite{ab}]
The $(\theta,\sigma)$-fractal-fractional
integral of a function $f$ with generalized kernel
is given by
\begin{equation}
\label{eq22}
{}^{\mathbf{FFML}}\mathcal{I}_{a,s}^{(\theta,\sigma)}f(s)
= \frac{\theta\sigma}{\mathcal{AB}(\theta) \Gamma(\theta)}
\int_a^s \mathfrak{w}^{\sigma -1} (s- \mathfrak{w})^{\theta-1}
f(\mathfrak{w})\, \mathrm{d}\mathfrak{w}
+\frac{(1-\theta)\sigma s^{\sigma-1}}{\mathcal{AB}(\theta)} f(s),
\end{equation}
if it exists, where $\theta, \sigma >0$.
\end{Definition}


\section{Existence}
\label{sec:3}

We begin by proving existence of solution to our problem
\eqref{model2}--\eqref{model2:IC}. For that we use fixed point theory.
To conduct our qualitative analysis, let us define the Banach space
$\mathbb{X} = \mathbb{C}^3$, where $\mathbb{C} = C(\mathbb{J},\mathbb{R})$ with
$$
\Vert \mathbb{K} \Vert_{\mathbb{X}}
= \Vert \big( {L_1}, {L_2}, {L_3}  \big) \Vert_{\mathbb{X}}
= \max \big\{ \vert W(s) \vert  :~~  s\in \mathbb{J} \big\},
$$
for $ \vert W \vert := \vert {L_1} \vert
+  \vert {L_2} \vert + \vert {L_3} \vert$.
We rewrite the right-hand-side of the fractal-fractional
polluted lake system \eqref{model2} as
\begin{equation}
\label{eq41}
\begin{cases}
\mathbb{Q}_1 \big( s, {L_1}(s), {L_2}(s), {L_3}(s) \big)
= \frac{F_{13}}{V_3}L_3(s)+c(s)-\frac{F_{31}}{V_1}L_1(s)
-\frac{F_{21}}{V_1}L_1(s),\\[0.3cm]
\mathbb{Q}_2 \big( s, {L_1}(s), {L_2}(s), {L_3}(s) \big)
= \frac{F_{21}}{V_1}L_1(s)-\frac{F_{32}}{V_2}L_2(s),\\[0.3cm]
\mathbb{Q}_3 \big( s, {L_1}(s), {L_2}(s), {L_3}(s) \big)
= \frac{F_{31}}{V_1}L_1(s)+\frac{F_{32}}{V_2}L_2(s)
-\frac{F_{13}}{V_3}L_3(s).
\end{cases}
\end{equation}
In this case, the fractal-fractional polluted lake system \eqref{model2}
is transformed into the following system:
\begin{equation}
\label{eq42}
\begin{cases}
{}^{\mathbf{ABR}}\mathcal{D}_{0,s}^{\theta} {L_1}(s)
= \sigma s^{\sigma - 1} \mathbb{Q}_1 \big( s, {L_1}(s),
{L_2}(s), {L_3}(s) \big), \\[0.3cm]
{}^{\mathbf{ABR}}\mathcal{D}_{0,s}^{\theta} {L_2}(s)
= \sigma s^{\sigma - 1} \mathbb{Q}_2 \big( s, {L_1}(s),
{L_2}(s), {L_3}(s) \big), \\[0.3cm]
{}^{\mathbf{ABR}}\mathcal{D}_{0,s}^{\theta} {L_3}(s)
= \sigma s^{\sigma - 1} \mathbb{Q}_3 \big( s, {L_1}(s),
{L_2}(s), {L_3}(s) \big).
\end{cases}
\end{equation}
In view of \eqref{eq42}, we rewrite our tree-state system
as the compact IVP
\begin{equation}
\label{eq43}
\begin{cases}
{}^{\mathbf{ABR}}\mathcal{D}_{0,s}^{\theta} \mathbb{K}(s)
= \sigma s^{\sigma - 1} \mathbb{Q} \big(s, \mathbb{K}(s)\big), \\[0.3cm]
\mathbb{K}(0)=\mathbb{K}_0,
\end{cases}
\end{equation}
where
\begin{equation}
\label{eq44}
\mathbb{K}(s) = \big({L_1}(s), {L_2}(s), {L_3}(s) \big)^T,
\quad\quad \mathbb{K}_0 = \big( L_{1,0}, L_{2,0}, L_{3,0} \big)^T,
~~~~\theta , \sigma \in (0,1],
\end{equation}
and
\begin{equation}
\label{eq45}
\mathbb{Q} \big(s, \mathbb{K}(s)\big)
=
\begin{cases}
\mathbb{Q}_1 \big(s, {L_1}(s), {L_2}(s), {L_3}(s) \big), \\[0.3cm]
\mathbb{Q}_2 \big(s, {L_1}(s), {L_2}(s), {L_3}(s) \big), \\[0.3cm]
\mathbb{Q}_3 \big(s, {L_1}(s), {L_2}(s), {L_3}(s) \big),
\quad s\in \mathbb{J}.
\end{cases}
\end{equation}
By definition and by \eqref{eq43}, we have
\begin{equation}
\label{eq4444}
\frac{\mathcal{AB}(\theta)}{1-\theta} \frac{\mathrm{d}}{\mathrm{d}s}
\int_0^{s} \mathbb{E}_\theta \left[ - \frac{\theta}{1-\theta }
(s-\mathfrak{w})^{\theta} \right] \mathbb{K} ( \mathfrak{w})\,
\mathrm{d}\mathfrak{w} = \sigma s^{\sigma - 1} \mathbb{Q} \big(s, \mathbb{K}(s)\big).
\end{equation}
Applying the fractal-fractional Atangana--Baleanu integral on \eqref{eq4444}, we get
\begin{equation}
\label{eq46}
\mathbb{K}(s) = \mathbb{K}(0) + \frac{\theta\sigma}{\mathcal{AB}(\theta)
\Gamma(\theta)} \int_0^s \mathfrak{w}^{\sigma -1} (s- \mathfrak{w})^{\theta-1}
\mathbb{Q}(\mathfrak{w}, \mathbb{K}(\mathfrak{w}))\, \mathrm{d}\mathfrak{w}
+\frac{(1-\theta)\sigma s^{\sigma-1}}{\mathcal{AB}(\theta)} \mathbb{Q}(s, \mathbb{K}(s)).
\end{equation}
The extended representation of \eqref{eq46} is given by
\begin{equation}
\label{eq477}
\begin{cases}
\displaystyle {L_1}(s) = L_{1,0}
+\frac{(1-\theta)\sigma s^{\sigma-1}}{\mathcal{AB}(\theta)}
\mathbb{Q}_1(s, {L_1}(s),{L_2}(s),{L_3}(s))\\[0.4cm]
~~~~~~~~~+ \displaystyle\frac{\theta\sigma}{\mathcal{AB}(\theta) \Gamma(\theta)}
\int_0^s \mathfrak{w}^{\sigma -1} (s- \mathfrak{w})^{\theta-1} \mathbb{Q}_1(\mathfrak{w},
{L_1}(\mathfrak{w}),{L_2}(\mathfrak{w}),{L_3}(\mathfrak{w}))\,
\mathrm{d}\mathfrak{w}, \\[0.6cm]
\displaystyle {L_2}(s) = L_{2,0}
+\frac{(1-\theta)\sigma s^{\sigma-1}}{\mathcal{AB}(\theta)}
\mathbb{Q}_2(s, {L_1}(s),{L_2}(s),{L_3}(s))\\[0.4cm]
~~~~~~~~~+ \displaystyle\frac{\theta\sigma}{\mathcal{AB}(\theta) \Gamma(\theta)}
\int_0^s \mathfrak{w}^{\sigma -1} (s- \mathfrak{w})^{\theta-1} \mathbb{Q}_2(\mathfrak{w},
{L_1}(\mathfrak{w}),{L_2}(\mathfrak{w}),{L_3}(\mathfrak{w}))
\, \mathrm{d}\mathfrak{w}, \\[0.6cm]
\displaystyle {L_3}(s) = L_{3,0}
+\frac{(1-\theta)\sigma s^{\sigma-1}}{\mathcal{AB}(\theta)}
\mathbb{Q}_3(s, {L_1}(s),{L_2}(s),{L_3}(s))\\[0.4cm]
~~~~~~~~~+ \displaystyle\frac{\theta\sigma}{\mathcal{AB}(\theta)
\Gamma(\theta)} \int_0^s \mathfrak{w}^{\sigma -1}
(s- \mathfrak{w})^{\theta-1} \mathbb{Q}_3(\mathfrak{w},
{L_1}(\mathfrak{w}),{L_2}(\mathfrak{w}),
{L_3}(\mathfrak{w}))\, \mathrm{d}\mathfrak{w}.
\end{cases}
\end{equation}
To derive a fixed-point problem,
we now define the self-map
$F: \mathbb{X} \to \mathbb{X}$ as
\begin{equation}
\label{eq47}
F(\mathbb{K}(s)) = \mathbb{K}(0)
+\frac{(1-\theta)\sigma s^{\sigma-1}}{\mathcal{AB}(\theta)}
\mathbb{Q}(s, \mathbb{K}(s))
+\frac{\theta\sigma}{\mathcal{AB}(\theta) \Gamma(\theta)}
\int_0^s \mathfrak{w}^{\sigma -1} (s- \mathfrak{w})^{\theta-1}
\mathbb{Q}(\mathfrak{w}, \mathbb{K}(\mathfrak{w}))\, \mathrm{d}\mathfrak{w}.
\end{equation}
	
To prove existence of solution to our fractal-fractional polluted
lake system \eqref{model2}, we make use of the following Leray--Schauder theorem.

\begin{Theorem}[Leray--Schauder fixed point theorem \cite{29}]
\label{Schauder}
Let $\mathbb{X}$ be a Banach space, $\mathbb{E} \subset \mathbb{X}$
a closed convex and bounded set, and $\mathbb{O} \subset \mathbb{E}$
an open set with $0\in \mathbb{O}$. Then, under the compact and continuous mapping
$F : \bar{\mathbb{O}} \to \mathbb{E}$, either:
\begin{enumerate}
\item[(Y1)] $\exists \, y \in \bar{\mathbb{O}}\, s.t.\, y = F(y)$, or
\item[(Y2)] $\exists\, y \in \partial \mathbb{O}$, $\mu \in (0,1)$
such that $y = \mu F (y)$.
\end{enumerate}
\end{Theorem}

Given that the polluted lake system models a real-world problem, 
its existence is subject to certain constraints. These constraints, 
denoted in Theorem~\ref{th33} as (P1) and (P2), play a crucial role 
in shaping the dynamics and characteristics of the system. Indeed,
(P1) and (P2) are indispensable to define and regulate the behavior 
of the polluted lake system within the confines of practicality and reality. 
Recognizing these constraints is essential for constructing a comprehensive 
understanding of the system and developing effective strategies.

\begin{Theorem}
\label{th33}
Let $\mathbb{Q}\in C(\mathbb{J}\times \mathbb{X}, \mathbb{X})$. If
\begin{itemize}
\item[(P1)] $\exists\,\varphi \in L^1(\mathbb{J}, \mathbb{R}^+)$
and $\exists\,A \in C([0,\infty ), (0,\infty))$ ($A$ non-decreasing)
such that $\forall\,s\in \mathbb{J}$ and $\mathbb{K} \in \mathbb{X}$,
$$
\big\vert \mathbb{Q}(s, \mathbb{K}(s)) \big\vert
\leq \varphi(s) A ( \vert \mathbb{K}(s) \vert );
$$
\item[(P2)] $\exists\,\omega>0$ such that
\begin{equation}
\label{eqK}
\frac{\omega}{\mathbb{K}_0 + \left[  \dfrac{(1-\theta)
\sigma S^{\sigma-1}}{\mathcal{AB}(\theta)}
+  \dfrac{\theta\sigma S^{\theta+\sigma - 1}
\Gamma(\sigma)}{\mathcal{AB}(\theta)
\Gamma(\theta + \sigma)} \right] \varphi_0^* A(\omega)}> 1
\end{equation}
with $\varphi_0^* = \sup_{s\in \mathbb{J}} \vert \varphi(s) \vert$;
\end{itemize}
then there exists a solution to the fractal-fractional
polluted lake system \eqref{model2}.
\end{Theorem}

\begin{proof}
First, consider $F: \mathbb{X} \rightarrow \mathbb{X}$,
which is formulated in~\eqref{eq47}, and assume
$$
N_r = \Big\{ \mathbb{K}\in \mathbb{X}\, :\, \Vert \mathbb{K}\Vert_{\mathbb{X}} \leq r \Big\},
$$
for some $r >0$. Clearly, as $\mathbb{Q}$ is continuous, 
thus $F$ is also so. From (P1), we get
\begin{align*}
\big\vert F (\mathbb{K}(s)) \big\vert
&\leq \big\vert \mathbb{K}(0) \big\vert
+ \frac{(1-\theta)\sigma s^{\sigma-1}}{\mathcal{AB}(\theta)} \big\vert
\mathbb{Q}(s, \mathbb{K}(s)) \big\vert \\[0.4cm]
&\quad + \frac{\theta\sigma}{\mathcal{AB}(\theta) \Gamma(\theta)}
\int_0^s \mathfrak{w}^{\sigma -1} (s- \mathfrak{w})^{\theta-1} \big\vert
\mathbb{Q}(\mathfrak{w}, \mathbb{K}(\mathfrak{w})) \big\vert \, \mathrm{d}\mathfrak{w} \\[0.4cm]
&\leq \mathbb{K}_0 + \frac{(1-\theta)\sigma s^{\sigma-1}}{\mathcal{AB}(\theta)}
\varphi (s) A (\vert \mathbb{K}(s)\vert) \\[0.4cm]
&\quad +  \frac{\theta\sigma}{\mathcal{AB}(\theta) \Gamma(\theta)}
\int_0^s \mathfrak{w}^{\sigma -1} (s- \mathfrak{w})^{\theta-1}
\varphi (\mathfrak{w}) A (\vert \mathbb{K}(\mathfrak{w})\vert)
\, \mathrm{d}\mathfrak{w} \\[0.4cm]
&\leq \mathbb{K}_0 + \frac{(1-\theta)\sigma S^{\sigma-1}}{\mathcal{AB}(\theta)}
\varphi_0^* A (r) +  \frac{\theta\sigma S^{\theta+\sigma - 1}
B(\theta,\sigma) }{\mathcal{AB}(\theta) \Gamma(\theta)} \varphi_0^* A(r) \\[0.4cm]
&= \mathbb{K}_0 + \frac{(1-\theta)\sigma S^{\sigma-1}}{\mathcal{AB}(\theta)}
\varphi_0^* A (r) +  \frac{\theta\sigma S^{\theta+\sigma - 1}
\Gamma(\sigma) }{\mathcal{AB}(\theta) \Gamma(\theta + \sigma)} \varphi_0^* A(r),
\end{align*}
for $\mathbb{K}\in N_r$. Hence,
\begin{equation}
\label{eq49}
\Vert F \mathbb{K} \Vert_{\mathbb{X}} \leq \mathbb{K}_0
+ \left[  \frac{(1-\theta)\sigma S^{\sigma-1}}{\mathcal{AB}(\theta)}
+  \frac{\theta\sigma S^{\theta+\sigma - 1} \Gamma(\sigma) }{\mathcal{AB}(\theta)
\Gamma(\theta + \sigma)} \right] \varphi_0^* A(r) < \infty.
\end{equation}
Thus, $F$ is uniformly bounded on $\mathbb{X}$. Now, take $s, v \in [0,S]$
such that $s<v$ and $\mathbb{K}\in N_r$. By denoting
$$
\sup_{(s,\mathbb{K})\in \mathbb{J}\times N_r}
\vert \mathbb{Q}(s, \mathbb{K}(s))\vert = \mathbb{Q}^* < \infty,
$$
we estimate
\begin{align}
\label{eq410}
\big\vert F (\mathbb{K}(v)) - F (\mathbb{K}(s)) \big\vert
&\leq \Bigg\vert \frac{(1-\theta)\sigma v^{\sigma-1}}{\mathcal{AB}(\theta)}
\mathbb{Q}(v, \mathbb{K}(v)) - \frac{(1-\theta)\sigma
s^{\sigma-1}}{\mathcal{AB}(\theta)} \mathbb{Q}(s, \mathbb{K}(s))  \nonumber\\[0.3cm]
&\qquad +\frac{\theta\sigma}{\mathcal{AB}(\theta) \Gamma(\theta)}
\int_0^v \mathfrak{w}^{\sigma -1} (v- \mathfrak{w})^{\theta-1}
\mathbb{Q}(\mathfrak{w}, \mathbb{K}(\mathfrak{w}))\, \mathrm{d}\mathfrak{w} \nonumber\\[0.4cm]
&\qquad -\frac{\theta\sigma}{\mathcal{AB}(\theta) \Gamma(\theta)}
\int_0^s \mathfrak{w}^{\sigma -1} (s- \mathfrak{w})^{\theta-1}
\mathbb{Q}(\mathfrak{w}, \mathbb{K}(\mathfrak{w}))\, \mathrm{d}\mathfrak{w} \Bigg\vert \nonumber\\[0.4cm]
&\leq \frac{(1-\theta)\sigma \mathbb{Q}^* }{\mathcal{AB}(\theta)} (v^{\sigma-1} - s^{\sigma-1}) \\[0.3cm]
&\qquad + \frac{\theta\sigma\mathbb{Q}^*}{\mathcal{AB}(\theta) \Gamma(\theta)}
\left\vert \int_0^{v} \mathfrak{w}^{\sigma - 1}(v-\mathfrak{w})^{\theta -1}
\, \mathrm{d}\mathfrak{w} - \int_0^{s} \mathfrak{w}^{\sigma -1}(s-\mathfrak{w})^{\theta -1}
\, \mathrm{d}\mathfrak{w} \right\vert \nonumber\\[0.4cm]
&\leq \frac{(1-\theta)\sigma \mathbb{Q}^* }{\mathcal{AB}(\theta)}
(v^{\sigma-1} - s^{\sigma-1}) + \frac{\theta\sigma
\mathbb{Q}^* B(\theta , \sigma)}{\mathcal{AB}(\theta)\Gamma(\theta)}
\big[ v^{\theta + \sigma - 1} - s^{\theta + \sigma - 1}  \big] \nonumber\\[0.4cm]
&= \frac{(1-\theta)\sigma \mathbb{Q}^*}{\mathcal{AB}(\theta)} (v^{\sigma-1}
- s^{\sigma-1}) +  \frac{\theta\sigma \mathbb{Q}^* \Gamma(\sigma) }{\mathcal{AB}(\theta)
\Gamma(\theta + \sigma)} \big[ v^{\theta + \sigma - 1} - s^{\theta + \sigma - 1}  \big]. \nonumber
\end{align}
We see that the right-hand side of \eqref{eq410} approaches to $0$
independent of $\mathbb{K}$, as $v \to s$. Consequently,
$$
\Vert F (\mathbb{K}(v)) - F (\mathbb{K}(s)) \Vert_{\mathbb{X}} \to 0,
$$
when $v\to s$. This gives the equicontinuity of $F$ and, accordingly, the compactness
of $F$ on $N_r$ by the Arzel\'{a}--Ascoli thoerem. As Theorem~\ref{Schauder}
is fulfilled on $F$, we have one of (Y1) or~(Y2). From~(P2), we set
$$
\Phi := \Big\{ \mathbb{K}\in \mathbb{X}\, : \, ~\Vert \mathbb{K}\Vert_{\mathbb{X}} < \omega \Big\},
$$
for some $\omega>0$, such that
$$
\mathbb{K}_0 + \left[  \dfrac{(1-\theta)\sigma S^{\sigma-1}}{\mathcal{AB}(\theta)}
+  \dfrac{\theta\sigma S^{\theta+\sigma - 1} \Gamma(\sigma) }{\mathcal{AB}(\theta)
\Gamma(\theta + \sigma)} \right] \varphi_0^* A(\omega) < \omega.
$$
From~(P1) and \eqref{eq49}, we have
\begin{equation}\label{eq411}
\Vert F \mathbb{K} \Vert_{\mathbb{X}}
\leq \mathbb{K}_0 + \left[  \frac{(1-\theta)\sigma S^{\sigma-1}}{\mathcal{AB}(\theta)}
+ \frac{\theta\sigma S^{\theta+\sigma - 1} \Gamma(\sigma) }{\mathcal{AB}(\theta)
\Gamma(\theta + \sigma)} \right] \varphi_0^* A(\Vert \mathbb{K}\Vert_\mathbb{X}).
\end{equation}
Suppose that there are $\mathbb{K}\in \partial \Phi$ and $0< \mu < 1$ such that
$\mathbb{K} = \mu F(\mathbb{K})$. Then, by~\eqref{eq411}, we write
\begin{align*}
\omega = \Vert \mathbb{K} \Vert_{\mathbb{X}}
= \mu \Vert F \mathbb{K}\Vert_{\mathbb{X}}
&< \mathbb{K}_0 + \left[  \frac{(1-\theta)\sigma S^{\sigma-1}}{\mathcal{AB}(\theta)}
+ \frac{\theta\sigma S^{\theta+\sigma - 1} \Gamma(\sigma) }{\mathcal{AB}(\theta)
\Gamma(\theta + \sigma)} \right] \varphi_0^* A(\Vert \mathbb{K}\Vert_\mathbb{X}) \\[0.3cm]
&< \mathbb{K}_0 + \left[  \frac{(1-\theta)\sigma S^{\sigma-1}}{\mathcal{AB}(\theta)}
+  \frac{\theta\sigma S^{\theta+\sigma - 1} \Gamma(\sigma) }{\mathcal{AB}(\theta)
\Gamma(\theta + \sigma)} \right] \varphi_0^* A(\omega) < \omega,
\end{align*}
which cannot hold true. Thus, (Y2) is not satisfied and $F$ admits a fixed-point in
$\bar{\Phi}$ by Theorem~\ref{Schauder}. This proves the existence of a solution
to the FF polluted lake model \eqref{model2}.
\end{proof}


\section{Uniqueness}
\label{sec:4}

As a first step to prove uniqueness of solution to our problem
\eqref{model2}--\eqref{model2:IC}, we begin by investigating
a Lipschitz property of the fractal-fractional polluted lake
system \eqref{model2}.

\begin{Lemma}
\label{lem51}
Consider ${L_1}, {L_2}, {L_3}, {L_1}^*, {L_2}^*,
{L_3}^* \in \mathbb{C} := C(\mathbb{J}, \mathbb{R})$, and let
\begin{itemize}
\item[(C1)]  $ \Vert {L_1} \Vert \leq \beta_1 $,
$\Vert {L_2} \Vert \leq \beta_2$,
$\Vert {L_3} \Vert \leq \beta_3$
for some constants $\beta_1,\beta_2,\beta_3 >0$.
\end{itemize}
Then, $\mathbb{Q}_1$, $\mathbb{Q}_2$, and $\mathbb{Q}_3$ defined in \eqref{eq41}
fulfill the Lipschitz property with constants $\alpha_1, \alpha_2, \alpha_3>0$
with respect to the relevant components, where
\begin{equation}
\label{eq51}
\alpha_1 = \frac{F_{31}+F_{21}}{V_1},
\quad \alpha_2 = \frac{F_{32}}{V_2},
\quad \alpha_3 =\frac{F_{13}}{V_3}.
\end{equation}
\end{Lemma}

\begin{proof}
For $\mathbb{Q}_1$, we take ${L_1}, {L_1}^* \in \mathbb{C}
:= C(\mathbb{J}, \mathbb{R})$ arbitrarily, and we have
\begin{equation}
\label{eq:above}
\begin{aligned}
\Vert \mathbb{Q}_1 &\big(s, {L_1}(s), {L_2}(s), {L_3}(s) \big)
- \mathbb{Q}_1 \big(s, {L_1}^*(s), {L_2}(s), {L_3}(s) \big) \Vert \\[0.3cm]
&= \left\Vert \left( -\frac{F_{31}}{V_1}{L_1}(s)-\frac{F_{21}}{V_1}{L_1}(s)\right)
\left( -\frac{F_{31}}{V_1}{L_1}^*(s)-\frac{F_{21}}{V_1}{L_1}^*(s)\right) \right\Vert \\[0.3cm]
&\leq \left[\frac{F_{31}+F_{21}}{V_1}  \right] \Vert {L_1}(s) - {L_1}^*(s) \Vert
= \alpha_1 \Vert {L_1}(s) - {L_1}^*(s) \Vert.
\end{aligned}
\end{equation}
From \eqref{eq:above}, we find out that $\mathbb{Q}_1$ is Lipschitz with respect to
${L_1}$ under the constant $\alpha_1 > 0$. For $\mathbb{Q}_2$, we choose
arbitrary ${L_2}, {L_2}^* \in \mathbb{C}
:= C(\mathbb{J}, \mathbb{R})$, and estimate
\begin{align*}
\Vert \mathbb{Q}_2 \big(s, {L_1}(s), {L_2}(s), {L_3}(s) \big)
&- \mathbb{Q}_2 \big(s, {L_1}(s), {L_2}^*(s), {L_3}(s) \big) \Vert \\[0.3cm]
&= \left\Vert \left( -\frac{F_{32}}{V_2}{L_2}(s) \right)
- \left( -\frac{F_{32}}{V_2} {L_2}^*(s)\right) \right\Vert \\[0.3cm]
&\leq \left[ \frac{F_{32}}{V_2} \right] \Vert {L_2}(s) - {L_2}^*(s) \Vert \\[0.3cm]
&= \alpha_2 \Vert {L_2}(s) - {L_2}^*(s) \Vert.
\end{align*}
This means that $\mathbb{Q}_2$ is Lipschitz with respect to
${L_2}$ under the constant $\alpha_2 > 0$. Finally, for arbitrary
elements ${L_3}, {L_3}^* \in \mathbb{C} := C(\mathbb{J}, \mathbb{R})$,
we have
\begin{align*}
\Vert \mathbb{Q}_3 \big(s, {L_1}(s), {L_2}(s), {L_3}(s) \big)
&- \mathbb{Q}_3 \big(s, {L_1}(s), {L_2}(s), {L_3}^*(s) \big) \Vert \\[0.3cm]
&= \left\Vert \left(- \frac{F_{13}}{V_3}{L_3}(s) \right)
- \left( -\frac{F_{13}}{V_3}{L_3}^*(s) \right) \right\Vert \\[0.3cm]
&\leq \left[ \frac{F_{13}}{V_3} \right] \Vert {L_3}(s) - {L_3}^*(s) \Vert \\[0.3cm]
&= \alpha_3 \Vert {L_3}(s) - {L_3}^*(s) \Vert.
\end{align*}
This shows that $\mathbb{Q}_3$ is Lipschitzian with respect to ${L_3}$
with $\alpha_3>0$. Therefore, the kernel functions $\mathbb{Q}_1$,
$\mathbb{Q}_2$, and $\mathbb{Q}_3$ are Lipschitz, respectively
with constants $\alpha_1, \alpha_2, \alpha_3 > 0$.
\end{proof}

By invoking Lemma~\ref{lem51}, we now prove
uniqueness of solution to the FF-system \eqref{model2}.

\begin{Theorem}
\label{thm52}
Let (C1) hold. If
\begin{equation}
\label{eq52}
\left[\frac{(1-\theta)\sigma S^{\sigma-1}}{\mathcal{AB}(\theta)}
+ \frac{\theta\sigma S^{\theta+\sigma - 1} \Gamma(\sigma) }{\mathcal{AB}(\theta)
\Gamma(\theta + \sigma)} \right] \alpha_j < 1,
\end{equation}
for $j\in \{ 1,2,3\}$ and where $\alpha_j > 0$
are the Lipschitz constants introduced by \eqref{eq51},
then the fractal-fractional polluted lake system
\eqref{model2} possesses exactly one solution.
\end{Theorem}

\begin{proof}
We do the proof by contradiction. Assume there exists another solution
to the fractal-fractional polluted lake system \eqref{model2}, namely
$\left( {L_1}^*(s), {L_2}^*(s), {L_3}^*(s) \right)$,
under initial conditions
$$
{L_1}^*(0)=L_{1,0},
\quad {L_2}^*(0) =L_{2,0},
\quad {L_3}^*(0)=L_{3,0}.
$$
From \eqref{eq477}, we have
\begin{align*}
{L_1}^*(s)
&= L_{1,0} + \frac{(1-\theta)\sigma s^{\sigma-1}}{\mathcal{AB}(\theta)}
\mathbb{Q}_1(s, {L_1}^*(s), {L_2}^*(s), {L_3}^*(s)) \\[0.3cm]
&\qquad +\frac{\theta\sigma}{\mathcal{AB}(\theta) \Gamma(\theta)}
\int_0^s \mathfrak{w}^{\sigma -1} (s- \mathfrak{w})^{\theta-1}
\mathbb{Q}_1(\mathfrak{w}, {L_1}^*(\mathfrak{w}),
{L_2}^*(\mathfrak{w}), {L_3}^*(\mathfrak{w}))
\, \mathrm{d}\mathfrak{w},
\end{align*}
\begin{align*}
{L_2}^*(s)
&= L_{2,0} +\frac{(1-\theta)\sigma s^{\sigma-1}}{\mathcal{AB}(\theta)}
\mathbb{Q}_2(s, {L_1}^*(s), {L_2}^*(s), {L_3}^*(s))\\[0.3cm]
&\qquad +\frac{\theta\sigma}{\mathcal{AB}(\theta) \Gamma(\theta)}
\int_0^s \mathfrak{w}^{\sigma -1} (s- \mathfrak{w})^{\theta-1}
\mathbb{Q}_2(\mathfrak{w}, {L_1}^*(\mathfrak{w}),
{L_2}^*(\mathfrak{w}), {L_3}^*(\mathfrak{w}))
\, \mathrm{d}\mathfrak{w},
\end{align*}
and
\begin{align*}
{L_3}^*(s)
&= L_{3,0} + \frac{(1-\theta)\sigma s^{\sigma-1}}{\mathcal{AB}(\theta)}
\mathbb{Q}_3(s, {L_1}^*(s), {L_2}^*(s), {L_3}^*(s))\\[0.3cm]
&\qquad +\frac{\theta\sigma}{\mathcal{AB}(\theta) \Gamma(\theta)}
\int_0^s \mathfrak{w}^{\sigma -1} (s- \mathfrak{w})^{\theta-1}
\mathbb{Q}_3(\mathfrak{w}, {L_1}^*(\mathfrak{w}),
{L_2}^*(\mathfrak{w}), {L_3}^*(\mathfrak{w}))
\, \mathrm{d}\mathfrak{w}.
\end{align*}
In this case, we estimate
\begin{align*}
\vert {L_1}(s) - {L_1}^*(s) \vert
&\leq \frac{(1-\theta)\sigma s^{\sigma-1}}{\mathcal{AB}(\theta)} \Big\vert
\mathbb{Q}_1(s, {L_1}(s), {L_2}(s), {L_3}(s))
- \mathbb{Q}_1(s, {L_1}^*(s), {L_2}^*(s),
{L_3}^*(s)) \Big\vert \\[0.3cm]
&\quad +\frac{\theta\sigma}{\mathcal{AB}(\theta) \Gamma(\theta)}
\int_0^s \mathfrak{w}^{\sigma -1} (s- \mathfrak{w})^{\theta-1} \\[0.3cm]
&\quad \times\Big\vert \mathbb{Q}_1(\mathfrak{w}, {L_1}(\mathfrak{w}),
{L_2}(\mathfrak{w}), {L_3}(\mathfrak{w}))
- \mathbb{Q}_1(\mathfrak{w}, {L_1}^*(\mathfrak{w}),
{L_2}^*(\mathfrak{w}), {L_3}^*(\mathfrak{w})) \Big\vert
\, \mathrm{d}\mathfrak{w} \\[0.3cm]
&\leq  \frac{(1-\theta)\sigma s^{\sigma-1}}{\mathcal{AB}(\theta)}
\alpha_1 \Vert {L_1}-{L_1}^*\Vert
+\frac{\theta\sigma}{\mathcal{AB}(\theta) \Gamma(\theta)}
\int_0^s \mathfrak{w}^{\sigma -1} (s- \mathfrak{w})^{\theta-1}
\alpha_1 \Vert {L_1}-{L_1}^*\Vert\, \mathrm{d}\mathfrak{w} \\[0.3cm]
&\leq \left[  \frac{(1-\theta)\sigma S^{\sigma-1}}{\mathcal{AB}(\theta)}
+ \frac{\theta\sigma S^{\theta+\sigma - 1} \Gamma(\sigma) }{\mathcal{AB}(\theta)
\Gamma(\theta + \sigma)} \right] \alpha_1 \Vert {L_1}-{L_1}^*\Vert,
\end{align*}
and so
$$
\left( 1 - \left[  \frac{(1-\theta)\sigma S^{\sigma-1}}{\mathcal{AB}(\theta)}
+ \frac{\theta\sigma S^{\theta+\sigma - 1} \Gamma(\sigma)}{\mathcal{AB}(\theta)
\Gamma(\theta + \sigma)} \right] \alpha_1 \right) \Vert {L_1}-{L_1}^*\Vert
\leq 0.
$$
From \eqref{eq52}, we can assert that the above inequality holds
if $\Vert {L_1} - {L_1}^* \Vert = 0$
or ${L_1} = {L_1}^*$.
Similarly, from
$$
\Vert {L_2} - {L_2}^* \Vert
\leq  \left[  \frac{(1-\theta)\sigma S^{\sigma-1}}{\mathcal{AB}(\theta)}
+ \frac{\theta\sigma S^{\theta+\sigma - 1} \Gamma(\sigma)}{\mathcal{AB}(\theta)
\Gamma(\theta + \sigma)} \right] \alpha_2 \Vert {L_2}-{L_2}^*\Vert,
$$
we obtain
$$
\left( 1 - \left[  \frac{(1-\theta)\sigma S^{\sigma-1}}{\mathcal{AB}(\theta)}
+  \frac{\theta\sigma S^{\theta+\sigma - 1} \Gamma(\sigma) }{\mathcal{AB}(\theta)
\Gamma(\theta + \sigma)} \right] \alpha_2 \right)
\Vert {L_2}-{L_2}^*\Vert \leq 0,
$$
which gives $\Vert {L_2} - {L_2}^* \Vert = 0$
or ${L_2} = {L_2}^*$. Furthermore,
$$
\Vert {L_3} - {L_3}^* \Vert
\leq  \left[  \frac{(1-\theta)\sigma S^{\sigma-1}}{\mathcal{AB}(\theta)}
+  \frac{\theta\sigma S^{\theta+\sigma - 1} \Gamma(\sigma) }{\mathcal{AB}(\theta)
\Gamma(\theta + \sigma)} \right] \alpha_3 \Vert {L_3}-{L_3}^*\Vert,
$$
which yields
$$
\left( 1 - \left[  \frac{(1-\theta)\sigma S^{\sigma-1}}{\mathcal{AB}(\theta)}
+ \frac{\theta\sigma S^{\theta+\sigma - 1} \Gamma(\sigma) }{\mathcal{AB}(\theta)
\Gamma(\theta + \sigma)} \right] \alpha_3 \right)
\Vert {L_3}-{L_3}^*\Vert \leq 0.
$$
Hence, ${L_3} = {L_3}^*$. As a consequence,
$$
\left( {L_1}(s), {L_2}(s), {L_3}(s) \right)
= \left( {L_1}^*(s), {L_2}^*(s), {L_3}^*(s) \right),
$$
which proves that the solution to the fractal-fractional
polluted lake system \eqref{model2} is unique.
\end{proof}


\section{Ulam--Hyers--Rassias stability}
\label{sec:5}

In this section, the stability of the solutions
to the polluted lake system of three lakes is studied.
Given the desire to establish robust mathematical 
foundations for the model, we consider four different 
notions of stability. More precisely, we prove stability 
for our fractal-fractional (FF)
polluted lake system \eqref{model2} with respect to
Ulam--Hyers and Ulam--Hyers--Rassias notions
and their respective generalizations.
Stability analysis is pivotal in ensuring mathematical models' 
reliability and predictability, especially in real-world applications 
such as the polluted lake system. Ulam stability, Hyers stability, 
and their generalizations offer valuable frameworks for understanding 
the behavior of solutions to dynamic systems under perturbations. Given 
the intricate nature of fractal-fractional systems, the use of these stability 
notions allows us to ascertain the system's resilience to variations and 
disturbances, providing insights into the long-term behavior and reliability 
of the proposed model. By choosing stability in this context, we aim to enhance 
the credibility of the model and its applicability in addressing 
the complexities inherent in polluted lake systems.

\begin{Definition}
\label{def41}
The FF polluted lake system \eqref{model2} is Ulam--Hyers stable
if there exists $a_{\mathbb{Q}_1}$, $a_{\mathbb{Q}_2}$,
$a_{\mathbb{Q}_3} \in \mathbb{R}^+$
such that for all $r_j>0$, $j = 1,2,3$, and
for all $\left( {L_1}^*, {L_2}^*, {L_3}^* \right)\in \mathbb{X}$
satisfying
\begin{align}
\label{eq61}
\begin{cases}
\Big\vert {}^{\mathbf{FFML}}\mathcal{D}_{0,s}^{(\theta,\sigma)}{L_1}^*(s)
- \mathbb{Q}_1 \left(s, {L_1}^*(s), {L_2}^*(s), {L_3}^*(s) \right)
\Big\vert < r_1, \\[0.3cm]
\Big\vert {}^{\mathbf{FFML}}\mathcal{D}_{0,s}^{(\theta,\sigma)}{L_2}^*(s)
- \mathbb{Q}_2 \left(s, {L_1}^*(s), {L_2}^*(s), {L_3}^*(s) \right)
\Big\vert < r_2, \\[0.3cm]
\Big\vert {}^{\mathbf{FFML}}\mathcal{D}_{0,s}^{(\theta,\sigma)}{L_3}^*(s)
- \mathbb{Q}_3 \left(s, {L_1}^*(s), {L_2}^*(s), {L_3}^*(s)\right)
\Big\vert < r_3,
\end{cases}
\end{align}
there exists $\left( {L_1}, {L_2}, {L_3}  \right)\in \mathbb{X}$
satisfying the fractal-fractional polluted lake system \eqref{model2} with
\begin{equation*}
\begin{cases}
\big\vert {L_1}^*(s) - {L_1}(s) \big\vert
\leq a_{\mathbb{Q}_1}r_1, \\[0.3cm]
\big\vert {L_2}^*(s) - {L_2}(s) \big\vert
\leq a_{\mathbb{Q}_2}r_2, \\[0.3cm]
\big\vert {L_3}^*(s) - {L_3}(s) \big\vert
\leq a_{\mathbb{Q}_3}r_3.
\end{cases}
\end{equation*}
\end{Definition}

\begin{Definition}
\label{def42}
The FF polluted lake system \eqref{model2} is generalized Ulam--Hyers stable
if $\exists\, a_{\mathbb{Q}_j} \in C(\mathbb{R}^+, \mathbb{R}^+)$,
$j\in \{ 1,2,3\}$, with $a_{\mathbb{Q}_j}(0)=0$ such that
$\forall\,r_j>0$ and $\forall\, \left( {L_1}^*, {L_2}^*, {L_3}^*  \right)
\in \mathbb{X} $ fulfilling \eqref{eq61}, there is
a solution $\left( {L_1}, {L_2}, {L_3} \right)\in \mathbb{X}$
of the given FF polluted lake system \eqref{model2} such that
\begin{equation*}
\begin{cases}
\big\vert {L_1}^*(s)
- {L_1}(s) \big\vert \leq a_{\mathbb{Q}_1}(r_1),\\[0.3cm]
\big\vert {L_2}^*(s)
- {L_2}(s) \big\vert \leq a_{\mathbb{Q}_2}(r_2),\\[0.3cm]
\big\vert {L_3}^*(s)
- {L_3}(s) \big\vert \leq a_{\mathbb{Q}_3}(r_3).
\end{cases}
\end{equation*}
\end{Definition}

\begin{Remark}
\label{Rem45}
The triplet $\left({L_1}^*, {L_2}^*, {L_3}^*  \right)\in \mathbb{X}$
is a solution for~\eqref{eq61} if, and only if,
$\exists\, z_1, z_2,z_3\in C([0,S],\mathbb{R})$
(each of them depend on $ {L_1}^*, {L_2}^*, {L_3}^* $, respectively)
such that $\forall\,s\in \mathbb{J}$,
\begin{enumerate}
\item[$(i)$] $\vert z_j(s)\vert < r_j,$
\item[$(ii)$] one has
\begin{align*}
\begin{cases}
{}^{\mathbf{FFML}}\mathcal{D}_{0,s}^{(\theta,\sigma)}{L_1}^*(s)
= \mathbb{Q}_1 \left(s, {L_1}^*(s), {L_2}^*(s),
{L_3}^*(s) \right) + z_1(s), \\[0.3cm]
{}^{\mathbf{FFML}}\mathcal{D}_{0,s}^{(\theta,\sigma)}{L_2}^*(s)
= \mathbb{Q}_2 \left(s, {L_1}^*(s), {L_2}^*(s),
{L_3}^*(s) \right) + z_2(s), \\[0.3cm]
{}^{\mathbf{FFML}}\mathcal{D}_{0,s}^{(\theta,\sigma)}{L_3}^*(s)
= \mathbb{Q}_3 \left(s, {L_1}^*(s),
{L_2}^*(s), {L_3}^*(s) \right) + z_3(s).
\end{cases}
\end{align*}
\end{enumerate}
\end{Remark}

\begin{Definition}
\label{def43}
The fractal-fractional polluted lake model \eqref{model2}
is Ulam--Hyers--Rassias stable with respect to
$\hbar_j$, $j\in \{ 1,2,3\}$, if
$\exists\, 0 < a_{(\mathbb{Q}_j,\hbar_j)} \in \mathbb{R}$
such that $\forall\,r_j>0$ and $\forall\,
\left( {L_1}^*, {L_2}^*, {L_3}^*
\right)\in \mathbb{X}$ fulfilling
\begin{align}
\label{eq62}
\begin{cases}
\Big\vert {}^{\mathbf{FFML}}\mathcal{D}_{0,s}^{(\theta,\sigma)}{L_1}^*(s)
- \mathbb{Q}_1 \left(s, {L_1}^*(s), {L_2}^*(s), {L_3}^*(s) \right)
\Big\vert < r_1\hbar_1(s), \\[0.3cm]
\Big\vert {}^{\mathbf{FFML}}\mathcal{D}_{0,s}^{(\theta,\sigma)}{L_2}^*(s)
- \mathbb{Q}_2 \left(s, {L_1}^*(s), {L_2}^*(s), {L_3}^*(s) \right)
\Big\vert < r_2\hbar_2(s), \\[0.3cm]
\Big\vert {}^{\mathbf{FFML}}\mathcal{D}_{0,s}^{(\theta,\sigma)}{L_3}^*(s)
- \mathbb{Q}_3 \left(s, {L_1}^*(s), {L_2}^*(s), {L_3}^*(s) \right)
\Big\vert < r_3\hbar_3(s),
\end{cases}
\end{align}
there exists a solution
$\left( {L_1}, {L_2}, {L_3}  \right) \in \mathbb{X}$
of the FF-model of polluted lake system \eqref{model2} such that
\begin{equation*}
\begin{cases}
\big\vert {L_1}^*(s) - {L_1}(s) \big\vert
\leq r_1a_{(\mathbb{Q}_1,\hbar_1)}\hbar_1(s),
\quad \forall\, s \in \mathbb{J}, \\[0.3cm]
\big\vert {L_2}^*(s) - {L_2}(s) \big\vert
\leq r_2a_{(\mathbb{Q}_2,\hbar_2)}\hbar_2(s),
\quad \forall\, s \in \mathbb{J}, \\[0.3cm]
\big\vert {L_3}^*(s) - {L_3}(s) \big\vert
\leq r_3a_{(\mathbb{Q}_3,\hbar_3)}\hbar_3(s),
\quad \forall\, s \in \mathbb{J},
\end{cases}
\end{equation*}
with $\hbar_1, \hbar_2, \hbar_3 \in C([0,S],\mathbb{R}^+)$.
\end{Definition}

\begin{Remark}
If $\hbar_j(s)=1$, then Definition~\ref{def43}
reduces to the Ulam--Hyers criterion.
\end{Remark}

\begin{Definition}
\label{def44}
The FF polluted lake system \eqref{model2} is
generalized  Ulam--Hyers--Rasias stable with respect to
$\hbar_j $ if exists $a_{(\mathbb{Q}_j,\hbar_j)}
\in C(\mathbb{R}^+, \mathbb{R}^+)$ such that for all $\big( {L_1}^*,
{L_2}^*, {L_3}^*  \big)\in \mathbb{X} $ satisfying
\begin{align*}
\begin{cases}
\Big\vert {}^{\mathbf{FFML}}\mathcal{D}_{0,s}^{(\theta,\sigma)}{L_1}^*(s)
- \mathbb{Q}_1 \left(s, {L_1}^*(s), {L_2}^*(s), {L_3}^*(s) \right)
\Big\vert < \hbar_1(s), \\[0.3cm]
\Big\vert {}^{\mathbf{FFML}}\mathcal{D}_{0,s}^{(\theta,\sigma)}{L_2}^*(s)
- \mathbb{Q}_2 \left(s, {L_1}^*(s), {L_2}^*(s), {L_3}^*(s) \right)
\Big\vert < \hbar_2(s), \\[0.3cm]
\Big\vert {}^{\mathbf{FFML}}\mathcal{D}_{0,s}^{(\theta,\sigma)}{L_3}^*(s)
- \mathbb{Q}_3 \left(s, {L_1}^*(s), {L_2}^*(s), {L_3}^*(s) \right)
\Big\vert < \hbar_3(s),
\end{cases}
\end{align*}
there exists a solution
$\left( {L_1}, {L_2}, {L_3}  \right) \in \mathbb{X}$
of the FF-model of polluted lake system \eqref{model2} such that
\begin{equation*}
\begin{cases}
\big\vert {L_1}^*(s) - {L_1}(s) \big\vert
\leq a_{(\mathbb{Q}_1,\hbar_1)}(r_1)\hbar_1(s), \\[0.3cm]
\big\vert {L_2}^*(s) - {L_2}(s) \big\vert
\leq a_{(\mathbb{Q}_2,\hbar_2)}(r_2)\hbar_2(s), \\[0.3cm]
\big\vert {L_3}^*(s) - {L_3}(s) \big\vert
\leq a_{(\mathbb{Q}_3,\hbar_3)}(r_3)\hbar_3(s).
\end{cases}
\end{equation*}
\end{Definition}

\begin{Remark}
\label{Rem46}
Note that $\left( {L_1}^*, {L_2}^*, {L_3}^*  \right)
\in \mathbb{X} $ is a solution for~\eqref{eq62} if, and only if,
$\exists\, z_1,z_2,z_3\in C([0,S],\mathbb{R})$ (each of them depend
on ${L_1}^*, {L_2}^*, {L_3}^*$, respectively)
such that $\forall\,s\in \mathbb{J}$,
\begin{enumerate}
\item[$(i)$] $ \vert z_j(s) \vert < r_j \hbar_j(s)$,
\item[$(ii)$] we have
\begin{align*}
\begin{cases}
{}^{\mathbf{FFML}}\mathcal{D}_{0,s}^{(\theta,\sigma)}{L_1}^*(s)
= \mathbb{Q}_1 \left(s, {L_1}^*(s), {L_2}^*(s), {L_3}^*(s) \right)
+ z_1(s),\\[0.3cm]
{}^{\mathbf{FFML}}\mathcal{D}_{0,s}^{(\theta,\sigma)}{L_2}^*(s)
= \mathbb{Q}_2 \left(s, {L_1}^*(s), {L_2}^*(s), {L_3}^*(s) \right)
+ z_2(s), \\[0.3cm]
{}^{\mathbf{FFML}}\mathcal{D}_{0,s}^{(\theta,\sigma)}{L_3}^*(s)
= \mathbb{Q}_3 \left(s, {L_1}^*(s), {L_2}^*(s), {L_3}^*(s) \right)
+ z_3(s).
\end{cases}
\end{align*}
\end{enumerate}
\end{Remark}

Lemmas~\ref{lem57} and \ref{lem58} are useful
to prove Theorems~\ref{thm47} and \ref{thm48},
respectively.
	
\begin{Lemma}
\label{lem57}
For each $r_1,r_2,r_3>0$, suppose that
$({L_1}^*, {L_2}^*, {L_3}^*)\in \mathbb{X}$
is a solution of \eqref{eq61}. Then, functions
${L_1}^*, {L_2}^*, {L_3}^* \in \mathbb{C}$
fulfill the following three inequalities:
\begin{align}
\label{eq53}
\Bigg\vert {L_1}^*(s)
&- \Big( L_{1,0} +\frac{(1-\theta)\sigma s^{\sigma-1}}{\mathcal{AB}(\theta)}
\mathbb{Q}_1\left(s, {L_1}^*(s), {L_2}^*(s), {L_3}^*(s)\right)
+ \frac{\theta\sigma}{\mathcal{AB}(\theta) \Gamma(\theta)}
\int_0^s \mathfrak{w}^{\sigma -1} (s- \mathfrak{w})^{\theta-1}  \nonumber\\[0.3cm]
&\times  \mathbb{Q}_1(\mathfrak{w}, {L_1}^*(\mathfrak{w}),
{L_2}^*(\mathfrak{w}), {L_3}^*(\mathfrak{w})) \, \mathrm{d}\mathfrak{w} \Big)
\Bigg\vert \leq \left[  \frac{(1-\theta)\sigma S^{\sigma-1}}{\mathcal{AB}(\theta)}
+  \frac{\theta\sigma S^{\theta+\sigma - 1} \Gamma(\sigma) }{\mathcal{AB}(\theta)
\Gamma(\theta + \sigma)} \right]r_1,
\end{align}
\begin{align}
\label{eq54}
\Bigg\vert {L_2}^*(s) &- \Big( L_{2,0}
+\frac{(1-\theta)\sigma s^{\sigma-1}}{\mathcal{AB}(\theta)}
\mathbb{Q}_2(s, {L_1}^*(s), {L_2}^*(s), {L_3}^*(s))
+ \frac{\theta\sigma}{\mathcal{AB}(\theta) \Gamma(\theta)}
\int_0^s \mathfrak{w}^{\sigma -1} (s- \mathfrak{w})^{\theta-1}  \nonumber\\[0.3cm]
&\times  \mathbb{Q}_2(\mathfrak{w}, {L_1}^*(\mathfrak{w}),
{L_2}^*(\mathfrak{w}), {L_3}^*(\mathfrak{w})) \, \mathrm{d}\mathfrak{w} \Big)
\Bigg\vert \leq \left[  \frac{(1-\theta)\sigma S^{\sigma-1}}{\mathcal{AB}(\theta)}
+ \frac{\theta\sigma S^{\theta+\sigma - 1} \Gamma(\sigma) }{\mathcal{AB}(\theta)
\Gamma(\theta + \sigma)} \right]r_2,
\end{align}
and
\begin{align}
\label{eq55}
\Bigg\vert {L_3}^*(s) &- \Big( L_{3,0}
+\frac{(1-\theta)\sigma s^{\sigma-1}}{\mathcal{AB}(\theta)}
\mathbb{Q}_3(s, {L_1}^*(s), {L_2}^*(s), {L_3}^*(s))
+ \frac{\theta\sigma}{\mathcal{AB}(\theta) \Gamma(\theta)}
\int_0^s \mathfrak{w}^{\sigma -1} (s- \mathfrak{w})^{\theta-1}  \nonumber\\[0.3cm]
&\times  \mathbb{Q}_3(\mathfrak{w}, {L_1}^*(\mathfrak{w}),
{L_2}^*(\mathfrak{w}), {L_3}^*(\mathfrak{w})) \, \mathrm{d}\mathfrak{w} \Big)
\Bigg\vert \leq \left[  \frac{(1-\theta)\sigma S^{\sigma-1}}{\mathcal{AB}(\theta)}
+ \frac{\theta\sigma S^{\theta+\sigma - 1} \Gamma(\sigma) }{\mathcal{AB}(\theta)
\Gamma(\theta + \sigma)} \right]r_3.
\end{align}
\end{Lemma}

\begin{proof}
Let $ r_1>0$ be arbitrary. Since ${L_1}^*\in \mathbb{C} $ satisfies
$$
\Big\vert {}^{\mathbf{FFML}}\mathcal{D}_{0,s}^{(\theta,\sigma)}{L_1}^*(s)
- \mathbb{Q}_1 \left(s, {L_1}^*(s), {L_2}^*(s), {L_3}^*(s) \right)
\Big\vert < r_1,
$$
it follows from Remark~\ref{Rem45} that one can take a function $ z_1(s) $ such that
$$
{}^{\mathbf{FFML}}\mathcal{D}_{0,s}^{(\theta,\sigma)}{L_1}^*(s)
= \mathbb{Q}_1 \left(s, {L_1}^*(s), {L_2}^*(s), {L_3}^*(s) \right)
+ z_1(s)
$$
and $ \vert z_1(s) \vert\leq r_1 $. Clearly,
\begin{align*}
{L_1}^*(s) &= L_{1,0}
+ \frac{(1-\theta)\sigma s^{\sigma-1}}{\mathcal{AB}(\theta)}
\left[ \mathbb{Q}_1(s, {L_1}^*(s), {L_2}^*(s),
{L_3}^*(s)) + z_1(s) \right] \\[0.3cm]
&\quad +\frac{\theta\sigma}{\mathcal{AB}(\theta) \Gamma(\theta)}
\int_0^s \mathfrak{w}^{\sigma -1} (s- \mathfrak{w})^{\theta-1}
\Big[ \mathbb{Q}_1(\mathfrak{w}, {L_1}^*(\mathfrak{w}),
{L_2}^*(\mathfrak{w}), {L_3}^*(\mathfrak{w}))
+ z_1(\mathfrak{w}) \Big]\, \mathrm{d}\mathfrak{w}.
\end{align*}
In this case, we estimate
\begin{align*}
\Bigg\vert {L_1}^*(s) &- \Big( L_{1,0}
+ \frac{(1-\theta)\sigma s^{\sigma-1}}{\mathcal{AB}(\theta)}
\mathbb{Q}_1(s, {L_1}^*(s), {L_2}^*(s), {L_3}^*(s))\\[0.3cm]
&\qquad + \frac{\theta\sigma}{\mathcal{AB}(\theta) \Gamma(\theta)}
\int_0^s \mathfrak{w}^{\sigma -1} (s- \mathfrak{w})^{\theta-1}
\mathbb{Q}_1(\mathfrak{w}, {L_1}^*(\mathfrak{w}),
{L_2}^*(\mathfrak{w}), {L_3}^*(\mathfrak{w}))
\, \mathrm{d}\mathfrak{w} \Big) \Bigg\vert\\[0.3cm]
&\leq \frac{(1-\theta)\sigma s^{\sigma-1}}{\mathcal{AB}(\theta)}
\vert z_1(s)\vert  + \frac{\theta\sigma}{\mathcal{AB}(\theta)
\Gamma(\theta)} \int_0^s \mathfrak{w}^{\sigma -1}
(s- \mathfrak{w})^{\theta-1} \vert z_1(\mathfrak{w})\vert
\, \mathrm{d}\mathfrak{w}\\[0.3cm]
&\leq \frac{(1-\theta)\sigma S^{\sigma-1}}{\mathcal{AB}(\theta)}r_1
+ \frac{\theta\sigma S^{\theta+\sigma - 1} \Gamma(\sigma) }{
\mathcal{AB}(\theta) \Gamma(\theta + \sigma)} r_1\\[0.3cm]
&= \left[  \frac{(1-\theta)\sigma S^{\sigma-1}}{\mathcal{AB}(\theta)}
+ \frac{\theta\sigma S^{\theta+\sigma - 1}
\Gamma(\sigma) }{\mathcal{AB}(\theta) \Gamma(\theta + \sigma)} \right]r_1.
\end{align*}
This means that \eqref{eq53} is fulfilled.
We prove \eqref{eq54} and \eqref{eq55}
in a similar way.
\end{proof}

To prove our next result (see Lemma~\ref{lem58}),
we consider the following condition:
\begin{enumerate}
\item[(C2)] there exists increasing mappings
$\hbar_j\in C([0,S],\mathbb{R}^+)$, $j\in \{ 1,2,3\}$,
and $\Delta_{\hbar_j}>0$, provided that
\begin{align}\label{lambda2}
{}^{\mathbf{FFML}}\mathcal{I}_{0,s}^{(\theta,\sigma)} \hbar_j(s)
< \Delta_{\hbar_j} \hbar_j(s),
~(j\in \{ 1,2,3\}), \forall\,s\in \mathbb{J}.
\end{align}
\end{enumerate}

\begin{Lemma}
\label{lem58}
Let $(C2)$ hold. For each $r_1,r_2,r_3>0$, suppose that
$\left({L_1}^*, {L_2}^*, {L_3}^*\right)\in \mathbb{X}$
is a solution of \eqref{eq62}. Then, functions
${L_1}^*, {L_2}^*, {L_3}^* \in \mathbb{C}$
fulfill the following three inequalities:
\begin{align}
\label{eq57}
&\Big\vert {L_1}^*(s) - \left( L_{1,0}
+{}^{\mathbf{FFML}}\mathcal{I}_{0,s}^{(\theta,\sigma)}
\mathbb{Q}_1(s, {L_1}^*(s), {L_2}^*(s), {L_3}^*(s)) \right)
\Big\vert \leq r_1 \Delta_{\hbar_1} \hbar_1 (s), \\[0.3cm]
&\Big\vert {L_2}^*(s) - \left( L_{2,0}
+ {}^{\mathbf{FFML}}\mathcal{I}_{0,s}^{(\theta,\sigma)}
\mathbb{Q}_2(s, {L_1}^*(s), {L_2}^*(s), {L_3}^*(s)) \right)
\Big\vert \leq r_2 \Delta_{\hbar_2} \hbar_2 (s), \\[0.3cm]
&\Big\vert {L_3}^*(s) - \left(L_{3,0}
+{}^{\mathbf{FFML}}\mathcal{I}_{0,s}^{(\theta,\sigma)}
\mathbb{Q}_3(s, {L_1}^*(s), {L_2}^*(s), {L_3}^*(s))\right)
\Big\vert \leq r_3 \Delta_{\hbar_3} \hbar_3 (s).
\end{align}
\end{Lemma}

\begin{proof}
Let $ r_1>0$. Since ${L_1}^*\in \mathbb{C} $ satisfies
$$
\Big\vert {}^{\mathbf{FFML}}\mathcal{D}_{0,s}^{(\theta,\sigma)}{L_1}^*(s)
- \mathbb{Q}_1 \left(s, {L_1}^*(s), {L_2}^*(s), {L_3}^*(s) \right)
\Big\vert < r_1\hbar_1(s),
$$
it follows from Remark~\ref{Rem46} that we can take $z_1(s)$ such that
$$
{}^{\mathbf{FFML}}\mathcal{D}_{0,s}^{(\theta,\sigma)}{L_1}^*(s)
= \mathbb{Q}_1 \big(s, {L_1}^*(s), {L_2}^*(s), {L_3}^*(s) \big) + z_1(s)
$$
and $ \vert z_1(s) \vert\leq r_1\hbar_1(s) $. Evidently,
\begin{align*}
{L_1}^*(s)
&= L_{1,0}
+\frac{(1-\theta)\sigma s^{\sigma-1}}{\mathcal{AB}(\theta)}
\left[ \mathbb{Q}_1(s, {L_1}^*(s), {L_2}^*(s),
{L_3}^*(s)) + z_1(s) \right] \\[0.3cm]
&\quad +\frac{\theta\sigma}{\mathcal{AB}(\theta) \Gamma(\theta)}
\int_0^s \mathfrak{w}^{\sigma -1} (s- \mathfrak{w})^{\theta-1}
\left[ \mathbb{Q}_1(\mathfrak{w}, {L_1}^*(\mathfrak{w}),
{L_2}^*(\mathfrak{w}), {L_3}^*(\mathfrak{w}))
+ z_1(\mathfrak{w}) \right]\, \mathrm{d}\mathfrak{w}.
\end{align*}
Then, we estimate
\begin{align*}
\Big\vert {L_1}^*(s) - &\left( L_{1,0} +
{}^{\mathbf{FFML}}\mathcal{I}_{0,s}^{(\theta,\sigma)}
\mathbb{Q}_1(s, {L_1}^*(s), {L_2}^*(s),
{L_3}^*(s)) \right) \Big\vert\\[0.3cm]
&= \Bigg\vert {L_1}^*(s) - \Big( L_{1,0}
+\frac{(1-\theta)\sigma s^{\sigma-1}}{\mathcal{AB}(\theta)}
\mathbb{Q}_1(s, {L_1}^*(s), {L_2}^*(s), {L_3}^*(s))\\[0.3cm]
&\qquad + \frac{\theta\sigma}{\mathcal{AB}(\theta) \Gamma(\theta)}
\int_0^s \mathfrak{w}^{\sigma -1} (s- \mathfrak{w})^{\theta-1}
\mathbb{Q}_1(\mathfrak{w}, {L_1}^*(\mathfrak{w}),
{L_2}^*(\mathfrak{w}), {L_3}^*(\mathfrak{w}))
\, \mathrm{d}\mathfrak{w} \Big) \Bigg\vert \\[0.3cm]
&\leq \frac{(1-\theta)\sigma s^{\sigma-1}}{\mathcal{AB}(\theta)}
\vert z_1(s)\vert  + \frac{\theta\sigma}{\mathcal{AB}(\theta)
\Gamma(\theta)} \int_0^s \mathfrak{w}^{\sigma -1} (s- \mathfrak{w})^{\theta-1}
\vert z_1(\mathfrak{w})\vert \, \mathrm{d}\mathfrak{w}\\[0.3cm]
&= {}^{\mathbf{FFML}}\mathcal{I}_{0,s}^{(\theta,\sigma)} \vert z_1(s)\vert\\[0.3cm]
&\leq {}^{\mathbf{FFML}}\mathcal{I}_{0,s}^{(\theta,\sigma)} r_1\hbar_1(s)\\[0.3cm]
&\leq r_1 \Delta_{\hbar_1} \hbar_1(s).
\end{align*}
We prove the remaining inequalities in a similar way.
\end{proof}

We are now in a position to investigate the Ulam--Hyers stability
for the FF-model of polluted lake system \eqref{model2}.

\begin{Theorem}
\label{thm47}
Assume (C1) holds. Then our polluted lake system
\eqref{model2} is both Ulam--Hyers and generalized Ulam--Hyers stable with
\begin{equation*}
\left[ \frac{(1-\theta)\sigma S^{\sigma-1}}{\mathcal{AB}(\theta)}
+  \frac{\theta\sigma S^{\theta+\sigma - 1} \Gamma(\sigma) }{\mathcal{AB}(\theta)
\Gamma(\theta + \sigma)} \right] \alpha_j < 1,
\quad j\in \{ 1,2,3\},
\end{equation*}
in which $\alpha_1, \alpha_2, \alpha_3 > 0$ are given by \eqref{eq51}.
\end{Theorem}

\begin{proof}
Let $ r_1>0$ and ${L_1}^*\in \mathbb{C} $ be an arbitrary solution of \eqref{eq61}.
By Theorem~\ref{thm52}, let $ {L_1} \in \mathbb{C} $ be the unique solution
of the FF polluted lake system \eqref{model2}. Then $ {L_1}(s) $ is defined as
\begin{align*}
{L_1}(s)
&= L_{1,0} + \frac{(1-\theta)\sigma s^{\sigma-1}}{\mathcal{AB}(\theta)}
\mathbb{Q}_1(s, {L_1}(s), {L_2}(s), {L_3}(s))\\[0.3cm]
&\quad +\frac{\theta\sigma}{\mathcal{AB}(\theta) \Gamma(\theta)}
\int_0^s \mathfrak{w}^{\sigma -1} (s- \mathfrak{w})^{\theta-1}
\mathbb{Q}_1(\mathfrak{w}, {L_1}(\mathfrak{w}), {L_2}(\mathfrak{w}),
{L_3}(\mathfrak{w}))\, \mathrm{d}\mathfrak{w}.
\end{align*}
From the triangle inequality, Lemma~\ref{lem57} gives
\begin{align*}
\big\vert &{L_1}^*(s)
- {L_1}(s) \big\vert \leq \Big\vert {L_1}^*(s)-L_{1,0}
-\frac{(1-\theta)\sigma s^{\sigma-1}}{\mathcal{AB}(\theta)}
\mathbb{Q}_1(s, {L_1}(s), {L_2}(s), {L_3}(s))\\[0.3cm]
&\quad -\frac{\theta\sigma}{\mathcal{AB}(\theta) \Gamma(\theta)}
\int_0^s \mathfrak{w}^{\sigma -1} (s- \mathfrak{w})^{\theta-1}
\mathbb{Q}_1(\mathfrak{w}, {L_1}(\mathfrak{w}), {L_2}(\mathfrak{w}),
{L_3}(\mathfrak{w}))\, \mathrm{d}\mathfrak{w} \Big\vert\\[0.3cm]
&\leq \Bigg\vert {L_1}^*(s) - \Big( L_{1,0}
+\frac{(1-\theta)\sigma s^{\sigma-1}}{\mathcal{AB}(\theta)}
\mathbb{Q}_1(s, {L_1}^*(s), {L_2}^*(s), {L_3}^*(s))\\[0.3cm]
&\quad + \frac{\theta\sigma}{\mathcal{AB}(\theta) \Gamma(\theta)}
\int_0^s \mathfrak{w}^{\sigma -1} (s- \mathfrak{w})^{\theta-1}
\mathbb{Q}_1(\mathfrak{w}, {L_1}^*(\mathfrak{w}),
{L_2}^*(\mathfrak{w}), {L_3}^*(\mathfrak{w}))
\, \mathrm{d}\mathfrak{w} \Big) \Bigg\vert\\[0.3cm]
&\quad + \frac{(1-\theta)\sigma s^{\sigma-1}}{\mathcal{AB}(\theta)}
\big\vert \mathbb{Q}_1(s, {L_1}^*(s), {L_2}^*(s),
{L_3}^*(s)) - \mathbb{Q}_1(s, {L_1}(s), {L_2}(s),
{L_3}(s)) \big\vert\\[0.3cm]
&\quad + \frac{\theta\sigma}{\mathcal{AB}(\theta) \Gamma(\theta)}
\int_0^s \mathfrak{w}^{\sigma -1} (s- \mathfrak{w})^{\theta-1}
\big\vert \mathbb{Q}_1(\mathfrak{w}, {L_1}^*(\mathfrak{w}),
{L_2}^*(\mathfrak{w}), {L_3}^*(\mathfrak{w}))\\
&\quad -  \mathbb{Q}_1(\mathfrak{w}, {L_1}(\mathfrak{w}),
{L_2}(\mathfrak{w}), {L_3}(\mathfrak{w}))
\big\vert \,\mathrm{d}\mathfrak{w}\\[0.3cm]
&\leq \left[  \frac{(1-\theta)\sigma S^{\sigma-1}}{\mathcal{AB}(\theta)}
+ \frac{\theta\sigma S^{\theta+\sigma - 1} \Gamma(\sigma) }{
\mathcal{AB}(\theta) \Gamma(\theta + \sigma)} \right]r_1
+ \frac{(1-\theta)\sigma S^{\sigma-1}}{\mathcal{AB}(\theta)}
\alpha_1 \Vert {L_1}^* - {L_1} \Vert \\[0.3cm]
&\quad + \frac{\theta\sigma S^{\theta+\sigma - 1} \Gamma(\sigma) }{
\mathcal{AB}(\theta) \Gamma(\theta + \sigma)} \alpha_1 \Vert {L_1}^*
- {L_1} \Vert \\[0.3cm]
&\leq \left[  \frac{(1-\theta)\sigma S^{\sigma-1}}{\mathcal{AB}(\theta)}
+  \frac{\theta\sigma S^{\theta+\sigma - 1} \Gamma(\sigma) }{
\mathcal{AB}(\theta) \Gamma(\theta + \sigma)} \right]  r_1
+ \left[  \frac{(1-\theta)\sigma S^{\sigma-1}}{\mathcal{AB}(\theta)}
+  \frac{\theta\sigma S^{\theta+\sigma - 1} \Gamma(\sigma) }{\mathcal{AB}(\theta)
\Gamma(\theta + \sigma)} \right] \alpha_1 \Vert {L_1}-{L_1}^*\Vert.
\end{align*}
Hence,
\begin{equation*}
\Vert {L_1}^* - {L_1} \Vert
\leq \frac{\left[  \dfrac{(1-\theta)\sigma S^{\sigma-1}}{\mathcal{AB}(\theta)}
+  \dfrac{\theta\sigma S^{\theta+\sigma - 1} \Gamma(\sigma) }{\mathcal{AB}(\theta)
\Gamma(\theta + \sigma)} \right]  r_1 }{1
- \left[  \dfrac{(1-\theta)\sigma S^{\sigma-1}}{\mathcal{AB}(\theta)}
+  \dfrac{\theta\sigma S^{\theta+\sigma - 1} \Gamma(\sigma) }{\mathcal{AB}(\theta)
\Gamma(\theta + \sigma)} \right] \alpha_1 }.
\end{equation*}
Set $a_{\mathbb{Q}_1}= \dfrac{\left[
\dfrac{(1-\theta)\sigma S^{\sigma-1}}{\mathcal{AB}(\theta)}
+  \dfrac{\theta\sigma S^{\theta+\sigma - 1} \Gamma(\sigma) }{
\mathcal{AB}(\theta) \Gamma(\theta + \sigma)} \right] }{1
- \left[  \dfrac{(1-\theta)\sigma S^{\sigma-1}}{\mathcal{AB}(\theta)}
+  \dfrac{\theta\sigma S^{\theta+\sigma - 1} \Gamma(\sigma) }{\mathcal{AB}(\theta)
\Gamma(\theta + \sigma)} \right] \alpha_1 }$. In this case,
$\Vert {L_1}^* - {L_1} \Vert  \leq  a_{\mathbb{Q}_1}r_1$.
Similarly, we obtain
$$
\Vert {L_2}^* - {L_2} \Vert
\leq  a_{\mathbb{Q}_2}r_2,
$$
$$
\Vert {L_3}^* - {L_3} \Vert
\leq  a_{\mathbb{Q}_3}r_3,
$$
where
$$
a_{\mathbb{Q}_j}= \dfrac{\left[\dfrac{(1-\theta)
\sigma S^{\sigma-1}}{\mathcal{AB}(\theta)}
+  \dfrac{\theta\sigma S^{\theta+\sigma - 1} \Gamma(\sigma) }{\mathcal{AB}(\theta)
\Gamma(\theta + \sigma)} \right]}{1- \left[  \dfrac{(1-\theta)\sigma
S^{\sigma-1}}{\mathcal{AB}(\theta)}  +  \dfrac{\theta\sigma
S^{\theta+\sigma - 1} \Gamma(\sigma) }{\mathcal{AB}(\theta)
\Gamma(\theta + \sigma)} \right] \alpha_j },
\quad j\in \{ 2,3\}.
$$
We conclude that the FF-model of polluted lake system \eqref{model2}
is Ulam--Hyers stable. On the other hand, if we take
$$
a_{\mathbb{Q}_j}(r_j)
= \dfrac{\left[  \dfrac{(1-\theta)\sigma S^{\sigma-1}}{\mathcal{AB}(\theta)}
+ \dfrac{\theta\sigma S^{\theta+\sigma - 1} \Gamma(\sigma) }{\mathcal{AB}(\theta)
\Gamma(\theta + \sigma)} \right]  r_j }{1- \left[
\dfrac{(1-\theta)\sigma S^{\sigma-1}}{\mathcal{AB}(\theta)}
+ \dfrac{\theta\sigma S^{\theta+\sigma - 1} \Gamma(\sigma) }{\mathcal{AB}(\theta)
\Gamma(\theta + \sigma)} \right] \alpha_j },
\quad j\in \{ 1,2,3\},
$$
then $a_{\mathbb{Q}_j}(0)=0$ and the proof is finished:
\eqref{model2} is generalized Ulam--Hyers stable.
\end{proof}

Theorem~\ref{thm48} establishes Ulam--Hyers-Rassias stability
for the fractal-fractional polluted lake system \eqref{model2}.

\begin{Theorem}
\label{thm48}
If (C1) and (C2) hold, then the FF-model of polluted lake system \eqref{model2} is
simultaneously stable in the sense of Definitions~\ref{def41} and \ref{def42}.
\end{Theorem}

\begin{proof}
Let $r_1>0$, and ${L_1}^* \in \mathbb{C} $ satisfy \eqref{eq62}.
By Theorem~\ref{thm52}, let $ {L_1} \in \mathbb{C} $
be the (unique) solution of the FF polluted lake system model \eqref{model2}.
Then $ {L_1}(s) $ becomes
\begin{align*}
{L_1}(s)
&= L_{1,0} + \frac{(1-\theta)\sigma s^{\sigma-1}}{\mathcal{AB}(\theta)}
\mathbb{Q}_1(s, {L_1}(s), {L_2}(s), {L_3}(s)) \\[0.3cm]
&\quad +\frac{\theta\sigma}{\mathcal{AB}(\theta) \Gamma(\theta)}
\int_0^s \mathfrak{w}^{\sigma -1} (s- \mathfrak{w})^{\theta-1}
\mathbb{Q}_1(\mathfrak{w}, {L_1}(\mathfrak{w}),
{L_2}(\mathfrak{w}), {L_3}(\mathfrak{w}))
\, \mathrm{d}\mathfrak{w}.
\end{align*}
With the aid of the triangle inequality, Lemma~\ref{lem58} gives
\begin{align*}
\big\vert  {L_1}^*(s)
&- {L_1}(s) \big\vert \leq \Big\vert {L_1}^*(s)
- L_{1,0} - \frac{(1-\theta)\sigma s^{\sigma-1}}{\mathcal{AB}(\theta)}
\mathbb{Q}_1(s, {L_1}(s), {L_2}(s), {L_3}(s)) \\[0.3cm]
&\quad -\frac{\theta\sigma}{\mathcal{AB}(\theta) \Gamma(\theta)}
\int_0^s \mathfrak{w}^{\sigma -1} (s- \mathfrak{w})^{\theta-1}
\mathbb{Q}_1(\mathfrak{w}, {L_1}(\mathfrak{w}),
{L_2}(\mathfrak{w}), {L_3}(\mathfrak{w}))\,
\mathrm{d}\mathfrak{w} \Big\vert \\[0.3cm]
&\leq \Bigg\vert {L_1}^*(s) - \Big( L_{1,0}
+ \frac{(1-\theta)\sigma s^{\sigma-1}}{\mathcal{AB}(\theta)}
\mathbb{Q}_1(s, {L_1}^*(s), {L_2}^*(s), {L_3}^*(s))\\[0.3cm]
&\quad + \frac{\theta\sigma}{\mathcal{AB}(\theta) \Gamma(\theta)}
\int_0^s \mathfrak{w}^{\sigma -1} (s- \mathfrak{w})^{\theta-1}
\mathbb{Q}_1(\mathfrak{w}, {L_1}^*(\mathfrak{w}),
{L_2}^*(\mathfrak{w}), {L_3}^*(\mathfrak{w}))
\, \mathrm{d}\mathfrak{w} \Big) \Bigg\vert\\[0.3cm]
&\quad + \frac{(1-\theta)\sigma s^{\sigma-1}}{\mathcal{AB}(\theta)}
\big\vert \mathbb{Q}_1(s, {L_1}^*(s), {L_2}^*(s),
{L_3}^*(s)) - \mathbb{Q}_1(s, {L_1}(s),
{L_2}(s), {L_3}(s)) \big\vert\\[0.3cm]
&\quad + \frac{\theta\sigma}{\mathcal{AB}(\theta) \Gamma(\theta)}
\int_0^s \mathfrak{w}^{\sigma -1} (s- \mathfrak{w})^{\theta-1}
\big\vert \mathbb{Q}_1(\mathfrak{w}, {L_1}^*(\mathfrak{w}),
{L_2}^*(\mathfrak{w}), {L_3}^*(\mathfrak{w}))\\
&\quad - \mathbb{Q}_1(\mathfrak{w}, {L_1}(\mathfrak{w}),
{L_2}(\mathfrak{w}), {L_3}(\mathfrak{w}))
\big\vert \,\mathrm{d}\mathfrak{w}\\[0.3cm]
&\leq \Big\vert {L_1}^*(s) - \Big( L_{1,0}
+ {}^{\mathbf{FFML}}\mathcal{I}_{0,s}^{(\theta,\sigma)}
\mathbb{Q}_1(s, {L_1}^*(s), {L_2}^*(s),
{L_3}^*(s)) \Big) \Big\vert\\[0.3cm]
&\quad + \frac{(1-\theta)\sigma s^{\sigma-1}}{\mathcal{AB}(\theta)}
\big\vert \mathbb{Q}_1(s, {L_1}^*(s), {L_2}^*(s),
{L_3}^*(s)) - \mathbb{Q}_1(s, {L_1}(s),
{L_2}(s), {L_3}(s)) \big\vert\\[0.3cm]
&\quad + \frac{\theta\sigma}{\mathcal{AB}(\theta) \Gamma(\theta)}
\int_0^s \mathfrak{w}^{\sigma -1} (s- \mathfrak{w})^{\theta-1}
\big\vert \mathbb{Q}_1(\mathfrak{w}, {L_1}^*(\mathfrak{w}),
{L_2}^*(\mathfrak{w}), {L_3}^*(\mathfrak{w}))\\
&\quad - \mathbb{Q}_1(\mathfrak{w}, {L_1}(\mathfrak{w}),
{L_2}(\mathfrak{w}), {L_3}(\mathfrak{w}))
\big\vert \,\mathrm{d}\mathfrak{w}\\[0.3cm]
&\leq r_1\Delta_{\hbar_1}\hbar_1(s) + \frac{(1-\theta)
\sigma S^{\sigma-1}}{\mathcal{AB}(\theta)} \alpha_1
\Vert {L_1}^* - {L_1} \Vert
+ \frac{\theta\sigma S^{\theta+\sigma - 1} \Gamma(\sigma) }{\mathcal{AB}(\theta)
\Gamma(\theta + \sigma)} \alpha_1 \Vert {L_1}^*
- {L_1} \Vert\\[0.3cm]
&\leq r_1\Delta_{\hbar_1}\hbar_1(s)
+ \left[\frac{(1-\theta)\sigma S^{\sigma-1}}{\mathcal{AB}(\theta)}
+  \frac{\theta\sigma S^{\theta+\sigma - 1} \Gamma(\sigma) }{\mathcal{AB}(\theta)
\Gamma(\theta + \sigma)} \right] \alpha_1 \Vert {L_1}-{L_1}^*\Vert.
\end{align*}
Accordingly, we obtain that
\begin{equation*}
\Vert {L_1}^* - {L_1} \Vert
\leq \frac{ r_1 \Delta_{\hbar_1} \hbar_1(s) }{1- \left[
\dfrac{(1-\theta)\sigma S^{\sigma-1}}{\mathcal{AB}(\theta)}
+  \dfrac{\theta\sigma S^{\theta+\sigma - 1}
\Gamma(\sigma)}{\mathcal{AB}(\theta)
\Gamma(\theta + \sigma)} \right] \alpha_1}.
\end{equation*}
Set
$$
a_{(\mathbb{Q}_1,\hbar_1)}
=\dfrac{\Delta_{\hbar_1}}{1- \left[  \dfrac{(1-\theta)\sigma
S^{\sigma-1}}{\mathcal{AB}(\theta)}  +  \dfrac{\theta\sigma
S^{\theta+\sigma - 1} \Gamma(\sigma) }{\mathcal{AB}(\theta)
\Gamma(\theta + \sigma)} \right] \alpha_1}.
$$
Then $\Vert {L_1}^* - {L_1} \Vert
\leq r_1 a_{(\mathbb{Q}_1,\hbar_1)} \hbar_1 (s)$. Similarly,
$$
\Vert {L_2}^* - {L_2} \Vert \leq r_2 a_{(\mathbb{Q}_2,\hbar_2)} \hbar_2 (s),
$$
$$
\Vert {L_3}^* - {L_3} \Vert \leq r_3 a_{(\mathbb{Q}_3,\hbar_3)} \hbar_3 (s),
$$
where
$$
a_{(\mathbb{Q}_j,\hbar_j)}
=\dfrac{\Delta_{\hbar_j}}{1- \left[
\dfrac{(1-\theta)\sigma S^{\sigma-1}}{\mathcal{AB}(\theta)}
+ \dfrac{\theta\sigma S^{\theta+\sigma - 1} \Gamma(\sigma) }{
\mathcal{AB}(\theta) \Gamma(\theta + \sigma)} \right]
\alpha_j},\quad j\in \{ 2,3\}.
$$
As a consequence, the fractal-fractional polluted lake system \eqref{model2}
is stable in the sense of Definition~\ref{def41}. By defining
$r_j=1$, $j\in \{ 1,2,3\}$, our FF polluted lake
system model \eqref{model2} is also stable
in the sense of Definition~\ref{def42}.
\end{proof}


\section{Numerical algorithm via the Adams--Bashforth method}
\label{sec:6}

The Adams--Bashforth method is a robust numerical integration 
technique commonly used for solving most differential equations. 
Its higher-order accuracy and efficiency make it particularly suitable 
for approximating the solution of dynamic systems, such as those 
describing the behavior of polluted lake systems. By choosing the 
Adams--Bashforth technique, we aim to achieve accurate and stable 
numerical solutions for the fractal-fractional polluted 
lake system \eqref{model2}.

To do this, we apply the
fractional Adams--Bashforth technique with two-step
Lagrange polynomials. For that we redefine the fractal-fractional integral
equations \eqref{eq477} at $s_{k+1}$. Precisely, we discretize the integral
equations \eqref{eq477} for $s =  s_{k+1}$ as follows:
\begin{align*}
\begin{cases}
{L_1}(s_{k+1}) = \displaystyle L_{1,0}
+ \frac{(1-\theta)\sigma s_k^{\sigma-1}}{\mathcal{AB}(\theta)}
\mathbb{Q}_1(s_k, {L_1}(s_k), {L_2}(s_k), {L_3}(s_k))\\[0.5cm]
\quad+\displaystyle\frac{\theta\sigma}{\mathcal{AB}(\theta) \Gamma(\theta)}
\int_0^{s_{k+1}} \mathfrak{w}^{\sigma -1} (s_{k+1}- \mathfrak{w})^{\theta-1}
\mathbb{Q}_1(\mathfrak{w}, {L_1}(\mathfrak{w}), {L_2}(\mathfrak{w}),
{L_3}(\mathfrak{w}))\, \mathrm{d}\mathfrak{w},\\[0.7cm]
{L_2}(s_{k+1}) = \displaystyle L_{2,0}
+\frac{(1-\theta)\sigma s_k^{\sigma-1}}{\mathcal{AB}(\theta)}
\mathbb{Q}_2(s_k, {L_1}(s_k), {L_2}(s_k), {L_3}(s_k))\\[0.5cm]
\quad+\displaystyle\frac{\theta\sigma}{\mathcal{AB}(\theta) \Gamma(\theta)}
\int_0^{s_{k+1}} \mathfrak{w}^{\sigma -1} (s_{k+1}- \mathfrak{w})^{\theta-1}
\mathbb{Q}_2(\mathfrak{w}, {L_1}(\mathfrak{w}),
{L_2}(\mathfrak{w}), {L_3}(\mathfrak{w}))\,
\mathrm{d}\mathfrak{w},\\[0.7cm]
{L_3}(s_{k+1}) = \displaystyle L_{3,0}
+ \frac{(1-\theta)\sigma s_k^{\sigma-1}}{\mathcal{AB}(\theta)}
\mathbb{Q}_3(s_k, {L_1}(s_k), {L_2}(s_k), {L_3}(s_k))\\[0.5cm]
\quad+\displaystyle\frac{\theta\sigma}{\mathcal{AB}(\theta)
\Gamma(\theta)} \int_0^{s_{k+1}} \mathfrak{w}^{\sigma -1}
(s_{k+1}- \mathfrak{w})^{\theta-1} \mathbb{Q}_3(\mathfrak{w},
{L_1}(\mathfrak{w}), {L_2}(\mathfrak{w}),
{L_3}(\mathfrak{w}))\, \mathrm{d}\mathfrak{w}.
\end{cases}
\end{align*}
The approximation of the above integrals are given by
\begin{align*}
\begin{cases}
{L_1}(s_{k+1}) = \displaystyle L_{1,0}
+\frac{(1-\theta)\sigma s_k^{\sigma-1}}{\mathcal{AB}(\theta)}
\mathbb{Q}_1(s_k, {L_1}(s_k), {L_2}(s_k),
{L_3}(s_k))\\[0.5cm]
\quad+\displaystyle\frac{\theta\sigma}{\mathcal{AB}(\theta)
\Gamma(\theta)} \sum_{\ell=1}^{k} \int_{s_\ell}^{s_{\ell+1}}
\mathfrak{w}^{\sigma -1} (s_{k+1}- \mathfrak{w})^{\theta-1}
\mathbb{Q}_1(\mathfrak{w}, {L_1}(\mathfrak{w}),
{L_2}(\mathfrak{w}), {L_3}(\mathfrak{w}))
\, \mathrm{d}\mathfrak{w}, \\[0.7cm]
{L_2}(s_{k+1}) = \displaystyle L_{2,0}
+\frac{(1-\theta)\sigma s_k^{\sigma-1}}{\mathcal{AB}(\theta)}
\mathbb{Q}_2(s_k, {L_1}(s_k), {L_2}(s_k),
{L_3}(s_k))\\[0.5cm]
\quad +\displaystyle\frac{\theta\sigma}{\mathcal{AB}(\theta)
\Gamma(\theta)} \sum_{\ell=1}^{k} \int_{s_\ell}^{s_{\ell+1}}
\mathfrak{w}^{\sigma -1} (s_{k+1}- \mathfrak{w})^{\theta-1}
\mathbb{Q}_2(\mathfrak{w}, {L_1}(\mathfrak{w}),
{L_2}(\mathfrak{w}), {L_3}(\mathfrak{w}))
\, \mathrm{d}\mathfrak{w}, \\[0.7cm]
{L_3}(s_{k+1}) = \displaystyle L_{3,0}
+\frac{(1-\theta)\sigma s_k^{\sigma-1}}{\mathcal{AB}(\theta)}
\mathbb{Q}_3(s_k, {L_1}(s_k), {L_2}(s_k),
{L_3}(s_k))\\[0.5cm]
\quad+\displaystyle\frac{\theta\sigma}{\mathcal{AB}(\theta)
\Gamma(\theta)} \sum_{\ell=1}^{k} \int_{s_\ell}^{s_{\ell+1}}
\mathfrak{w}^{\sigma -1} (s_{k+1}- \mathfrak{w})^{\theta-1}
\mathbb{Q}_3(\mathfrak{w}, {L_1}(\mathfrak{w}),
{L_2}(\mathfrak{w}), {L_3}(\mathfrak{w}))
\, \mathrm{d}\mathfrak{w}.
\end{cases}
\end{align*}
Next, we approximate $\mathfrak{w}^{\sigma-1}
\mathbb{Q}_j (\mathfrak{w}, {L_1}(\mathfrak{w}),
{L_2}(\mathfrak{w}), {L_3}(\mathfrak{w}))$, $j=1,2,3$,
on $[s_{\ell} , s_{\ell+1}]$ by applying two-step Lagrange
interpolation polynomials under the step size $\mathbf{h} = s_\ell - s_{\ell-1}$.
By direct computations, we obtain the following algorithm that yields
numerical solutions to the FF-model of polluted lake system \eqref{model2}:
\begin{align}
\label{eqq61}
{L_1}_{,k+1}
&= L_{1,0} + \frac{(1-\theta)\sigma s_k^{\sigma-1}}{\mathcal{AB}(\theta)}
\mathbb{Q}_1(s_k, {L_1}_,k, {L_2}_,k, {L_3}_,k)
+\frac{\sigma \mathbf{h}^{\theta}}{\mathcal{AB}(\theta)\Gamma(\theta+2)}\nonumber\\[0.5cm]
&\times \sum_{\ell=1}^k \left[ s_\ell^{\sigma - 1}
\mathbb{Q}_1\left(s_\ell, {L_1}_{,\ell}, {L_2}_{,\ell}, {L_3}_{,\ell}\right)
\hat{Y}_{1}(k,\ell) - s_{\ell-1}^{\sigma -1} \mathbb{Q}_1\big(s_{\ell-1},
{L_1}_{,\ell-1}, {L_2}_{,\ell-1}, {L_3}_{,\ell-1}  \big)\hat{Y}_{2}(k,\ell) \right],
\end{align}
\begin{align}
\label{eqq62}
{L_2}_{,k+1}
&= L_{2,0} + \frac{(1-\theta)\sigma s_k^{\sigma-1}}{\mathcal{AB}(\theta)}
\mathbb{Q}_2(s_k, {L_1}_,k, {L_2}_,k, {L_3}_,k)
+\frac{\sigma \mathbf{h}^{\theta}}{\mathcal{AB}(\theta)
\Gamma(\theta+2)}\nonumber\\[0.5cm]
&\times \sum_{\ell=1}^k \left[ s_\ell^{\sigma - 1}
\mathbb{Q}_2\left(s_\ell, {L_1}_{,\ell}, {L_2}_{,\ell},
{L_3}_{,\ell}  \right) \hat{Y}_{1}(k,\ell) - s_{\ell-1}^{\sigma -1}
\mathbb{Q}_2\big(s_{\ell-1}, {L_1}_{,\ell-1},
{L_2}_{,\ell-1}, {L_3}_{,\ell-1}\big)\hat{Y}_{2}(k,\ell) \right],
\end{align}
\begin{align}
\label{eqq63}
{L_3}_{,k+1}
&= L_{3,0} + \frac{(1-\theta)\sigma
s_k^{\sigma-1}}{\mathcal{AB}(\theta)}
\mathbb{Q}_3(s_k, {L_1}_,k, {L_2}_,k, {L_3}_,k)
+\frac{\sigma \mathbf{h}^{\theta}}{\mathcal{AB}(\theta)
\Gamma(\theta+2)}\nonumber\\[0.5cm]
&\times\sum_{\ell=1}^k \left[ s_\ell^{\sigma - 1}
\mathbb{Q}_3\big(s_\ell, {L_1}_{,\ell}, {L_2}_{,\ell},
{L_3}_{,\ell}  \big) \hat{Y}_{1}(k,\ell)
- s_{\ell-1}^{\sigma -1} \mathbb{Q}_3\left(s_{\ell-1},
{L_1}_{,\ell-1}, {L_2}_{,\ell-1},
{L_3}_{,\ell-1}  \right)\hat{Y}_{2}(k,\ell) \right],
\end{align}
where
$$
\hat{Y}_{1}(k,\ell) = (k+1-\ell)^\theta (k-\ell+2+\theta)
-(k-\ell)^\theta (k-\ell+2+2\theta)
$$
and
$$
\hat{Y}_{2}(k,\ell)
= (k+1-\ell)^{\theta +1} - (k-\ell)^\theta (k-\ell+1+\theta).
$$


\section{Numerical algorithm via Newton's polynomials}
\label{sec:7}

Here we develop a different approximation algorithm 
(based on Newton's Polynomials) to compute numerically 
the solutions of our fractal-fractional
polluted lake system \eqref{model2}.
The use of Newton's polynomials in interpolation is motivated 
by their simplicity and applicability for approximating functions 
based on a set of given data points. In the context of modeling and analysis, 
Newton's polynomials offer a flexible approach to represent complex relationships 
within the polluted lake system. The polynomial interpolation technique enables 
us to construct a continuous function that approximates the behavior of the system, 
facilitating a more detailed and comprehensive understanding of its dynamics.
To the best of our knowledge, the idea was first introduced in \cite{aa}.
Precisely, we follow \cite{aa} with the IVP \eqref{eq43}
subject to the conditions \eqref{eq44} and \eqref{eq45}.
In this case, we have
\begin{equation*}
\mathbb{K}(s) - \mathbb{K}(0) = \frac{\theta\sigma}{\mathcal{AB}(\theta)
\Gamma(\theta)} \int_0^s \mathfrak{w}^{\sigma -1} (s- \mathfrak{w})^{\theta-1}
\mathbb{Q}(\mathfrak{w}, \mathbb{K}(\mathfrak{w}))\, \mathrm{d}\mathfrak{w}
+\frac{(1-\theta)\sigma s^{\sigma-1}}{\mathcal{AB}(\theta)} \mathbb{Q}(s, \mathbb{K}(s)).
\end{equation*}
Set $ \mathbb{Q}^*(s,\mathbb{K}(s)) = \sigma s^{\sigma -1} \mathbb{Q}(s,\mathbb{K}(s))$. Then,
\begin{equation*}
\mathbb{K}(s) - \mathbb{K}(0) = \frac{\theta}{\mathcal{AB}(\theta)
\Gamma(\theta)} \int_0^s (s- \mathfrak{w})^{\theta-1}
\mathbb{Q}^*(\mathfrak{w}, \mathbb{K}(\mathfrak{w}))\, \mathrm{d}\mathfrak{w}
+\frac{(1-\theta)}{\mathcal{AB}(\theta)} \mathbb{Q}^*(s, \mathbb{K}(s)).
\end{equation*}
By discretizing the above equation at $s=s_{k+1}=(k+1)\mathbf{h}$, we get
\begin{equation*}
\mathbb{K}(s_{k+1}) - \mathbb{K}(0)
= \frac{\theta}{\mathcal{AB}(\theta) \Gamma(\theta)}
\int_0^{s_{k+1}} (s_{k+1}- \mathfrak{w})^{\theta-1}
\mathbb{Q}^*(\mathfrak{w}, \mathbb{K}(\mathfrak{w}))\,
\mathrm{d}\mathfrak{w} + \frac{(1-\theta)}{\mathcal{AB}(\theta)}
\mathbb{Q}^*(s_k, \mathbb{K}(s_k)).
\end{equation*}
Approximating the above integral, we can write that
\begin{equation}
\label{eq71}
\begin{aligned}
\mathbb{K}(s_{k+1})
&= \mathbb{K}_0 + \frac{(1-\theta)}{\mathcal{AB}(\theta)}
\mathbb{Q}^*(s_k, \mathbb{K}(s_k))\\[0.2cm]
&\quad +\frac{\theta}{\mathcal{AB}(\theta) \Gamma(\theta)}
\sum_{\ell=2}^{k}\int_{s_\ell}^{s_{\ell+1}} (s_{k+1}
- \mathfrak{w})^{\theta-1} \mathbb{Q}^*(\mathfrak{w},
\mathbb{K}(\mathfrak{w}))\, \mathrm{d}\mathfrak{w}.
\end{aligned}
\end{equation}
Now we approximate function $\mathbb{Q}^*(s, \mathbb{K}(s))$
with the Newton polynomial
\begin{multline}
\label{eq72}
P_k^*(\mathfrak{w})
= \mathbb{Q}^*(s_{k-2}, \mathbb{K}(s_{k-2}))
+ \frac{\mathbb{Q}^*(s_{k-1}, \mathbb{K}(s_{k-1}))
- \mathbb{Q}^*(s_{k-2}, \mathbb{K}(s_{k-2}))}{\mathbf{h}}
(\mathfrak{w} - s_{k-2})\\
+ \frac{\mathbb{Q}^*(s_{k}, \mathbb{K}(s_{k}))
- 2 \mathbb{Q}^*(s_{k-1}, \mathbb{K}(s_{k-1}))
+ \mathbb{Q}^*(s_{k-2}, \mathbb{K}(s_{k-2}))}{2\mathbf{h}^2}
(\mathfrak{w} - s_{k-2})(\mathfrak{w} - s_{k-1}).
\end{multline}
Substituting \eqref{eq72} into \eqref{eq71}, we obtain that
\begin{align*}
\mathbb{K}_{k+1}
&= \mathbb{K}_0 + \frac{(1-\theta)}{\mathcal{AB}(\theta)}
\mathbb{Q}^*(s_k, \mathbb{K}(s_k)) +\frac{\theta}{\mathcal{AB}(\theta)
\Gamma(\theta)} \sum_{\ell=2}^{k}\int_{s_\ell}^{s_{\ell+1}}
\Big[ \mathbb{Q}^*(s_{\ell-2}, \mathbb{K}_{\ell-2})\\[0.3cm]
&\quad + \frac{\mathbb{Q}^*(s_{\ell-1}, \mathbb{K}_{\ell-1})
- \mathbb{Q}^*(s_{\ell-2}, \mathbb{K}_{\ell-2}) }{\mathbf{h}}
(\mathfrak{w} - s_{\ell-2}) \\[0.3cm]
&\quad + \frac{\mathbb{Q}^*(s_{\ell}, \mathbb{K}_{\ell})
- 2 \mathbb{Q}^*(s_{\ell-1}, \mathbb{K}_{\ell-1})
+ \mathbb{Q}^*(s_{\ell-2}, \mathbb{K}_{\ell-2})}{2\mathbf{h}^2} (\mathfrak{w}
- s_{\ell-2})(\mathfrak{w} - s_{\ell-1}) \Big]\\[0.3cm]
&\quad \times (s_{k+1}- \mathfrak{w})^{\theta-1} \, \mathrm{d}\mathfrak{w}.
\end{align*}
Simplifying the above relations, we get
\begin{align*}
\mathbb{K}_{k+1}
&= \mathbb{K}_0
+\frac{(1-\theta)}{\mathcal{AB}(\theta)}
\mathbb{Q}^*(s_k, \mathbb{K}(s_k)) \\[0.3cm]
&\quad +\frac{\theta}{\mathcal{AB}(\theta) \Gamma(\theta)}
\sum_{\ell=2}^{k} \left[ \int_{s_\ell}^{s_{\ell+1}}
\mathbb{Q}^*(s_{\ell-2}, \mathbb{K}_{\ell-2})(s_{k+1}
-\mathfrak{w})^{\theta - 1}\, \mathrm{d}\mathfrak{w} \right.\\[0.3cm]
&\quad + \int_{s_\ell}^{s_{\ell+1}} \frac{\mathbb{Q}^*(s_{\ell-1},
\mathbb{K}_{\ell-1}) - \mathbb{Q}^*(s_{\ell-2},
\mathbb{K}_{\ell-2}) }{\mathbf{h}} (\mathfrak{w} - s_{\ell-2})(s_{k+1}
-\mathfrak{w})^{\theta - 1}\, \mathrm{d}\mathfrak{w} \\[0.3cm]
&\quad + \int_{s_\ell}^{s_{\ell+1}} \frac{\mathbb{Q}^*(s_{\ell},
\mathbb{K}_{\ell}) - 2 \mathbb{Q}^*(s_{\ell-1}, \mathbb{K}_{\ell-1})
+ \mathbb{Q}^*(s_{\ell-2}, \mathbb{K}_{\ell-2})}{2\mathbf{h}^2}
(\mathfrak{w} - s_{\ell-2})(\mathfrak{w} - s_{\ell-1}) \\[0.3cm]
&\quad \times (s_{k+1}- \mathfrak{w})^{\theta-1}
\, \mathrm{d}\mathfrak{w} \Big]
\end{align*}
and it follows that
\begin{equation}
\label{eq73}
\begin{aligned}
\mathbb{K}_{k+1}
&= \mathbb{K}_0 +
\frac{(1-\theta)}{\mathcal{AB}(\theta)}
\mathbb{Q}^*(s_k, \mathbb{K}(s_k)) \\[0.3cm]
&\quad +\frac{\theta}{\mathcal{AB}(\theta) \Gamma(\theta)}
\sum_{\ell=2}^{k}  \mathbb{Q}^*(s_{\ell-2}, \mathbb{K}_{\ell-2})
\int_{s_\ell}^{s_{\ell+1}}(s_{k+1}-\mathfrak{w})^{\theta - 1}
\, \mathrm{d}\mathfrak{w} \\[0.3cm]
&\quad + \frac{\theta}{\mathcal{AB}(\theta) \Gamma(\theta)}
\sum_{\ell=2}^{k} \frac{\mathbb{Q}^*(s_{\ell-1}, \mathbb{K}_{\ell-1})
- \mathbb{Q}^*(s_{\ell-2}, \mathbb{K}_{\ell-2}) }{\mathbf{h}}
\int_{s_\ell}^{s_{\ell+1}}  (\mathfrak{w} - s_{\ell-2})(s_{k+1}
-\mathfrak{w})^{\theta - 1}\, \mathrm{d}\mathfrak{w} \\[0.3cm]
&\quad + \frac{\theta}{\mathcal{AB}(\theta) \Gamma(\theta)}
\sum_{\ell=2}^{k}  \frac{\mathbb{Q}^*(s_{\ell}, \mathbb{K}_{\ell})
- 2 \mathbb{Q}^*(s_{\ell-1}, \mathbb{K}_{\ell-1})
+ \mathbb{Q}^*(s_{\ell-2}, \mathbb{K}_{\ell-2})}{2\mathbf{h}^2}\\[0.3cm]
&\quad \times \int_{s_\ell}^{s_{\ell+1}} (\mathfrak{w}
- s_{\ell-2})(\mathfrak{w} - s_{\ell-1}) (s_{k+1}
- \mathfrak{w})^{\theta-1} \, \mathrm{d}\mathfrak{w}.
\end{aligned}
\end{equation}
On the other hand, by computing the above three integrals separately,
one gets
\begin{equation}
\label{eq74}
\int_{s_\ell}^{s_{\ell+1}}
(s_{k+1}-\mathfrak{w})^{\theta - 1}\, \mathrm{d}\mathfrak{w}
= \frac{\mathbf{h}^\theta}{\theta}
\left[ (k-\ell + 1)^\theta - (k-\ell)^\theta \right],
\end{equation}
\begin{multline}
\label{eq75}
\int_{s_\ell}^{s_{\ell+1}}  (\mathfrak{w} - s_{\ell-2})(s_{k+1}
-\mathfrak{w})^{\theta - 1}\, \mathrm{d}\mathfrak{w}\\
= \frac{\mathbf{h}^{\theta+1}}{\theta(\theta+1)}
\left[ (k-\ell +1)^\theta (k-\ell +3 + 2\theta)
- (k-\ell +1)^\theta (k-\ell +3 +3\theta) \right],
\end{multline}
and
\begin{equation}
\label{eq76}
\begin{aligned}
\int_{s_\ell}^{s_{\ell+1}} (\mathfrak{w} - s_{\ell-2})
&(\mathfrak{w} - s_{\ell-1})  (s_{k+1}- \mathfrak{w})^{\theta-1}
\, \mathrm{d}\mathfrak{w}
= \frac{\mathbf{h}^{\theta+2}}{\theta(\theta+1)(\theta+2)}
\Big( (k-\ell +1)^\theta \big[ 2(k-\ell)^2 \\[0.2cm]
&+ (3\theta + 10)(k-\ell) + 2\theta^2 + 9\theta + 12 \big]
- (k-\ell)^\theta \big[ 2(k-\ell)^2 \\[0.2cm]
&+ (5\theta + 10)(k-\ell) + 6\theta^2 + 18\theta + 12 \big] \Big).
\end{aligned}
\end{equation}
By putting \eqref{eq74}, \eqref{eq75}, and \eqref{eq76} into \eqref{eq73},
we obtain that
\begin{align}\label{eq77}
\mathbb{K}_{k+1}
&= \mathbb{K}_0
+\frac{(1-\theta)}{\mathcal{AB}(\theta)}
\mathbb{Q}^*(s_k, \mathbb{K}(s_k)) \nonumber\\[0.3cm]
&\quad +\frac{\theta\mathbf{h}^\theta}{\mathcal{AB}(\theta)
\Gamma(\theta + 1)} \sum_{\ell=2}^{k}  \mathbb{Q}^*(s_{\ell-2},
\mathbb{K}_{\ell-2})
\big[ (k-\ell + 1)^\theta - (k-\ell)^\theta \big] \nonumber\\[0.3cm]
&\quad + \frac{\theta \mathbf{h}^\theta }{\mathcal{AB}(\theta)
\Gamma(\theta + 2)} \sum_{\ell=2}^{k} \big[\mathbb{Q}^*(s_{\ell-1},
\mathbb{K}_{\ell-1}) - \mathbb{Q}^*(s_{\ell-2},
\mathbb{K}_{\ell-2}) \big] \nonumber\\[0.3cm]
&\quad \times \big[ (k-\ell +1)^\theta (k-\ell +3 + 2\theta)
- (k-\ell +1)^\theta (k-\ell +3 +3\theta) \big] \nonumber\\[0.3cm]
&\quad + \frac{\theta \mathbf{h}^\theta }{2\mathcal{AB}(\theta)
\Gamma(\theta + 3)} \sum_{\ell=2}^{k}
\big[ \mathbb{Q}^*(s_{\ell}, \mathbb{K}_{\ell})
- 2 \mathbb{Q}^*(s_{\ell-1}, \mathbb{K}_{\ell-1})
+ \mathbb{Q}^*(s_{\ell-2}, \mathbb{K}_{\ell-2}) \big] \nonumber\\[0.3cm]
&\quad \times \Big[ (k-\ell +1)^\theta \big[ 2(k-\ell)^2
+ (3\theta + 10)(k-\ell) + 2\theta^2 + 9\theta + 12 \big]
- (k-\ell)^\theta \big[ 2(k-\ell)^2 \nonumber\\[0.2cm]
&\quad + (5\theta + 10)(k-\ell) + 6\theta^2 + 18\theta + 12 \big] \Big].
\end{align}
Finally, we replace $ \mathbb{Q}^*(s,\mathbb{K}(s))
= \sigma s^{\sigma -1} \mathbb{Q}(s,\mathbb{K}(s))$ into
\eqref{eq77}, and we get that
\begin{equation}
\label{eq78}
\begin{aligned}
\mathbb{K}_{k+1}
&= \mathbb{K}_0 +
\frac{(1-\theta)\sigma s_k^{\sigma-1}}{\mathcal{AB}(\theta)}
\mathbb{Q}(s_k, \mathbb{K}(s_k)) \\[0.3cm]
&+\frac{\theta\sigma\mathbf{h}^\theta}{\mathcal{AB}(\theta)
\Gamma(\theta + 1)} \sum_{\ell=2}^{k}  s_{\ell-2}^{\sigma-1}
\mathbb{Q}(s_{\ell-2}, \mathbb{K}_{\ell-2}) \hat{\Psi}_1(k,\ell,\theta)\\[0.3cm]
&+ \frac{\theta \sigma \mathbf{h}^\theta }{\mathcal{AB}(\theta)
\Gamma(\theta + 2)} \sum_{\ell=2}^{k} \Big[ s_{\ell-1}^{\sigma-1}
\mathbb{Q}(s_{\ell-1}, \mathbb{K}_{\ell-1}) - s_{\ell-2}^{\sigma-1}
\mathbb{Q}(s_{\ell-2}, \mathbb{K}_{\ell-2}) \Big]
\hat{\Psi}_2(k,\ell,\theta) \\[0.3cm]
&+ \frac{\theta \sigma \mathbf{h}^\theta }{2\mathcal{AB}(\theta)
\Gamma(\theta + 3)} \sum_{\ell=2}^{k} \Big[ s_{\ell}^{\sigma-1}
\mathbb{Q}(s_{\ell}, \mathbb{K}_{\ell}) - 2 s_{\ell-1}^{\sigma-1}
\mathbb{Q}(s_{\ell-1}, \mathbb{K}_{\ell-1}) + s_{\ell-2}^{\sigma-1}
\mathbb{Q}(s_{\ell-2}, \mathbb{K}_{\ell-2}) \Big]
\hat{\Psi}_3(k,\ell,\theta),
\end{aligned}
\end{equation}
where
\begin{equation}
\label{eq79}
\begin{aligned}
\hat{\Psi}_1(k,\ell,\theta)
&= (k-\ell + 1)^\theta - (k-\ell)^\theta,\\[0.2cm]
\hat{\Psi}_2(k,\ell,\theta)
&=  (k-\ell +1)^\theta (k-\ell +3 + 2\theta)
- (k-\ell)^\theta (k-\ell +3 +3\theta),\\[0.2cm]
\hat{\Psi}_3(k,\ell,\theta)
&=  (k-\ell +1)^\theta \left[ 2(k-\ell)^2
+ (3\theta + 10)(k-\ell) + 2\theta^2 + 9\theta + 12 \right]\\[0.2cm]
&\quad - (k-\ell)^\theta \left[ 2(k-\ell)^2 + (5\theta + 10)(k-\ell)
+ 6\theta^2 + 18\theta + 12 \right].
\end{aligned}
\end{equation}
Using the numerical scheme \eqref{eq78},
the numerical solutions to the fractal-fractional
polluted lake system \eqref{model2} are given by
\begin{equation}
\label{eq710}
\begin{aligned}
{L_1}_{,k+1}
&= L_{1,0} +
\frac{(1-\theta)\sigma s_k^{\sigma-1}}{\mathcal{AB}(\theta)}
\mathbb{Q}_1(s_k, {L_1}(s_k), {L_2}(s_k), {L_3}(s_k))\\[0.3cm]
&\quad +\frac{\theta\sigma\mathbf{h}^\theta}{\mathcal{AB}(\theta)
\Gamma(\theta + 1)} \sum_{\ell=2}^{k}  s_{\ell-2}^{\sigma-1}
\mathbb{Q}_1(s_{\ell-2}, {L_1}_{,\ell-2}, {L_2}_{,\ell-2},
{L_3}_{,\ell-2}) \hat{\Psi}_1(k,\ell,\theta)\\[0.3cm]
&\quad + \frac{\theta \sigma \mathbf{h}^\theta }{\mathcal{AB}(\theta)
\Gamma(\theta + 2)} \sum_{\ell=2}^{k} \Big[ s_{\ell-1}^{\sigma-1}
\mathbb{Q}_1(s_{\ell-1}, {L_1}_{,\ell-1}, {L_2}_{,\ell-1},
{L_3}_{,\ell-1})\\
&\quad - s_{\ell-2}^{\sigma-1} \mathbb{Q}_1(s_{\ell-2},
{L_1}_{,\ell-2}, {L_2}_{,\ell-2},
{L_3}_{,\ell-2}) \Big] \hat{\Psi}_2(k,\ell,\theta)\\[0.3cm]
&\quad + \frac{\theta \sigma \mathbf{h}^\theta }{2\mathcal{AB}(\theta)
\Gamma(\theta + 3)} \sum_{\ell=2}^{k} \Big[ s_{\ell}^{\sigma-1}
\mathbb{Q}_1(s_{\ell}, {L_1}_{,\ell},
{L_2}_{,\ell}, {L_3}_{,\ell})\\
&\quad - 2 s_{\ell-1}^{\sigma-1} \mathbb{Q}_1(s_{\ell-1},
{L_1}_{,\ell-1}, {L_2}_{,\ell-1},
{L_3}_{,\ell-1})\\[0.3cm]
&\quad + s_{\ell-2}^{\sigma-1} \mathbb{Q}_1(s_{\ell-2},
{L_1}_{,\ell-2}, {L_2}_{,\ell-2},
{L_3}_{,\ell-2}) \Big] \hat{\Psi}_3(k,\ell,\theta),
\end{aligned}
\end{equation}
\begin{equation}
\label{eq711}
\begin{aligned}
{L_2}_{,k+1}
&= L_{2,0} +
\frac{(1-\theta)\sigma s_k^{\sigma-1}}{\mathcal{AB}(\theta)}
\mathbb{Q}_2(s_k, {L_1}(s_k), {L_2}(s_k), {L_3}(s_k))\\[0.3cm]
&\quad +\frac{\theta\sigma\mathbf{h}^\theta}{\mathcal{AB}(\theta)
\Gamma(\theta + 1)} \sum_{\ell=2}^{k}  s_{\ell-2}^{\sigma-1}
\mathbb{Q}_2(s_{\ell-2}, {L_1}_{,\ell-2}, {L_2}_{,\ell-2},
{L_3}_{,\ell-2}) \hat{\Psi}_1(k,\ell,\theta)\\[0.3cm]
&\quad + \frac{\theta \sigma \mathbf{h}^\theta }{\mathcal{AB}(\theta)
\Gamma(\theta + 2)} \sum_{\ell=2}^{k} \Big[ s_{\ell-1}^{\sigma-1}
\mathbb{Q}_2(s_{\ell-1}, {L_1}_{,\ell-1}, {L_2}_{,\ell-1},
{L_3}_{,\ell-1})\\
&\quad - s_{\ell-2}^{\sigma-1} \mathbb{Q}_2(s_{\ell-2}, {L_1}_{,\ell-2},
{L_2}_{,\ell-2}, {L_3}_{,\ell-2}) \Big] \hat{\Psi}_2(k,\ell,\theta)\\[0.3cm]
&\quad + \frac{\theta \sigma \mathbf{h}^\theta }{2\mathcal{AB}(\theta)
\Gamma(\theta + 3)} \sum_{\ell=2}^{k} \Big[ s_{\ell}^{\sigma-1}
\mathbb{Q}_2(s_{\ell}, {L_1}_{,\ell}, {L_2}_{,\ell}, {L_3}_{,\ell})\\
&\quad - 2 s_{\ell-1}^{\sigma-1} \mathbb{Q}_2(s_{\ell-1}, {L_1}_{,\ell-1},
{L_2}_{,\ell-1}, {L_3}_{,\ell-1}) \\[0.3cm]
&\quad + s_{\ell-2}^{\sigma-1} \mathbb{Q}_2(s_{\ell-2}, {L_1}_{,\ell-2},
{L_2}_{,\ell-2}, {L_3}_{,\ell-2}) \Big] \hat{\Psi}_3(k,\ell,\theta),
\end{aligned}
\end{equation}
and
\begin{equation}
\label{eq712}
\begin{aligned}
{L_3}_{,k+1}
&= L_{3,0} +
\frac{(1-\theta)\sigma s_k^{\sigma-1}}{\mathcal{AB}(\theta)}
\mathbb{Q}_3(s_k, {L_1}(s_k), {L_2}(s_k),
{L_3}(s_k))\\[0.3cm]
&\quad +\frac{\theta\sigma\mathbf{h}^\theta}{\mathcal{AB}(\theta) \Gamma(\theta + 1)}
\sum_{\ell=2}^{k}  s_{\ell-2}^{\sigma-1} \mathbb{Q}_3(s_{\ell-2},
{L_1}_{,\ell-2}, {L_2}_{,\ell-2}, {L_3}_{,\ell-2})
\hat{\Psi}_1(k,\ell,\theta)\\[0.3cm]
&\quad + \frac{\theta \sigma \mathbf{h}^\theta }{\mathcal{AB}(\theta)
	\Gamma(\theta + 2)} \sum_{\ell=2}^{k}
\Big[ s_{\ell-1}^{\sigma-1} \mathbb{Q}_3(s_{\ell-1},
{L_1}_{,\ell-1}, {L_2}_{,\ell-1}, {L_3}_{,\ell-1})\\
&\quad - s_{\ell-2}^{\sigma-1} \mathbb{Q}_3(s_{\ell-2}, {L_1}_{,\ell-2},
{L_2}_{,\ell-2}, {L_3}_{,\ell-2}) \Big]
\hat{\Psi}_2(k,\ell,\theta)\\[0.3cm]
&\quad + \frac{\theta \sigma \mathbf{h}^\theta }{2\mathcal{AB}(\theta)
\Gamma(\theta + 3)} \sum_{\ell=2}^{k} \Big[ s_{\ell}^{\sigma-1}
\mathbb{Q}_3(s_{\ell}, {L_1}_{,\ell}, {L_2}_{,\ell}, {L_3}_{,\ell})\\
&\quad - 2 s_{\ell-1}^{\sigma-1} \mathbb{Q}_3(s_{\ell-1}, {L_1}_{,\ell-1},
{L_2}_{,\ell-1}, {L_3}_{,\ell-1})\\[0.3cm]
&\quad + s_{\ell-2}^{\sigma-1} \mathbb{Q}_3(s_{\ell-2},
{L_1}_{,\ell-2}, {L_2}_{,\ell-2},
{L_3}_{,\ell-2}) \Big] \hat{\Psi}_3(k,\ell,\theta),
\end{aligned}
\end{equation}
where $\hat{\Psi}_j(k,\ell,\theta)$
are defined in \eqref{eq79}, $j = 1,2,3$.


\section{Numerical simulations and discussion}
\label{sec:8}

Now we apply the Adams--Bashforth method (ABM) and Newton's polynomials method (NPM),
proposed respectively in Sections~\ref{sec:6} and \ref{sec:7}, to examine
and find numerical solutions $L_1, L_2, L_3$ of the proposed FF-model
and to observe the applicability, accuracy, and exactness of the developed algorithms.
To simulate the quantity of pollution in the modeled lakes, we coded
the algorithms \eqref{eqq61}--\eqref{eqq63} and
\eqref{eq710}--\eqref{eq712} in MATLAB, version R2019A.

To compare the results, we borrow from \cite{baz}
the following values for the parameters: $F_{21}=18$ $mi^{3}/year$,
$F_{31}=20$ $mi^{3}/year$, $F_{32}=18$ $mi^{3}/year$, $F_{13}=38$ $mi^{3}/year$,
$V_{1}=2900$ $mi^{3}$, $V_{2}=850$ $mi^{3}$, $V_{3}=1180$ $mi^{3}$.
Moreover, $L_{1,0}=L_{2,0}=L_{3,0}=0$. Also, various fractal
dimensions and fractional orders, i.e., $\theta=\sigma=0.85, 0.90, 0.95, 0.99$,
are considered for the simulations of the three state functions $L_1$,
$L_2$, and $L_3$.

We consider the suggested FF-model in three cases:
linear (Section~\ref{sec:8.1}),
exponentially decaying (Section~\ref{sec:8.2}),
and periodic (Section~\ref{sec:8.3}) input models.


\subsection{Linear input model}
\label{sec:8.1}

In this case, we consider the model in which the Lake~1 has a contaminant
with a linear concentration. Linear input states the steady behavior of the pollutant.
At time zero, the pollutant concentration is zero but, as the time increases,
the addition of pollutant is started and then is remained steadily. For example,
when a factory starts production at time zero, waste discharge begins at a fixed
rate and concentration. As a particular case, we chose $c(s)=\mu s$. Then,
for $\mu=100$, from \eqref{model2} we have
\begin{equation}
\label{model2x2}
\begin{cases}
{}^{\mathbf{FFML}}\mathcal{D}_{0,s}^{(\theta,\sigma)}{L_1}(s)
= \frac{38}{1180}L_3(s)+100s-\frac{20}{2900}
L_1(s)-\frac{18}{2900}L_1(s),\\[0.3cm]
{}^{\mathbf{FFML}}\mathcal{D}_{0,s}^{(\theta,\sigma)}{L_2}(s)
= \frac{18}{2900}L_1(s)-\frac{18}{850}L_2(s),\\[0.3cm]
{}^{\mathbf{FFML}}\mathcal{D}_{0,s}^{(\theta,\sigma)}{L_3}(s)
=\frac{20}{2900}L_1(s)+\frac{18}{850}L_2(s)
-\frac{38}{1180}L_3(s).
\end{cases}
\end{equation}
In Figures~\ref{fig:figsub2} (a), (b), and (c), the behavior
of the ABM approximations for each pair of the state functions
$L_{1}-L_{2}$, $L_{1}-L_{3}$,
and $L_{2}-L_{3}$, respectively, are given;
while in Figure \ref{fig:figsub2} (d), the 3D view
of $L_{1}-L_{2}-L_{3}$
under integer-order derivatives are graphically
illustrated for the linear input model
with time $s\in[0,60]$ and step size $h=0.1$.

Note that the parameter $h$ is explicitly defined 
as the step size, distinct from the stability parameters $\hbar_j$, 
$j\in \{ 1,2,3\}$, discussed in the stability Section~\ref{sec:5}. 
While here we emphasize $h$ as the step size in a specific context, 
Ulam--Hyers--Rassias stability, as a theory, is primarily concerned 
with the stability properties of functional equations. Unlike the numerical 
solution of differential equations, the choice of step size is not a direct 
consideration in the realm of Ulam--Hyers--Rassias stability. This stability 
theory focuses on understanding how small variations in functional equations 
lead to proportionate changes in the solutions, and the concept of a step size 
does not play a prominent role in that context.

\begin{figure}[ht!]
\centering
\subfigure[]{
\includegraphics[width=0.48\linewidth]{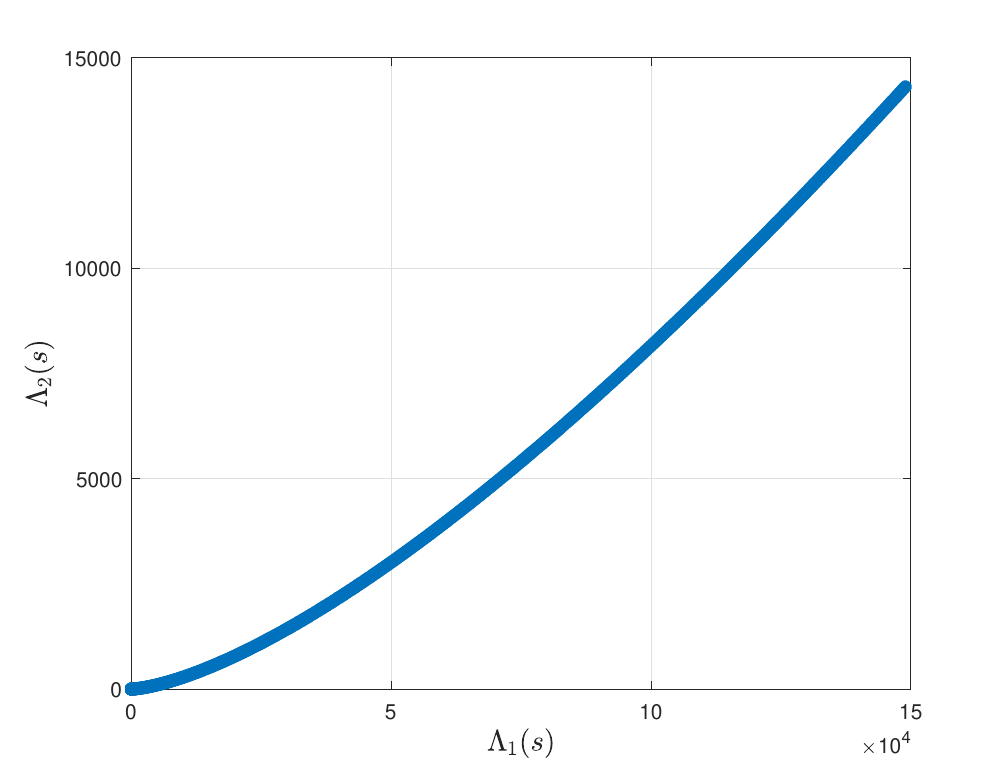}
\label{fig:sub2-first}}
\subfigure[]{
\includegraphics[width=0.48\linewidth]{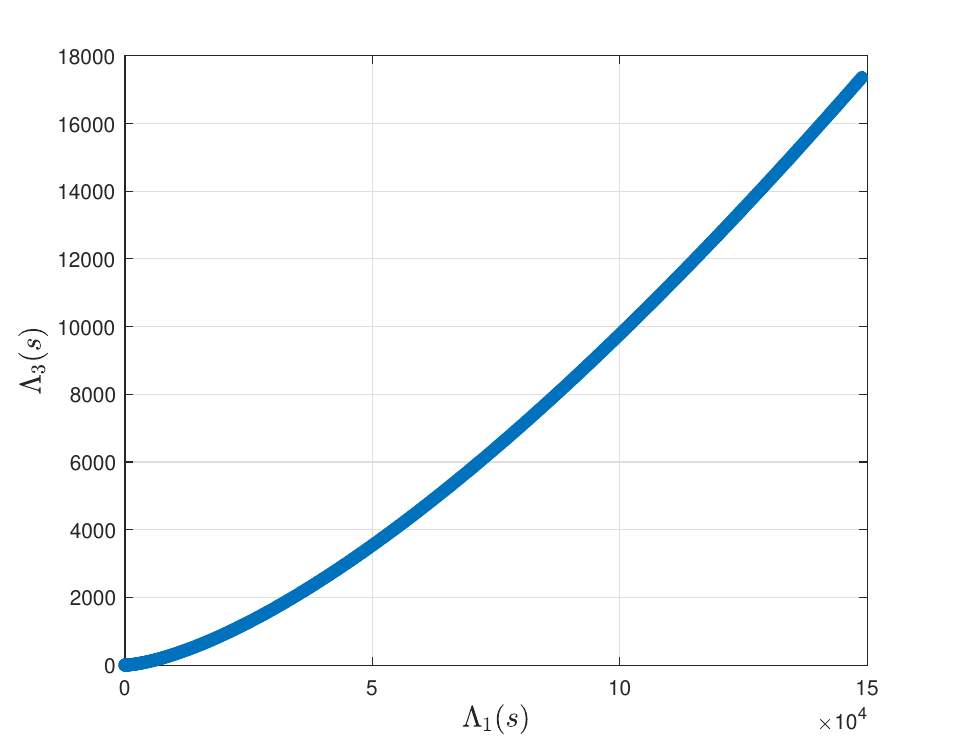}
\label{fig:sub2-second}}
\newline
\subfigure[]{
\includegraphics[width=0.48\linewidth]{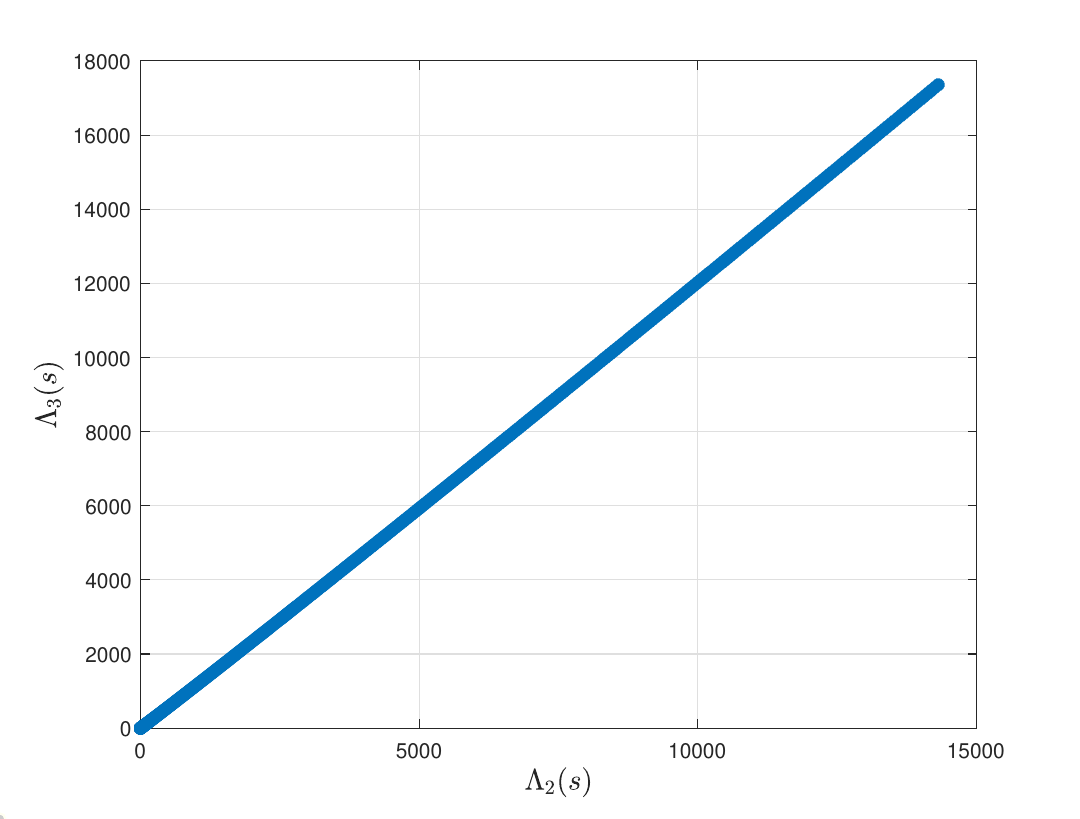}
\label{fig:sub2-third}}
\subfigure[]{
\includegraphics[width=0.48\linewidth]{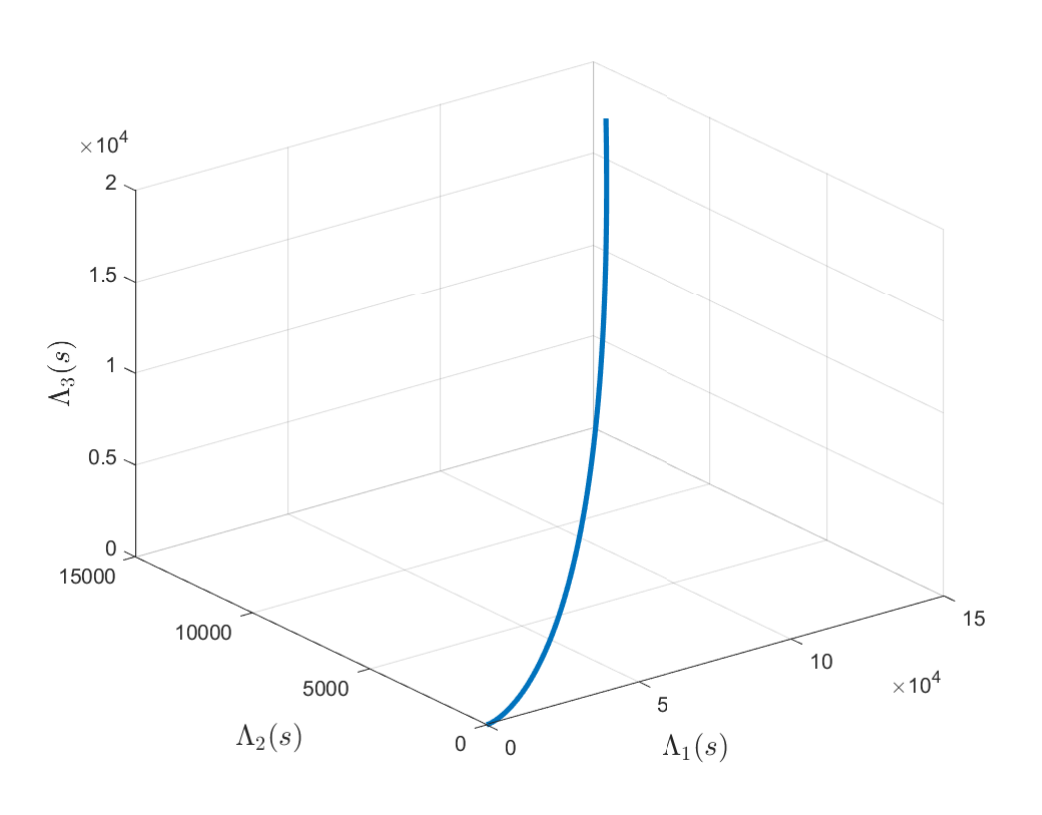}
\label{fig:sub2-fourth}}
\caption{Behaviors of each pair of state functions
(a) $L_{1}-L_{2}$, (b) $L_{1}-L_{3}$,
(c) $L_{2}-L_{3}$ and (d)~3D view of
$L_{1}-L_{2}-L_{3}$ under the integer-order.}
\label{fig:figsub2}
\end{figure}

In Table~\ref{table:1}, we present some numerical results of the two numerical techniques,
ABM and NPM, for the three state functions $L_{1}$, $L_{2}$ and $L_{3}$
in the linear input case, under integer-order derivatives
and step size $h=0.1$. From the obtained numerical results,
we can assert that the Adams--Bashforth approximations for the phase functions
$L_{1}$, $L_{2}$, and $L_{3}$ strongly agree with the ones obtained
by the Newton polynomials method for the time $s$ up to $10$ years.
\begin{table}[ht!]
\begin{minipage}[t]{1\linewidth}
\makeatletter\def\@captype{table}\makeatother
\caption{Comparison between ABM and NPM for the linear input case.}
\label{table:1}
\begin{center}
\begin{tabular}{c c c c c c c}
\toprule
\multirow{2.5}{*}{$s$} & \multicolumn{3}{c}{Adams--Bashforth}
& \multicolumn{3}{c}{Newton Polynomials}\\
\cmidrule(lr){2-4} \cmidrule(lr){5-7} 	
& $L_{1}$ &  $L_{2}$  &  $L_{3}$
&  $L_{1}$  &  $L_{2}$  &  $L_{3}$\\ \midrule
$0$ & $0.000000$ & $0.00000$ & $0.000000$ & $0.000000$ & $0.000000$ & $0.000000$\\		
$1$ & $58.241762$ & $0.122200$ & $0.136038$ & $58.249029$ & $0.114027$ & $0.126944$\\				
$2$ & $216.569135$ & $0.912690$ & $1.018175$ & $216.603012$ & $0.891958$ & $0.995030$\\				
$3$ & $472.231987$ & $2.959673$ & $3.308340$ & $472.291839$ & $2.926723$ & $3.271438$\\				
$4$ & $824.018631$ & $6.830863$ & $7.650506$ & $824.103833$ & $6.786029$ & $7.600138$\\				
$5$ & $1270.753481$ & $13.074750$ & $14.671769$ & $1270.863423$ & $13.018359$ & $14.608218$\\				
$6$ & $1811.296055$ & $22.221220$ & $24.982726$ & $1811.430139$ & $22.153588$ & $24.906273$\\				
$7$ & $2444.539990$ & $34.782156$ & $39.177854$ & $2444.697632$ & $34.703595$ & $39.088774$\\				
$8$ & $3169.412094$ & $51.252025$ & $57.835881$ & $3169.592721$ & $51.162836$ & $57.734443$\\				
$9$ & $3984.871413$ & $72.108440$ & $81.520147$ & $3985.074465$ & $72.008917$ & $81.406618$\\				
$10$ & $4889.908324$ & $97.812711$ & $110.778966$ & $4890.133254$ & $97.703141$ & $110.653605$\\[1ex]
\bottomrule
\end{tabular}
\end{center}
\end{minipage}
\end{table}

In Figure~\ref{fig:fig4}, the comparison of the numerical results from ABM and NPM for the state
functions $L_{1}$, $L_{2}$, and $L_{3}$ is shown graphically, for the time
$s\in[0,60]$ and the linear input case. We observe that the results of ABM and NPM have a high
agreement between them for each one of the state functions, even in the longer period of 60 years.
\begin{figure}[ht!]
\centering
\subfigure[]{
\includegraphics[width=0.48\linewidth]{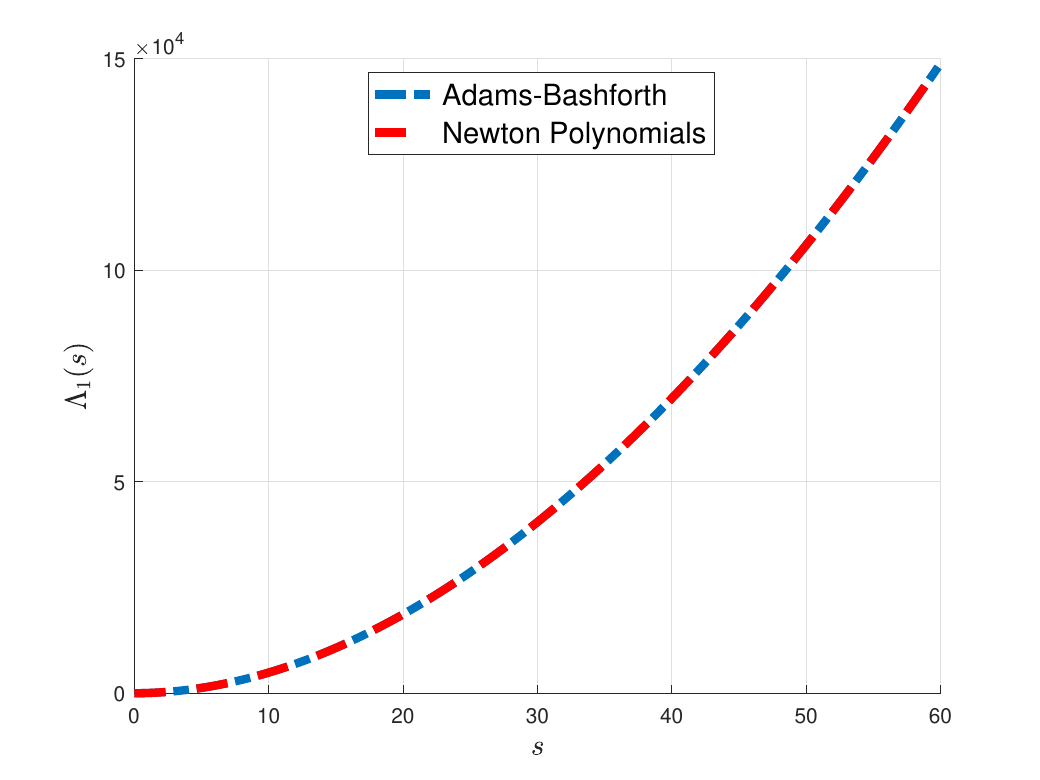}
\label{fig_linear1}}
\subfigure[]{
\includegraphics[width=0.48\linewidth]{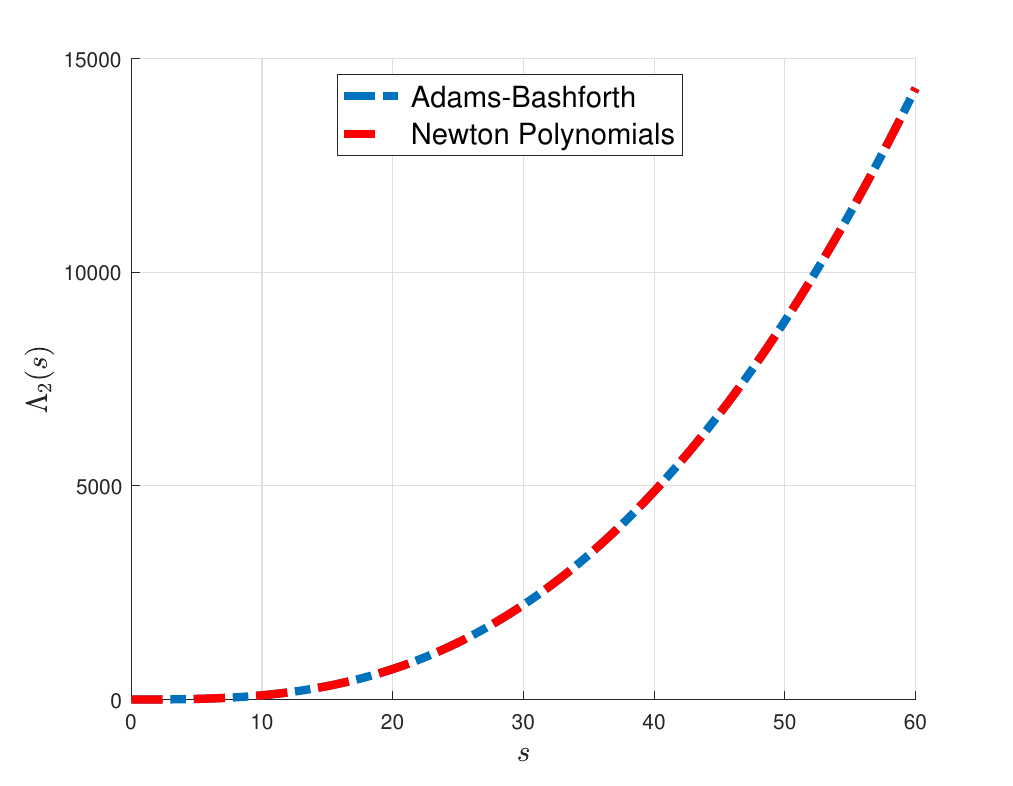}
\label{fig_linear2}}
\newline
\hfill
\subfigure[]{
\includegraphics[width=0.48\linewidth]{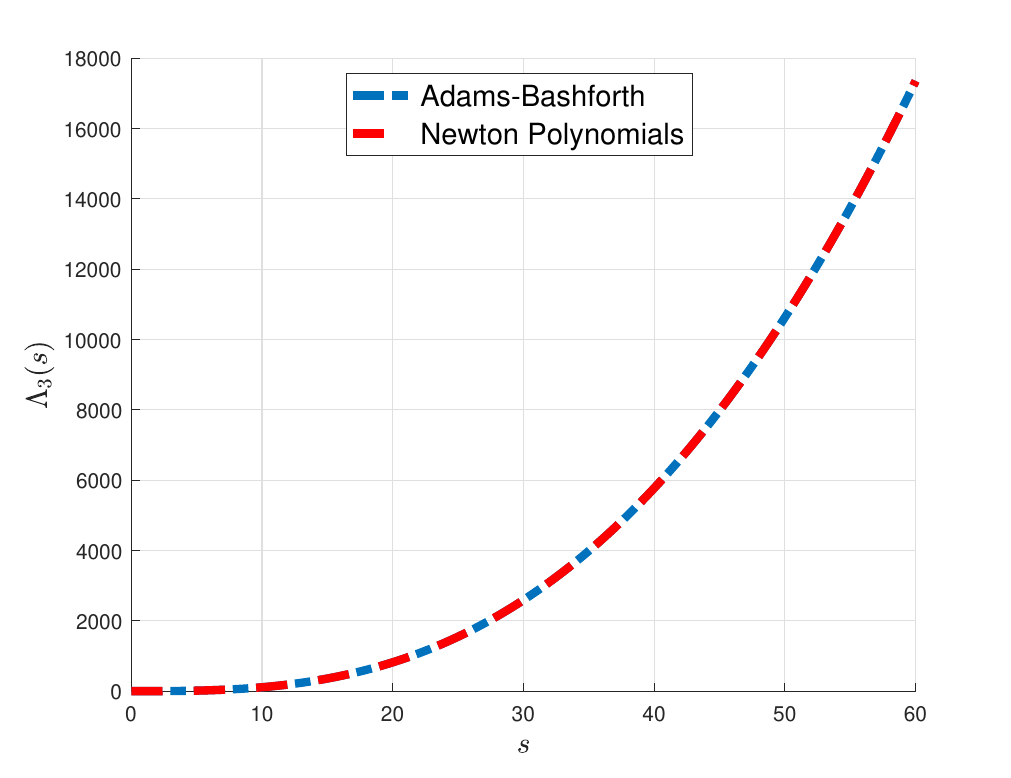}
\label{fig_linear3}}
\caption{Comparison between the ABM and NPM for
(a) $L_{1}(s)$, (b) $L_{2}(s)$ and
(c) $L_{3}(s)$ in the linear input model.}
\label{fig:fig4}
\end{figure}


In Figure~\ref{fig:figfraclinear}, we illustrate the behavior
of the three state functions $L_{1}$, $L_{2}$, and $L_{3}$
when the ABM is applied under the fractal-fractional orders $\theta=\sigma=0.85, 0.90, 0.95, 0.99$.
From these figures, we can observe that when the fractal-fractional order is getting closer
to the integer case, then the effect of the pollution is increasing on each lake model
at about the same rate. As an observation of these graphs, it can be said that the
non-integer order operator has a positive effect on the pollution
reduction in the lake pollution model.
\begin{figure}[ht!]
\centering
\subfigure[]{
\includegraphics[width=0.48\linewidth]{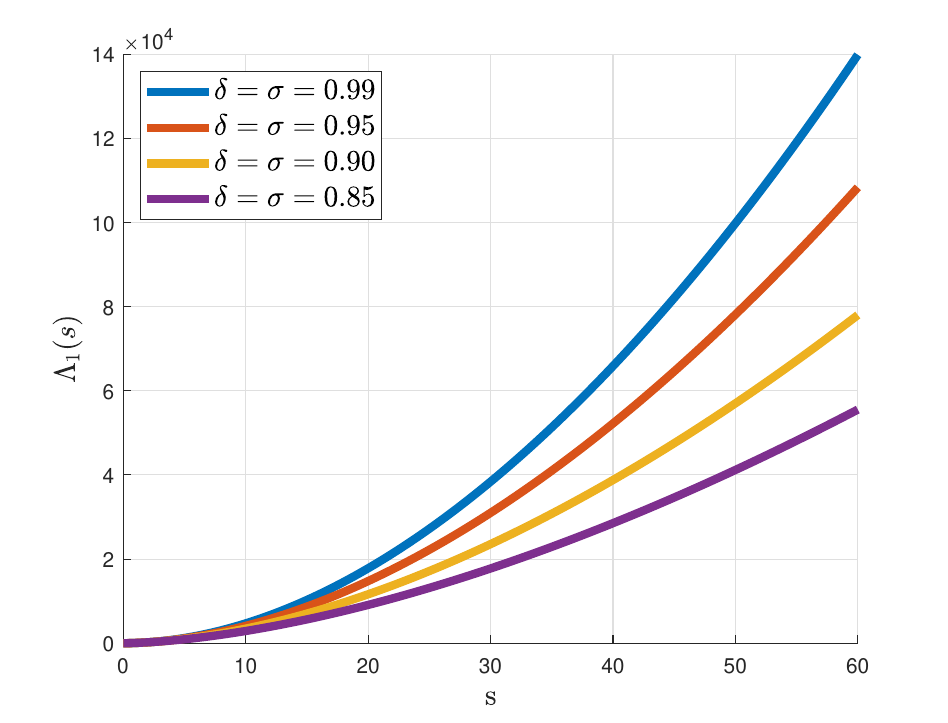}
\label{fig_fraclinear1}}
\subfigure[]{
\includegraphics[width=0.48\linewidth]{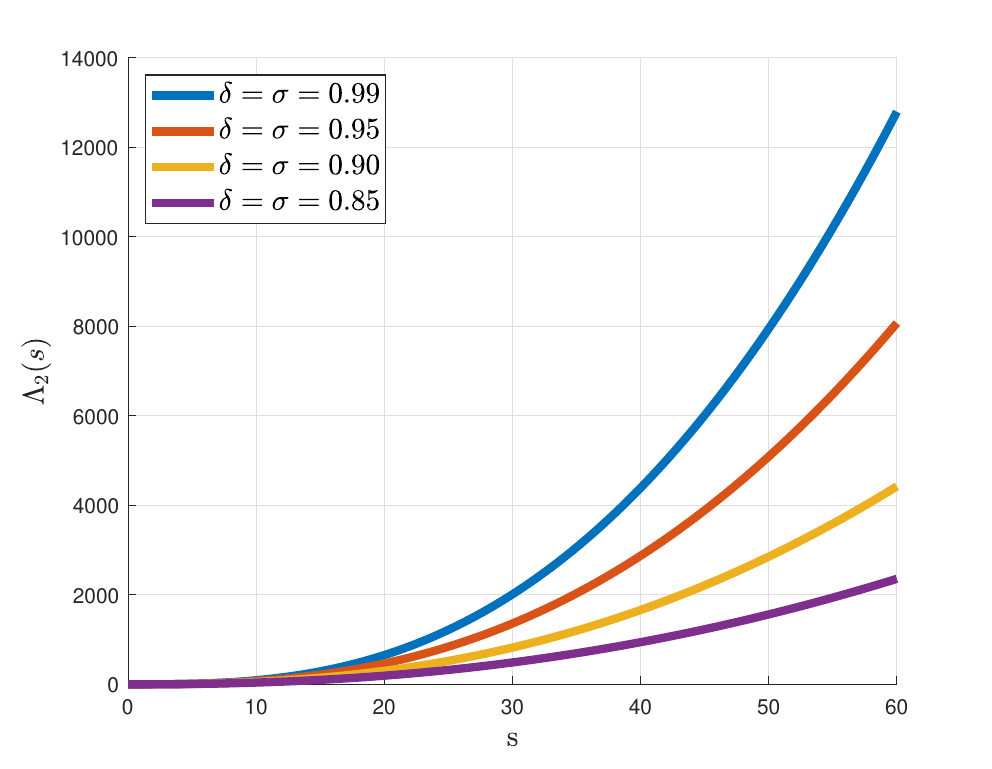}
\label{fig_fraclinear2}}
\newline
\hfill
\subfigure[]{
\includegraphics[width=0.48\linewidth]{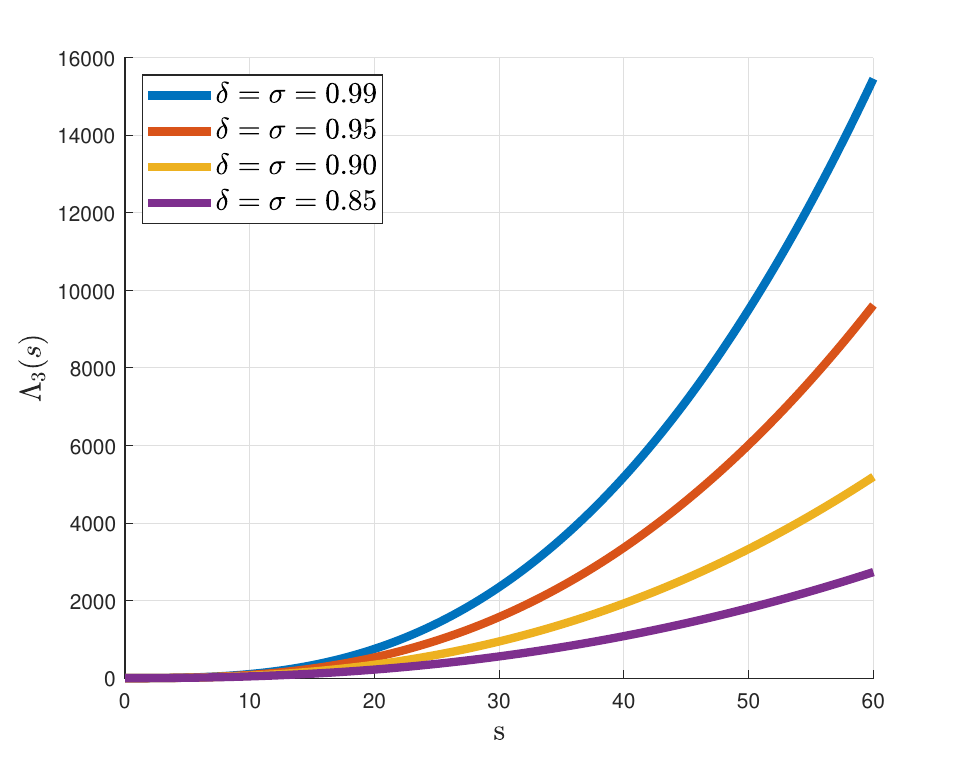}
\label{fig_fraclinear3}}
\caption{Behaviors of (a) $L_{1}(s)$, (b) $L_{2}(s)$
and (c) $L_{3}(s)$ for the linear input model under fractal
and fractional orders $\theta=\sigma=0.85, 0.90, 0.95, 0.99$.}
\label{fig:figfraclinear}
\end{figure}

A word is due about our choice of the values of the fractal-fractional orders. 
We considered fractional orders within the range of $[0.85, 1]$ because 
within this interval we observed consistent behaviors for different 
fractional orders. Specifically, as the fractal-fractional order decreases, 
we noted a proportional reduction in the impact of pollution 
on each lake model at about the same rate. This consistent trend in behavior 
as the fractional order decreases led us to cut the interval at the 
value 0.85. This choice captures the essential aspects of the model's 
response to varying fractional orders and provides a meaningful 
representation of the system dynamics.


\subsection{Exponentially decaying input model}
\label{sec:8.2}

When heavy dumping of pollutant is present,
it makes sense to consider the exponentially decaying input model,
i.e., the case when $c(s)=re^{-ps}$. An example of this case occurs
if every industry placed in a city collects and stores its wastage
during some days and then dumps it to Lake~1 after that stored period.
If we take $r=200$ and $p=10$, then system \eqref{model2} becomes
\begin{equation}
\label{model2x3}
\begin{cases}
{}^{\mathbf{FFML}}\mathcal{D}_{0,s}^{(\theta,\sigma)}{L_1}(s)
= \frac{38}{1180}L_3(s)+200e^{-10s}-\frac{20}{2900}L_1(s)
-\frac{18}{2900}L_1(s),\\[0.3cm]
{}^{\mathbf{FFML}}\mathcal{D}_{0,s}^{(\theta,\sigma)}{L_2}(s)
= \frac{18}{2900}L_1(s)-\frac{18}{850}L_2(s),\\[0.3cm]
{}^{\mathbf{FFML}}\mathcal{D}_{0,s}^{(\theta,\sigma)}{L_3}(s)
=\frac{20}{2900}L_1(s)+\frac{18}{850}L_2(s)
-\frac{38}{1180}L_3(s).
\end{cases}
\end{equation}

The graphical representation of the input function $c(s)$ is illustrated
in Figure~\ref{fig3x} for the exponentially decaying input case
$c(s)=200e^{-10s}$, $s\in[0,1]$.
\begin{figure}[ht!]
\centering
\includegraphics[width=0.6\linewidth]{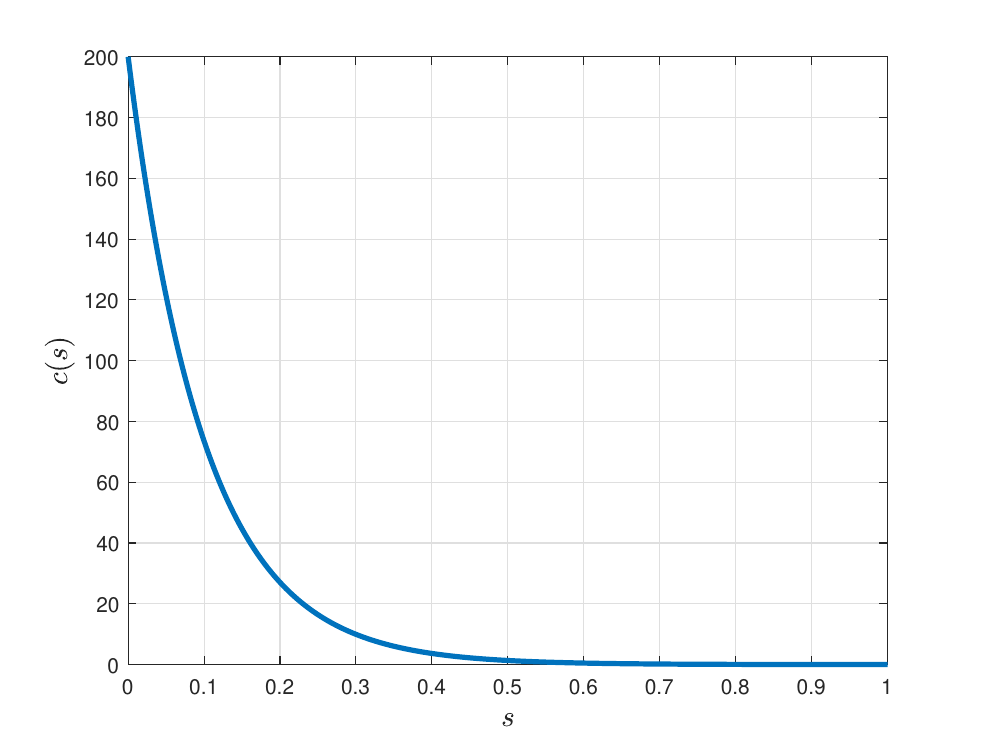}
\caption{Graphic of the exponentially decaying input.}
\label{fig3x}
\end{figure}

In Figures~\ref{fig:figsub3} (a), (b), and (c), the behavior
of the ABM approximations for each pair of the state functions
$L_{1}-L_{2}$, $L_{1}-L_{3}$, and
$L_{2}-L_{3}$, respectively, is shown,
while in Figure \ref{fig:figsub3} (d), the 3D view of $L_{1}-L_{2}-L_{3}$
under the integer-order derivative is graphically illustrated for the exponentially
decaying input model with time $s\in[0,60]$ and step size $h=0.01$.
\begin{figure}[ht!]
\centering
\subfigure[]{
\includegraphics[width=0.48\linewidth]{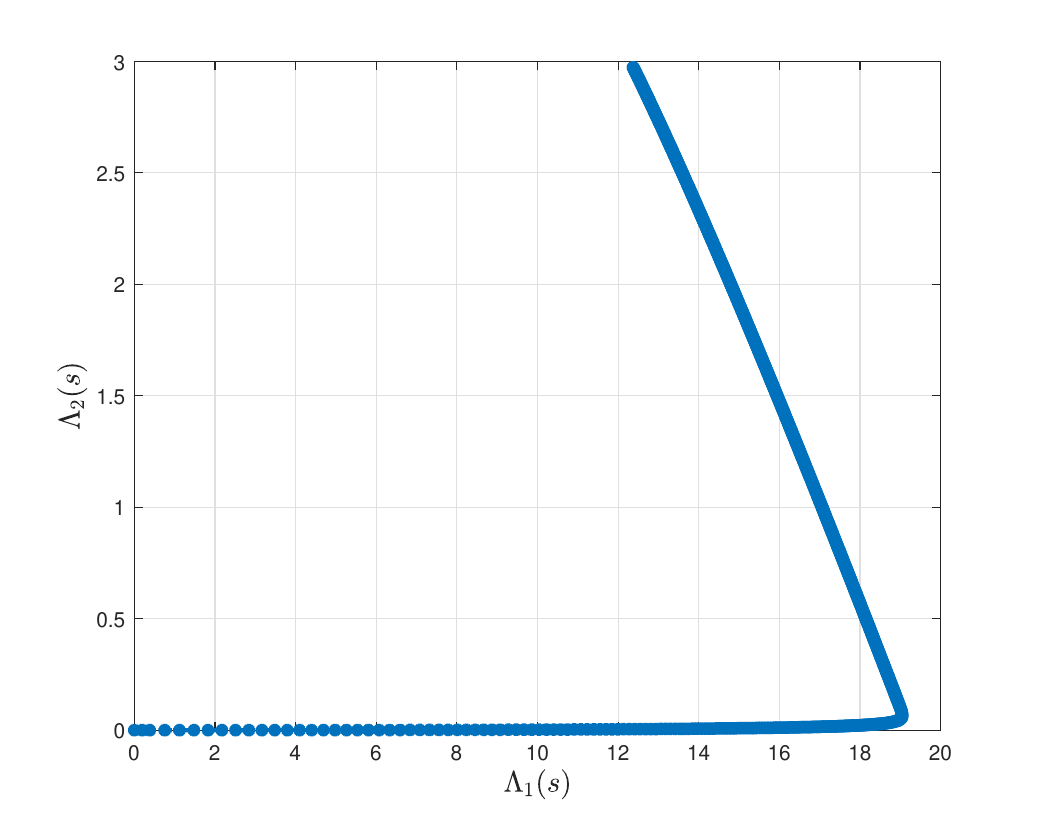}
\label{fig:sub3-first}}
\subfigure[]{
\includegraphics[width=0.48\linewidth]{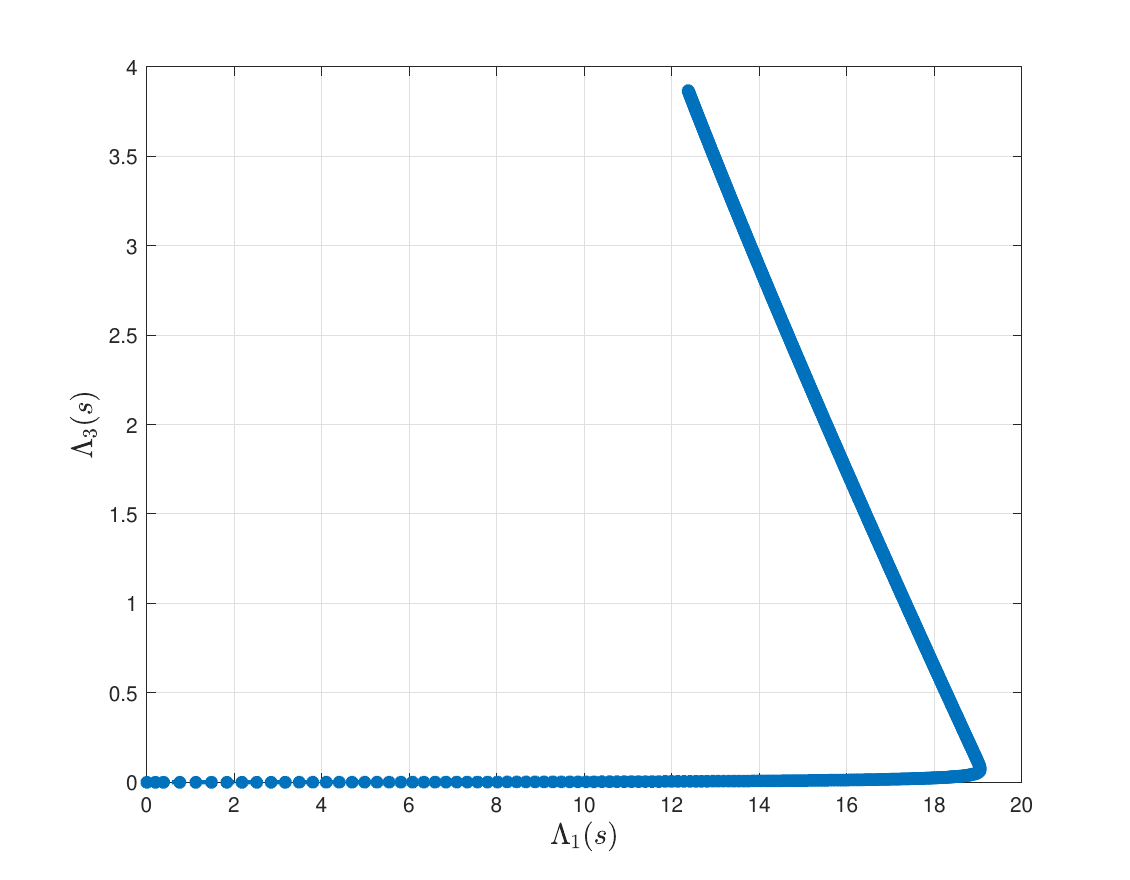}
\label{fig:sub3-second}}
\newline
\subfigure[]{
\includegraphics[width=0.48\linewidth]{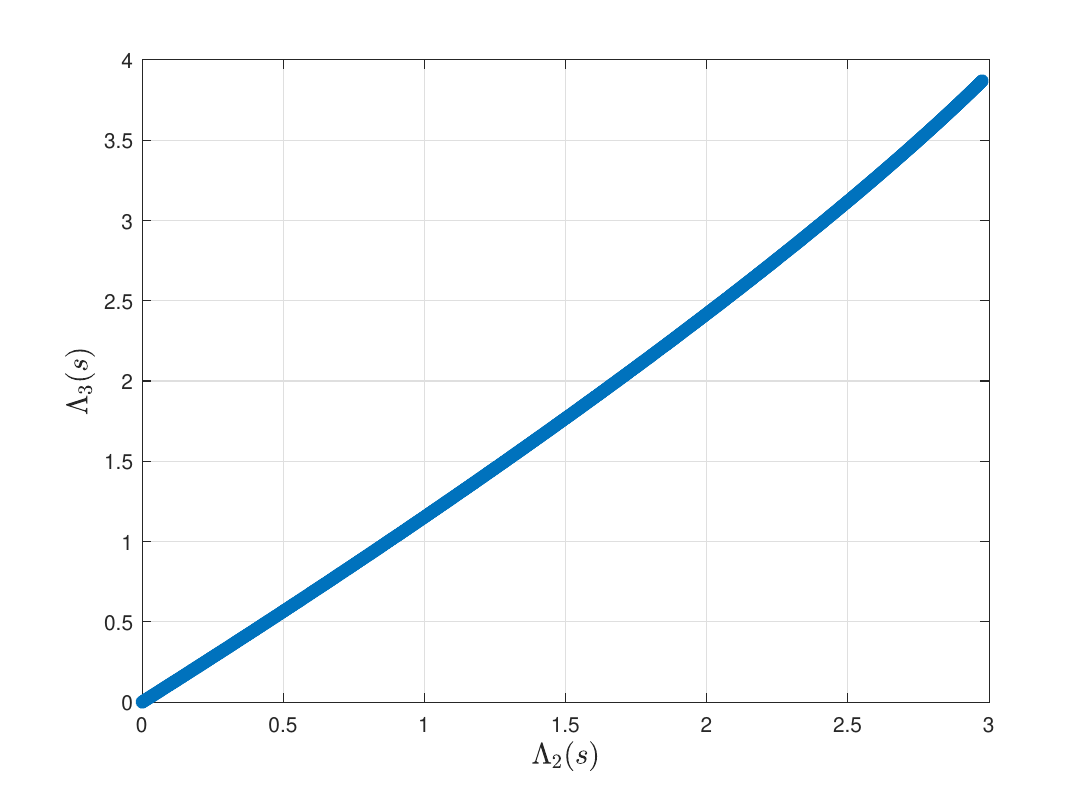}
\label{fig:sub3-third}}
\subfigure[]{
\includegraphics[width=0.48\linewidth]{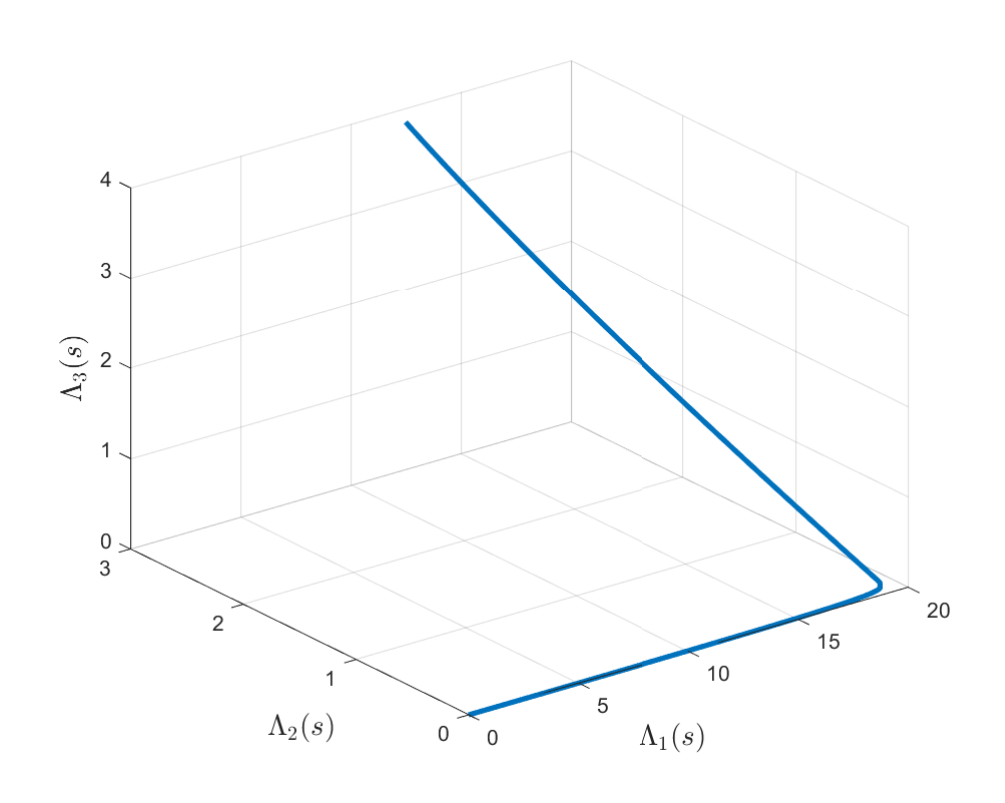}
\label{fig:sub3-fourth}}
\caption{Behaviors of each pair of state functions
(a) $L_{1}-L_{2}$, (b) $L_{1}-L_{3}$,
(c) $L_{2}-L_{3}$ and (d)~3D view of
$L_{1}-L_{2}-L_{3}$ under the integer-order.}
\label{fig:figsub3}
\end{figure}

In Table~\ref{table:2}, we provide a comparison
between the approximate solutions obtained using ABM and NPM
for the exponentially decaying input case with time $s\in[0,10]$,
step size $h=0.01$, and $\theta=\sigma=1$. From Table~\ref{table:2},
we can conclude that the ABM approximations are also in good agreement
with the NPM ones for the exponentially decaying input model.
\begin{table}[ht!]
\begin{minipage}[t]{1\linewidth}
\makeatletter\def\@captype{table}\makeatother
\caption{Numerical comparison between ABM and NPM
for the exponentially decaying input case.}
\label{table:2}
\begin{center}
\begin{tabular}{c c c c c c c}
\toprule
\multirow{2.5}{*}{$s$} & \multicolumn{3}{c}{Adams--Bashforth}
& \multicolumn{3}{c}{Newton Polynomials}\\ \cmidrule(lr){2-4}
\cmidrule(lr){5-7} 	& $L_{1}$ &  $L_{2}$  &  $L_{3}$
& $L_{1}$  &  $L_{2}$  &  $L_{3}$\\ \midrule
$0$ & $0.000000$ & $0.00000$ & $0.000000$ & $0.000000$ & $0.000000$ & $0.000000$\\				
$1$ & $18.988642$ & $0.105437$ & $0.117578$ & $18.988592$ & $0.105145$ & $0.117021$\\				
$2$ & $18.748146$ & $0.219107$ & $0.245295$ & $18.752821$ & $0.216590$ & $0.242237$\\				
$3$ & $18.513934$ & $0.328934$ & $0.369679$ & $18.523194$ & $0.324269$ & $0.364185$\\				
$4$ & $18.286698$ & $0.435044$ & $0.490805$ & $18.300406$ & $0.428303$ & $0.482940$\\				
$5$ & $18.066244$ & $0.537556$ & $0.608748$ & $18.084267$ & $0.528808$ & $0.598573$\\				
$6$ & $17.852383$ & $0.636586$ & $0.723579$ & $17.874593$ & $0.625901$ & $0.711155$\\				
$7$ & $17.644931$ & $0.732247$ & $0.835369$ & $17.671201$ & $0.719691$ & $0.820757$\\				
$8$ & $17.443709$ & $0.824649$ & $0.944190$ & $17.473918$ & $0.810285$ & $0.927446$\\				
$9$ & $17.248540$ & $0.913898$ & $1.050109$ & $17.282570$ & $0.897787$ & $1.031291$\\				
$10$ & $17.059256$ & $1.000097$ & $1.153195$ & $17.096990$ & $0.982299$ & $1.132359$\\[1ex]
\bottomrule
\end{tabular}
\end{center}
\end{minipage}
\end{table}

In Figure~\ref{fig:fig6}, we present a graphical comparison between
ABM and NPM approximations for the state functions $L_{1}$,
$L_{2}$, and $L_{3}$ in the exponentially decaying input case
with time $s\in[0,60]$. From Figure~\ref{fig:fig6}, it can be concluded that
the two introduced methods strongly agree with each other even in a large
time domain of 60 years.
\begin{figure}[ht!]
\centering
\subfigure[]{
\includegraphics[width=0.48\linewidth]{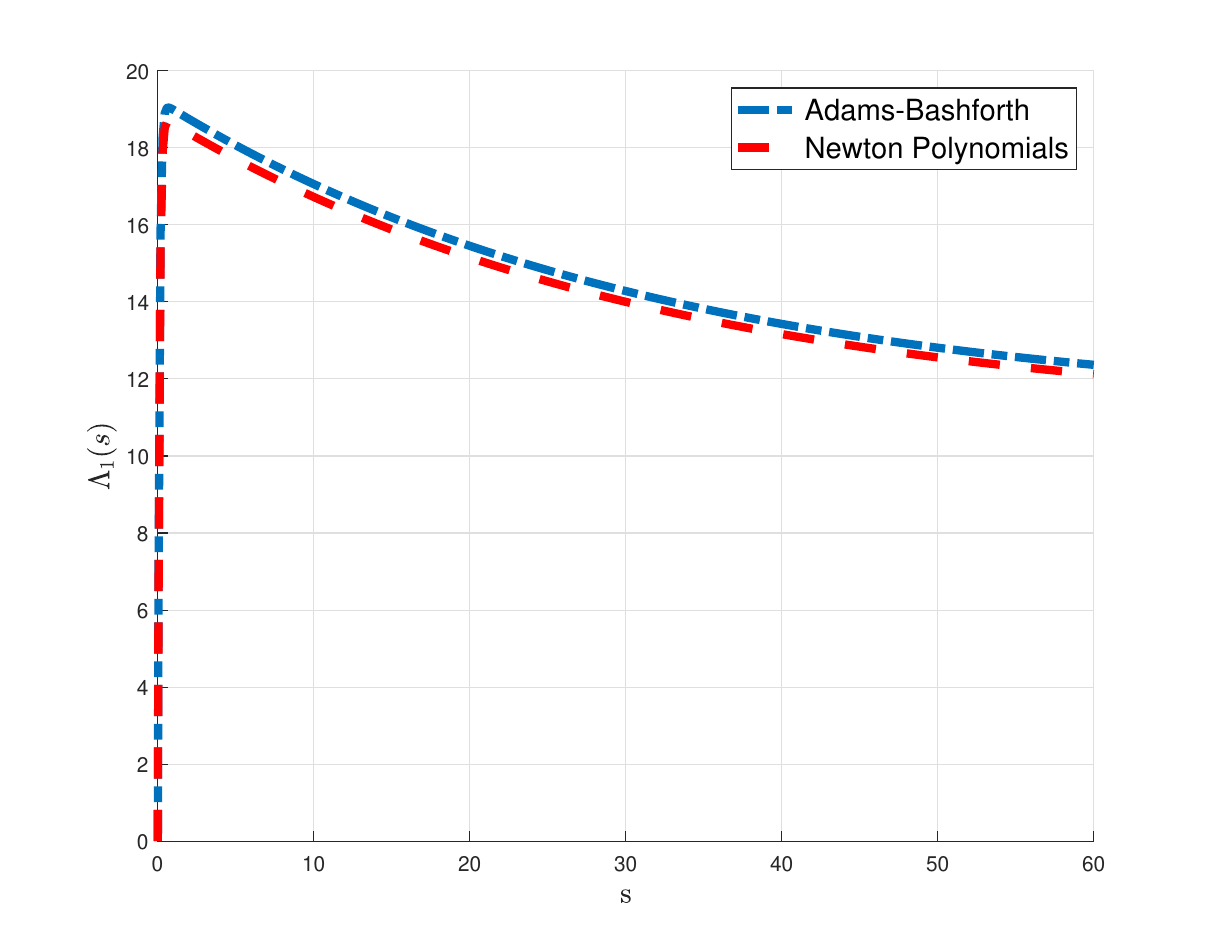}
\label{fig_exponential1}}
\subfigure[]{
\includegraphics[width=0.48\linewidth]{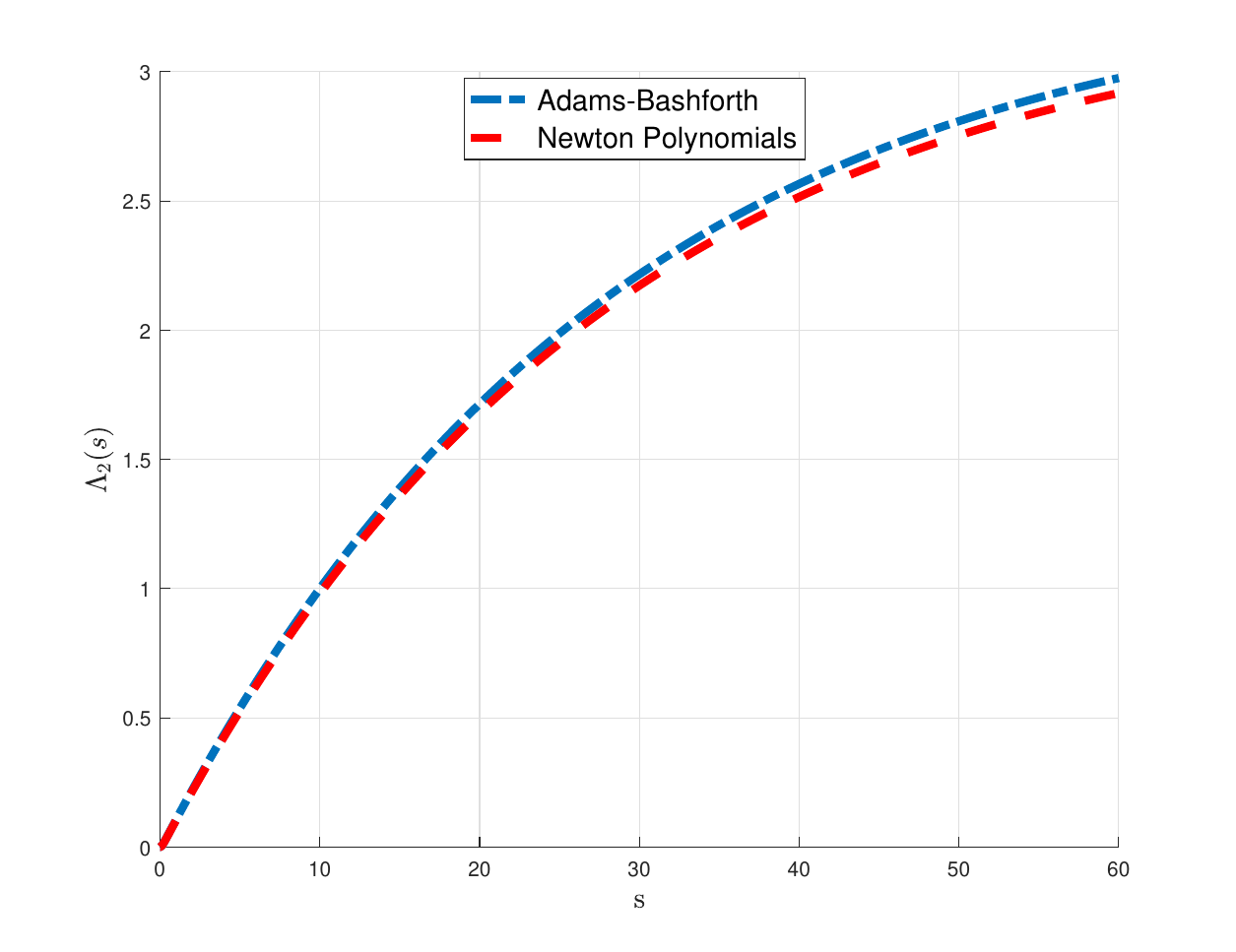}
\label{fig_exponential2}}
\newline
\hfill
\subfigure[]{
\includegraphics[width=0.48\linewidth]{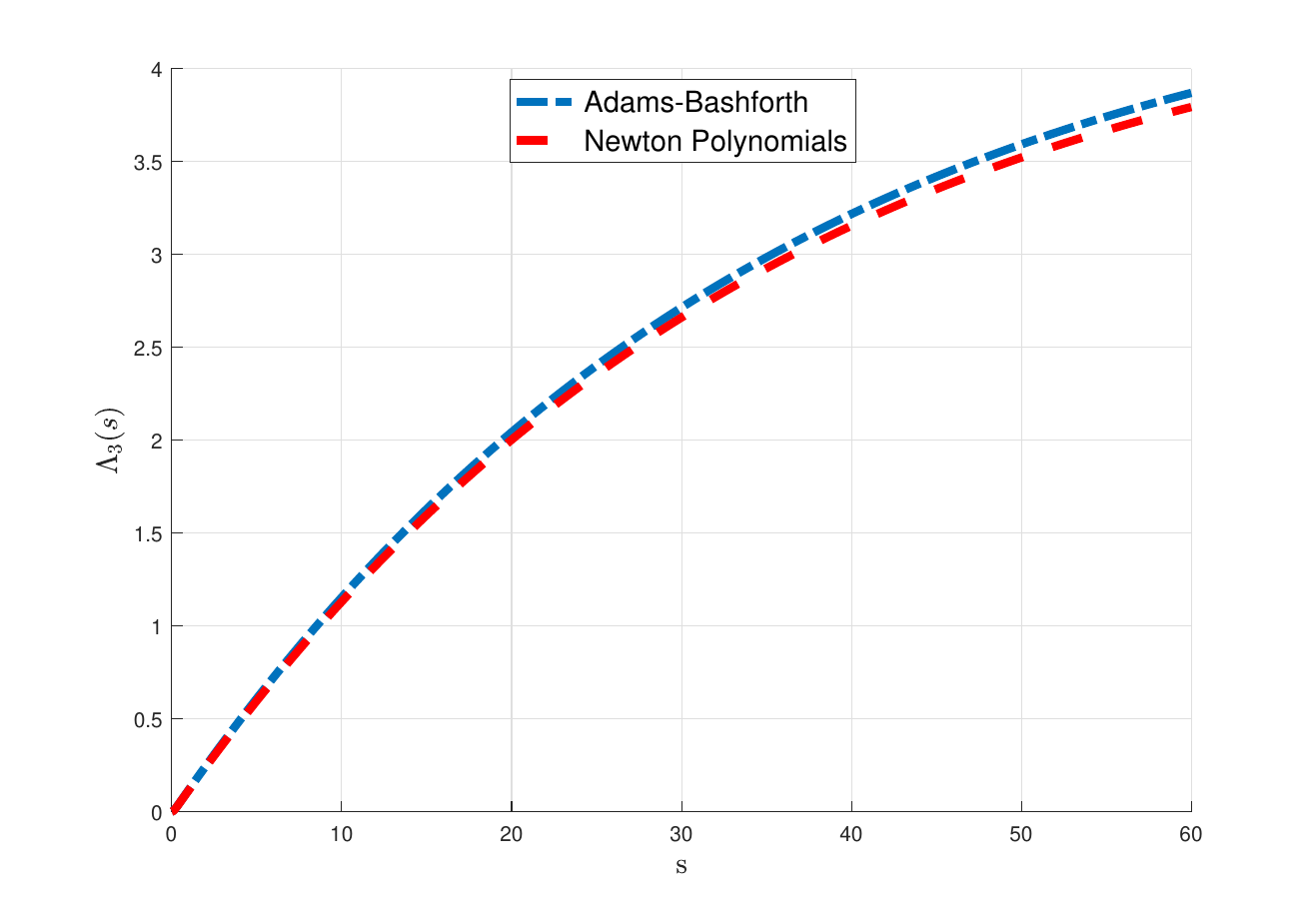}
\label{fig_exponential3}}
\caption{Comparison between the ABM and NPM for
(a) $L_{1}(s)$, (b) $L_{2}(s)$
and (c) $L_{3}$ in the exponentially decaying input model.}
\label{fig:fig6}
\end{figure}

In Figure~\ref{fig:figfracexpo}, we illustrate the numerical results
of the three state variables $L_{1}, L_{2}$, and $L_{3}$
for the exponentially decaying input model when the ABM is applied under
various fractal-fractional orders: $\theta=\sigma=0.85, 0.90, 0.95, 0.99$.
Figure~\ref{fig:figfracexpo} shows that the non-integer fractal-fractional operators
have an effect on decreasing the amount of pollution for each model when the time $s$ increases,
that is, the pollution is increasing harmoniously with the fractal-fractional order,
getting closer to the integer-order case.
\begin{figure}[ht!]
\centering
\subfigure[]{
\includegraphics[width=0.48\linewidth]{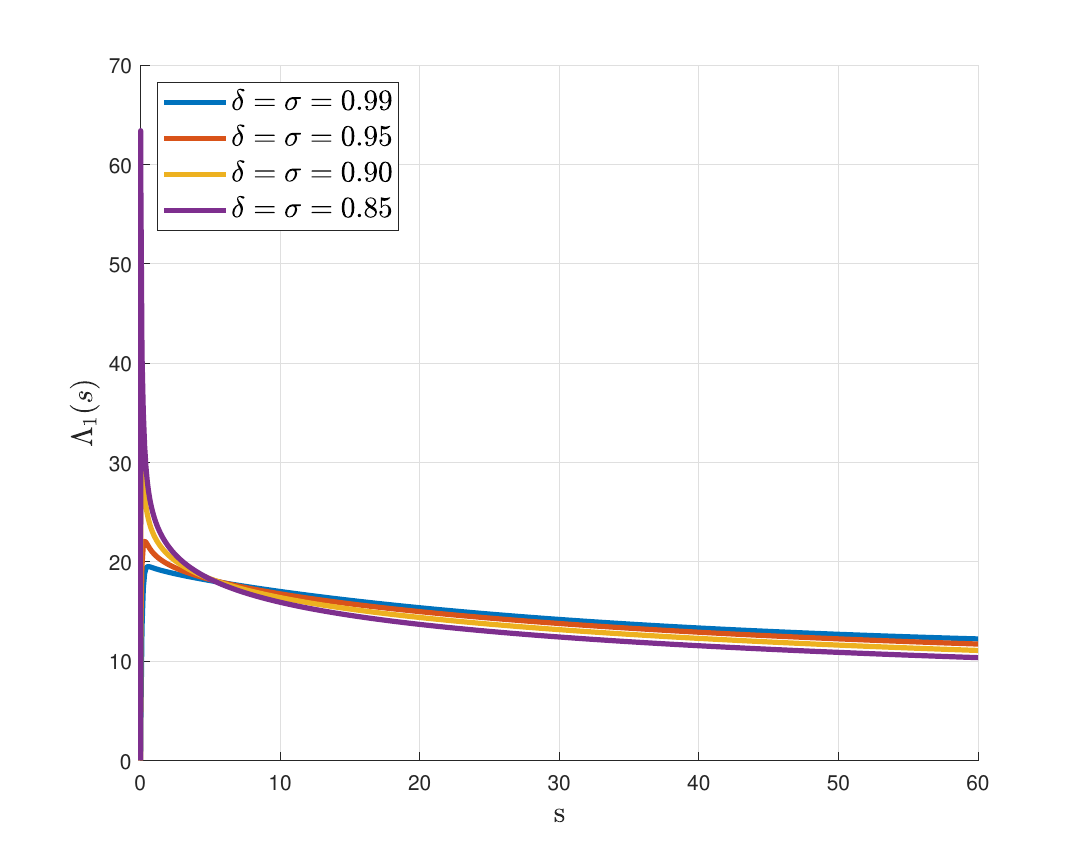}
\label{fig_fracexpo1}}
\subfigure[]{
\includegraphics[width=0.48\linewidth]{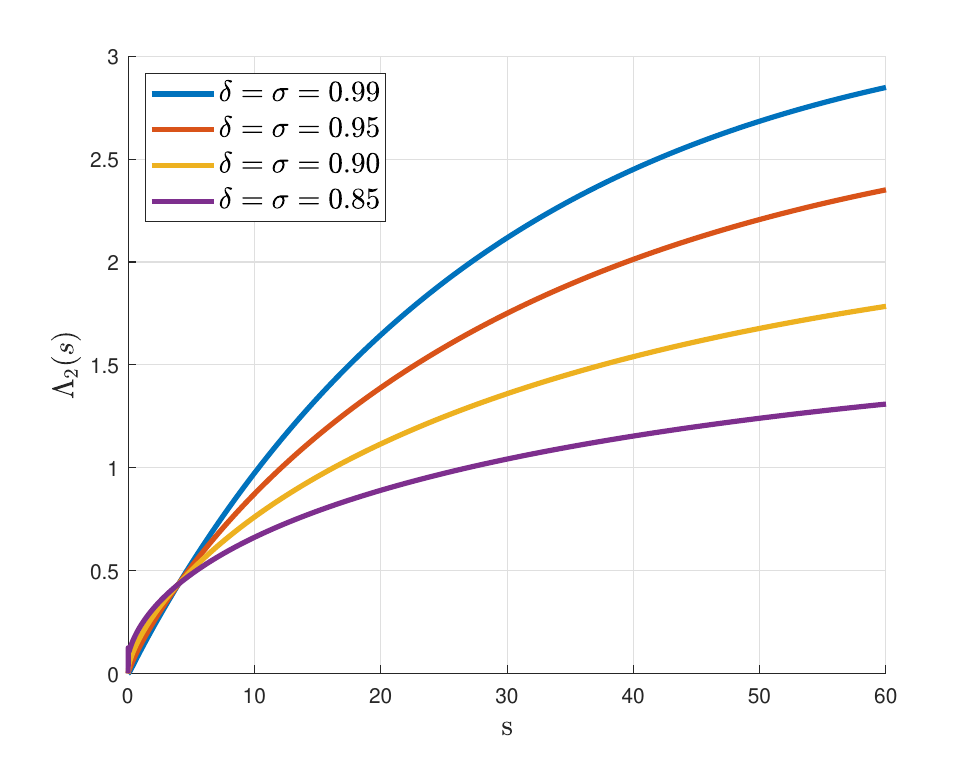}
\label{fig_fracexpo2}}
\newline
\hfill
\subfigure[]{
\includegraphics[width=0.48\linewidth]{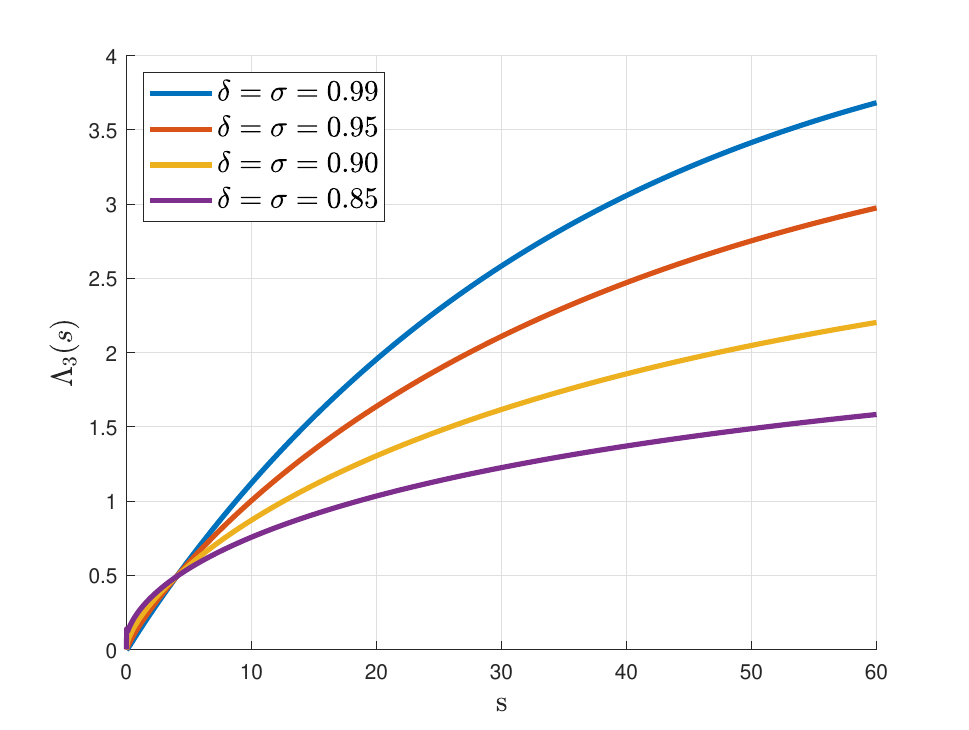}
\label{fig_fracexpo3}}
\caption{Behaviors of (a) $L_{1}(s)$, (b) $L_{2}(s)$ and
(c) $L_{3}(s)$ for the exponentially decaying input model under
fractal and fractional orders $\theta=\sigma=0.85, 0.90, 0.95, 0.99$.}
\label{fig:figfracexpo}
\end{figure}


\subsection{Periodic input model}
\label{sec:8.3}

As a last case of study, we consider a periodic input model in which
the pollutant appears in the lake periodically. A factory
that works during daytime only, can be an example of this case:
it generates waste and dump it in the lakes during the day
while at night the mixing of new pollutants stops.
For a concrete case, we selected $c(s)=a+\tau \sin(bs)$,
where $\tau$ and $b$ stands for the variations of amplitude and frequency,
respectively. Also, $a$ is considered as the average input of pollutant concentration.
In such a case $a=b=\tau=1$, system \eqref{model2} takes the following form:
\begin{equation}
\label{model2x4}
\begin{cases}
{}^{\mathbf{FFML}}\mathcal{D}_{0,s}^{(\theta,\sigma)}{L_1}(s)
= \frac{38}{1180}L_3(s)+1+\sin(s)-\frac{20}{2900}L_1(s)
-\frac{18}{2900}L_1(s),\\[0.3cm]
{}^{\mathbf{FFML}}\mathcal{D}_{0,s}^{(\theta,\sigma)}{L_2}(s)
= \frac{18}{2900}L_1(s)-\frac{18}{850}L_2(s),\\[0.3cm]
{}^{\mathbf{FFML}}\mathcal{D}_{0,s}^{(\theta,\sigma)}{L_3}(s)
=\frac{20}{2900}L_1(s)+\frac{18}{850}L_2(s)-\frac{38}{1180}L_3(s).
\end{cases}
\end{equation}

The graphical representation of the input function $c(s)$ is illustrated
in Figure~\ref{fig4x} for the periodic input case
$c(s)=1+\sin(s)$, $s\in[0,20]$.
\begin{figure}[ht!]
\centering
\includegraphics[width=0.6\linewidth]{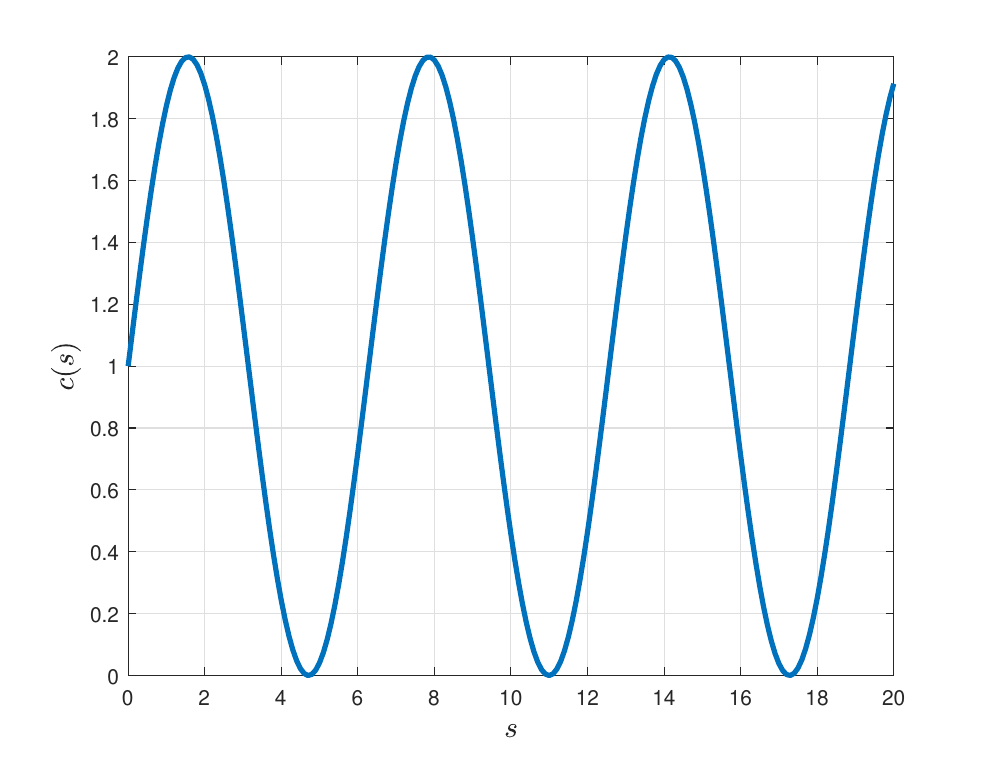}
\caption{Graphic of the periodic input.}
\label{fig4x}
\end{figure}
\begin{figure}[ht!]
\centering
\subfigure[]{
\includegraphics[width=0.48\linewidth]{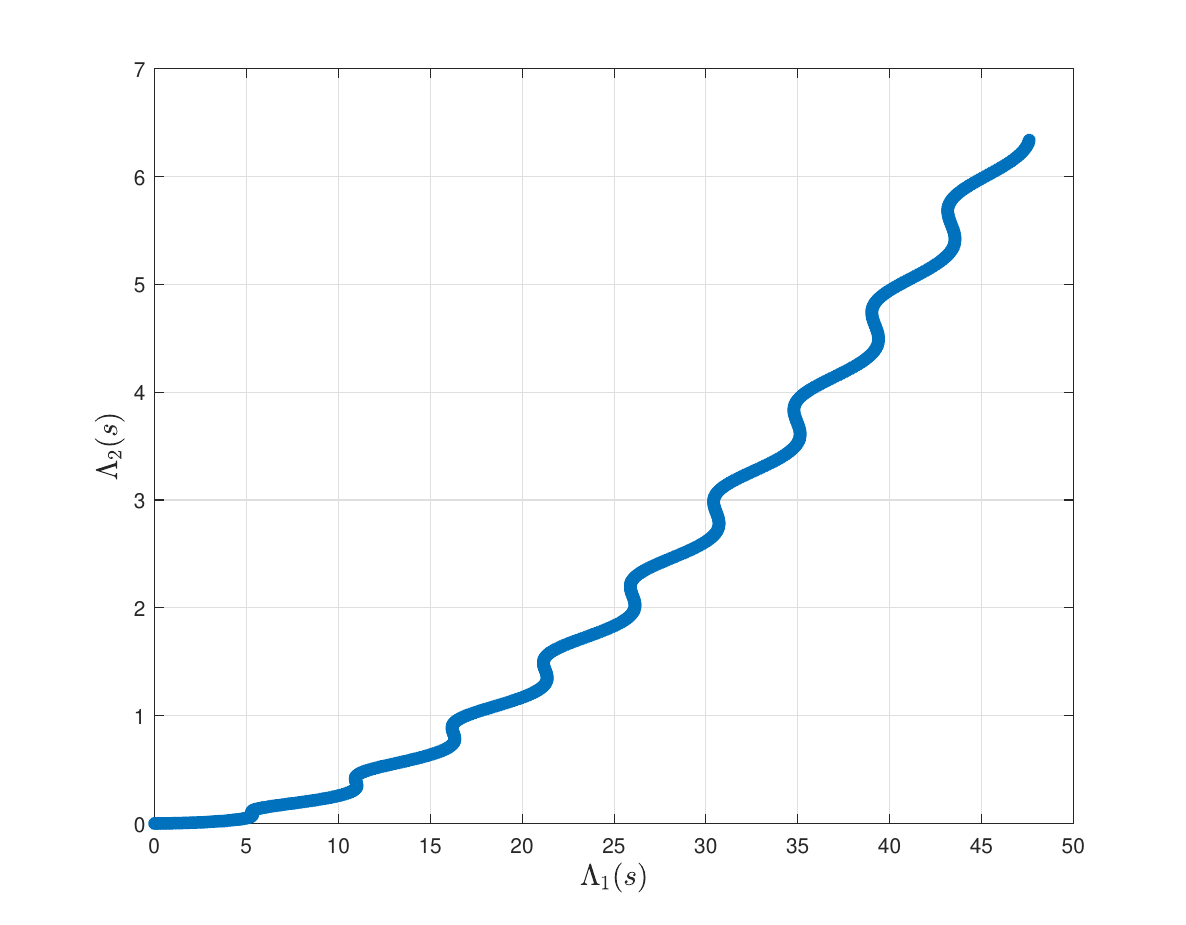}
\label{fig:sub4-first}}
\subfigure[]{
\includegraphics[width=0.48\linewidth]{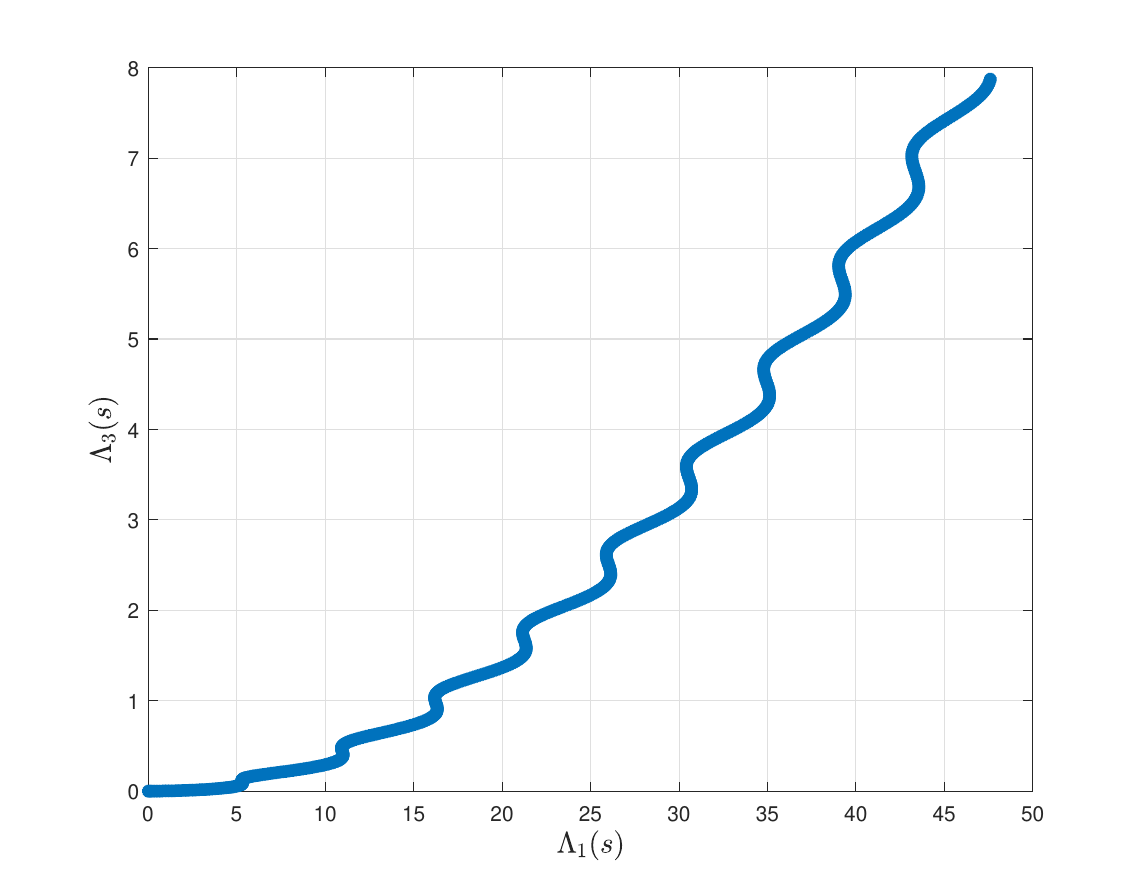}
\label{fig:sub4-second}}
\newline
\subfigure[]{
\includegraphics[width=0.48\linewidth]{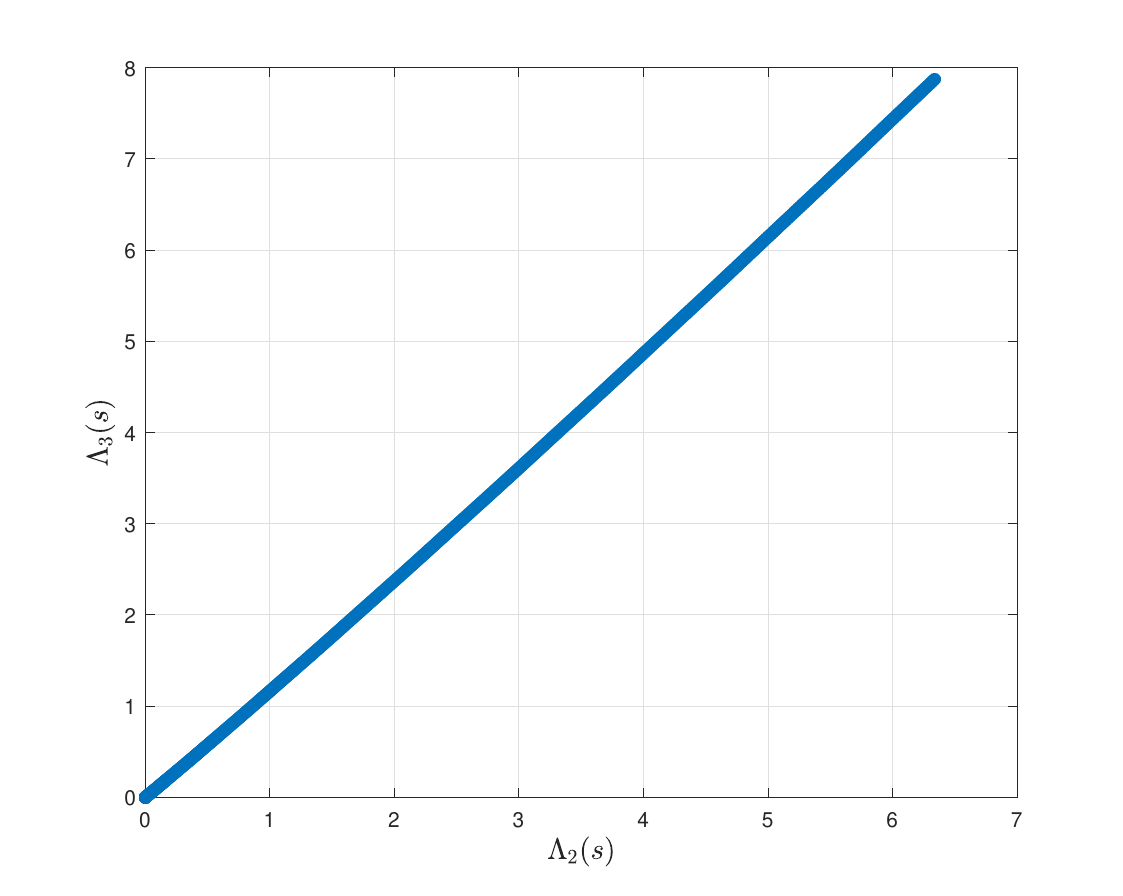}
\label{fig:sub4-third}}
\subfigure[]{
\includegraphics[width=0.48\linewidth]{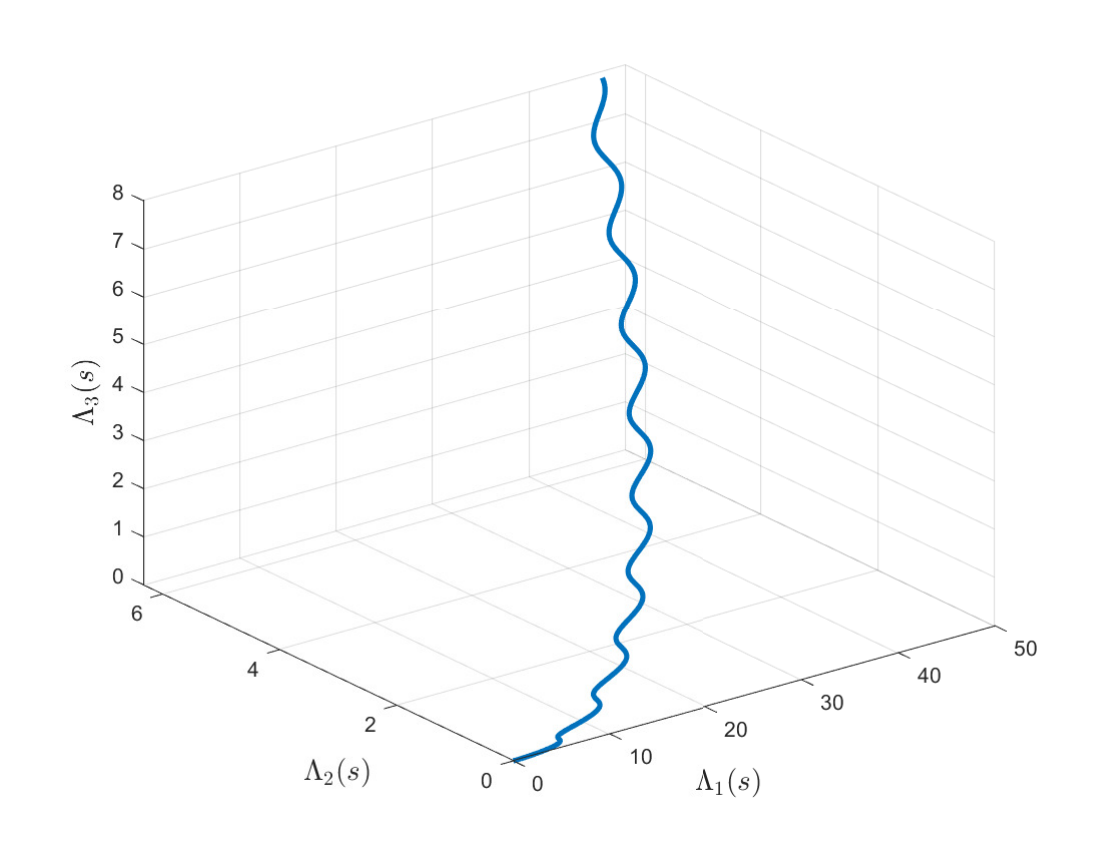}
\label{fig:sub4-fourth}}
\caption{Behaviors of each pair of state functions
(a) $L_{1}-L_{2}$, (b) $L_{1}-L_{3}$,
(c) $L_{2}-L_{3}$ and
(d)~3D view of $L_{1}-L_{2}-L_{3}$ under the integer-order.}
\label{fig:figsub4}
\end{figure}

In Figures~\ref{fig:figsub4} (a), (b), and (c),
the graphical behavior of each pair of the state functions
$L_{1}-L_{2}$, $L_{1}-L_{3}$, and $L_{2}-L_{3}$,
respectively, is shown. In Figure~\ref{fig:figsub4} (d), the 3D view of
$L_{1}-L_{2}-L_{3}$ under the integer-order derivative
is illustrated for the periodic input model
with time $s\in[0,60]$ and step size $h=0.1$.

The tabular comparison between the numerical results obtained from the proposed
techniques, ABM and NPM, for the three state functions $L_{1}$, $L_{2}$,
and $L_{3}$ under the periodic input case, are reported in Table~\ref{table:3}
for time $s\in[0,10]$, step size $h=0.1$, and $\theta=\sigma=1$.
From these results, we conclude that the solutions obtained
by ABM and NPM highly agree with each other.
\begin{table}[ht!]
\begin{minipage}[t]{1\linewidth}
\makeatletter\def\@captype{table}\makeatother
\caption{Comparison between ABM and NPM in the periodic input case.}
\label{table:3}
\begin{center}
\begin{tabular}{c c c c c c c} \toprule
\multirow{2.5}{*}{$s$} & \multicolumn{3}{c}{Adams--Bashforth}
& \multicolumn{3}{c}{Newton Polynomials}\\
\cmidrule(lr){2-4} \cmidrule(lr){5-7} 	& $L_{1}$
&  $L_{2}$  &  $L_{3}$  &  $L_{1}$
&  $L_{2}$  &  $L_{3}$\\ \midrule
$0$ & $0.000000$ & $0.00000$ & $0.000000$ & $0.000000$ & $0.000000$ & $0.000000$\\
$1$ & $1.420794$ & $0.003639$ & $0.004053$ & $1.420289$ & $0.003637$ & $0.004052$\\
$2$ & $3.352256$ & $0.018270$ & $0.020396$ & $3.349409$ & $0.017505$ & $0.019538$\\				
$3$ & $4.795526$ & $0.043383$ & $0.048562$ & $4.791798$ & $0.041836$ & $0.046821$\\
$4$ & $5.304589$ & $0.073971$ & $0.083050$ & $5.303706$ & $0.071660$ & $0.080440$\\
$5$ & $5.281611$ & $0.104988$ & $0.118274$ & $5.286096$ & $0.101974$ & $0.114856$\\
$6$ & $5.607511$ & $0.135825$ & $0.153543$ & $5.616323$ & $0.132176$ & $0.149388$\\
$7$ & $6.832362$ & $0.170701$ & $0.193553$ & $6.841809$ & $0.166454$ & $0.188698$\\
$8$ & $8.669954$ & $0.214629$ & $0.243923$ & $8.677051$ & $0.209775$ & $0.238353$\\
$9$ & $10.261233$ & $0.268650$ & $0.305891$ & $10.266407$ & $0.263162$ & $0.299573$\\				
$10$ & $10.964396$ & $0.328730$ & $0.375035$ & $10.971053$ & $0.322611$ & $0.367966$\\[1ex]
\bottomrule
\end{tabular}
\end{center}
\end{minipage}
\end{table}

In Figure~\ref{fig:fig8}, we illustrate our findings graphically,
comparing the numerical results from ABM and NPM
for each state function $L_{1}$, $L_{2}$, and $L_{3}$,
where the Lake~1 has a periodic pollutant input. Figure~\ref{fig:fig8}
shows that the two introduced techniques, ABM and NPM, strongly agree
with each other for the time $s\in [0,60]$, step size $h=0.1$, and $\theta=\sigma=1$.
\begin{figure}[ht!]
\centering
\subfigure[]{
\includegraphics[width=0.48\linewidth]{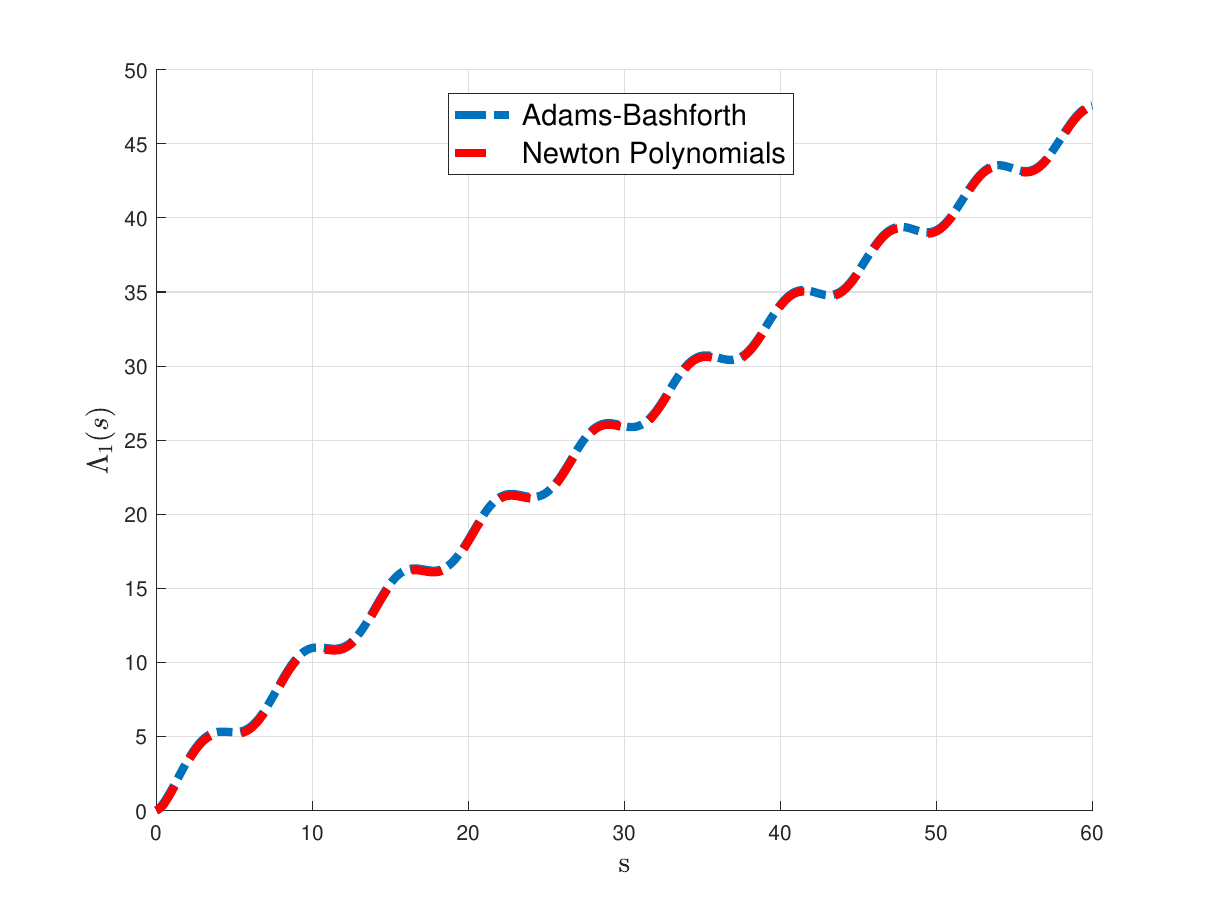}
\label{fig_periodic1}}
\subfigure[]{
\includegraphics[width=0.48\linewidth]{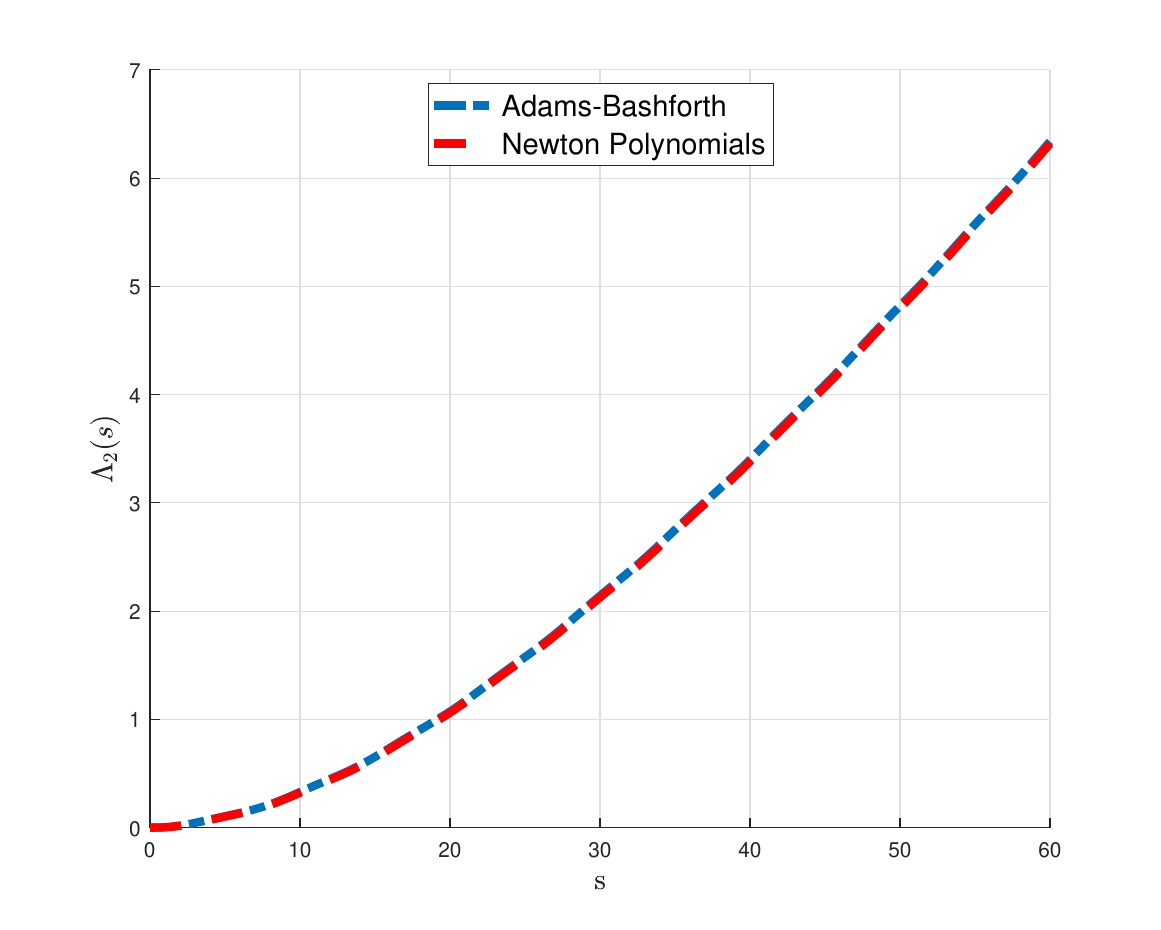}
\label{fig_periodic2}}
\newline
\hfill
\subfigure[]{
\includegraphics[width=0.48\linewidth]{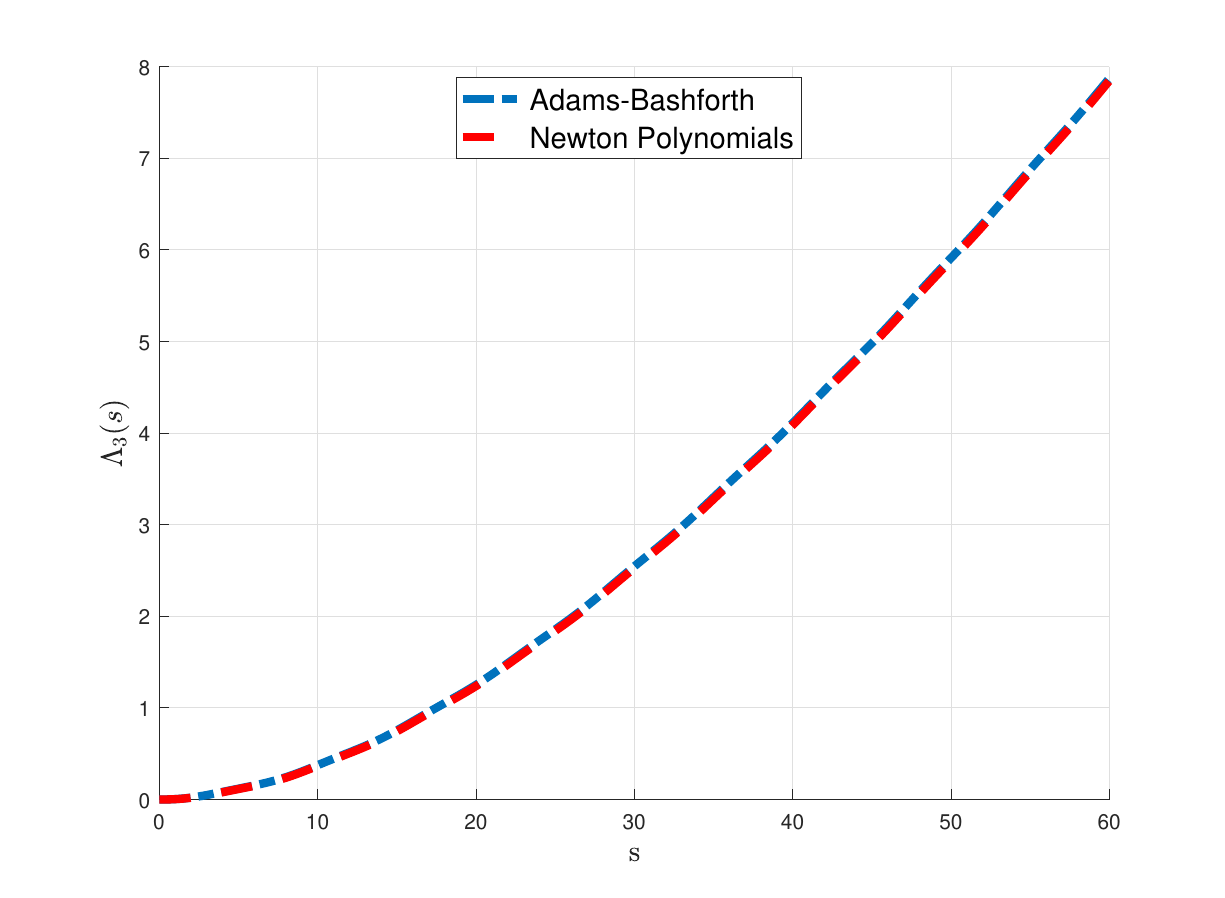}
\label{fig_periodic3}}
\caption{Comparison between the ABM and NPM for (a) $L_{1}(s)$,
(b) $L_{2}(s)$ and (c) $L_{3}$ in the periodic input model.}
\label{fig:fig8}
\end{figure}

In Figure~\ref{fig:figfracperiod}, we illustrate the ABM approximations
of the three state functions $L_{1}$, $L_{2}$, and $L_{3}$
under various fractal-fractional orders: $\theta=\sigma=0.85$, $0.90$,
$0.95$, $0.99$ for the periodic input case. Similar to cases
of Sections~\ref{sec:8.1} and \ref{sec:8.2}, we observe that the non-integer order
fractal-fractional operators have a great effect on decreasing the
amount of contamination for each model while the time $s$ increases.
\begin{figure}[ht!]
\centering
\subfigure[]{
\includegraphics[width=0.48\linewidth]{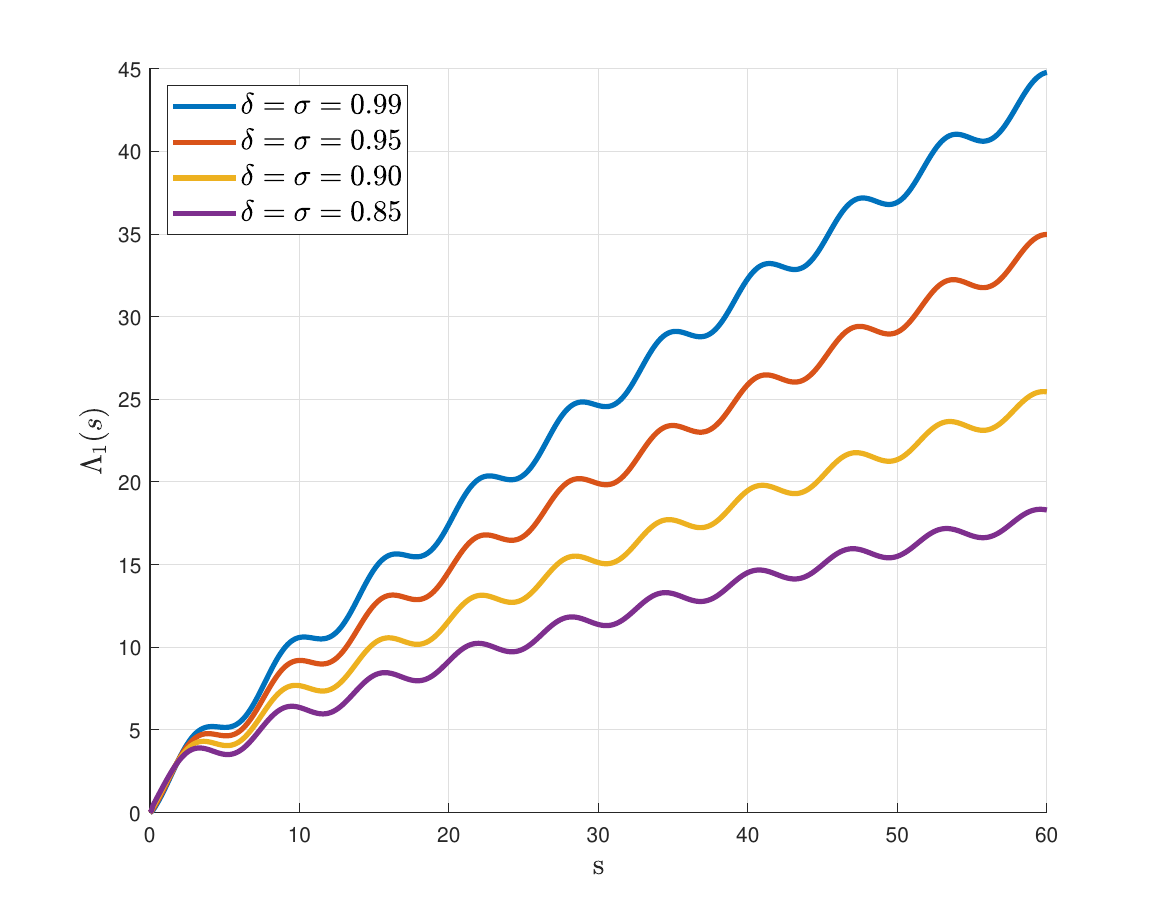}
\label{fig_fracperiod1}}
\subfigure[]{
\includegraphics[width=0.48\linewidth]{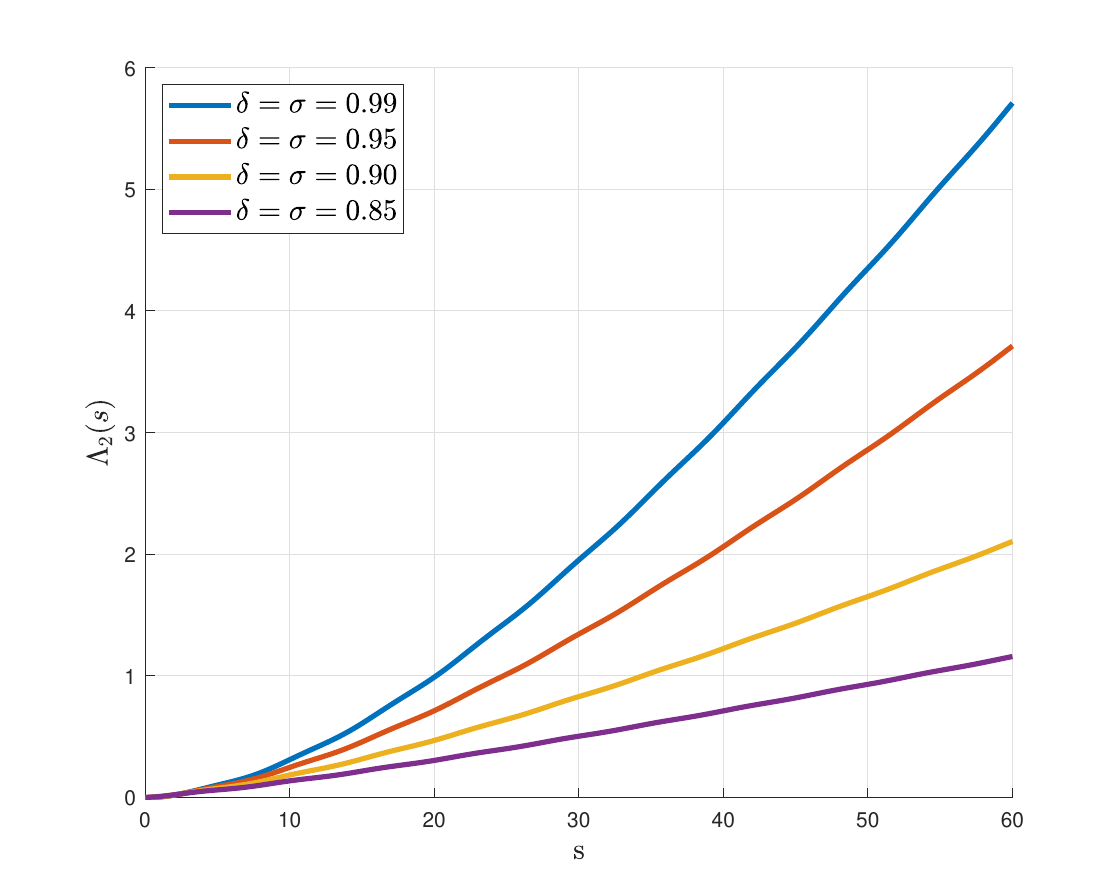}
\label{fig_fracperiod2}}
\newline
\hfill
\subfigure[]{
\includegraphics[width=0.48\linewidth]{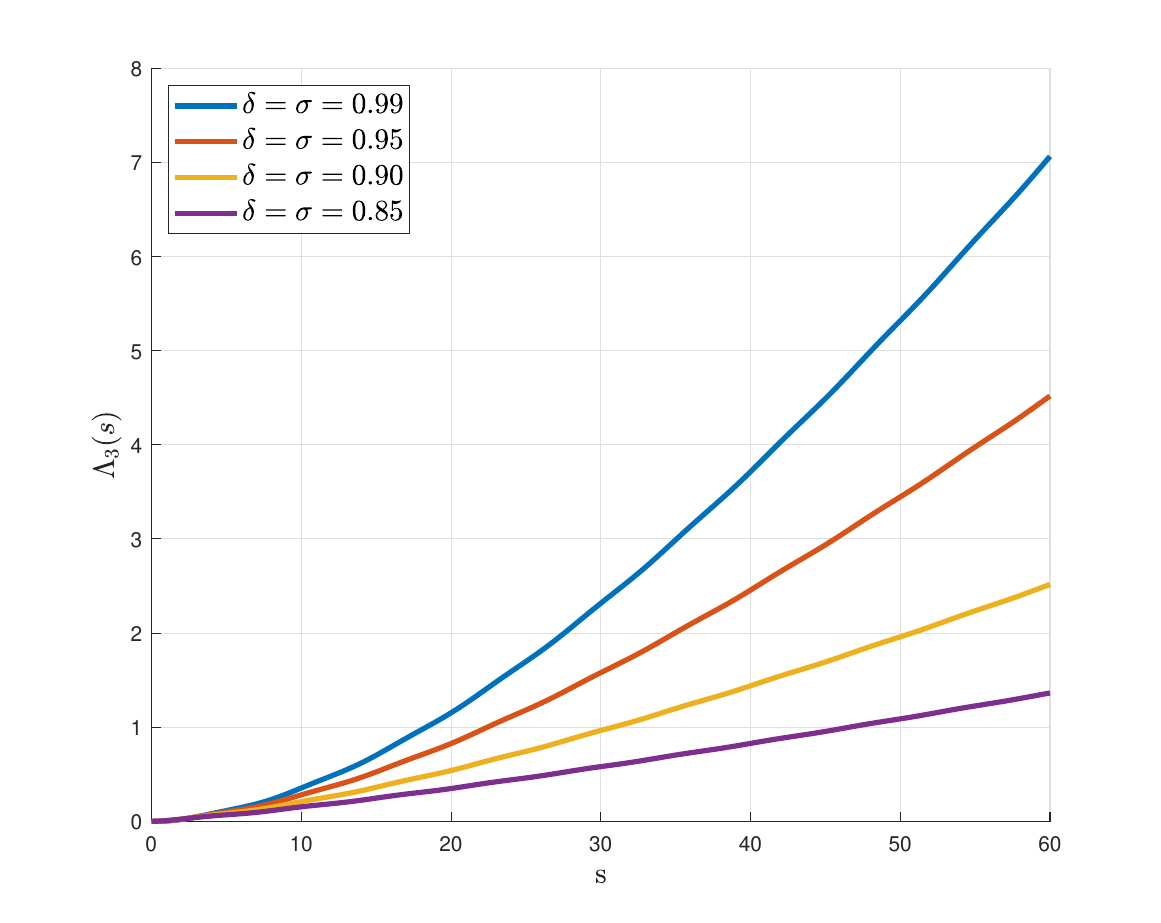}
\label{fig_fracperiod3}}
\caption{Behaviors of (a) $L_{1}(s)$, (b) $L_{2}(s)$
and (c) $L_{3}(s)$ for the periodic input model under
fractal and fractional orders $\theta=\sigma=0.85, 0.90, 0.95, 0.99$.}
\label{fig:figfracperiod}
\end{figure}


\section{Conclusion}
\label{sec:9}

We employed Mittag--Leffler type kernels to solve a system of fractional
differential equations using fractal-fractional (FF) operators with two
fractal and fractional orders. We derived equivalent FF-integral equations
from a compact initial value problem, and then proved existence and uniqueness
results. A stability analysis was conducted in different versions.
In the next sections, we examined and captured the behavior of the considered
fractal-fractional operator model \eqref{model2} with the help of two
different numerical techniques: an Adams--Bashforth method (ABM)
and a Newton polynomials method (NPM). From the obtained results,
we conclude that the considered techniques, ABM and NPM, are in highly agreement
and are very efficient to examine the system of fractional differential equations
under fractal-fractional operators describing the dynamics of the pollution in the lakes.
We also analyzed the considered model under various
fractal-fractional orders and examined the effects
of these non-integer orders on the behavior of each state variable $L_{1}(s)$,
$L_{2}(s)$, and $L_{3}(s)$ for three specific input models:
linear, exponentially decaying, and periodic. For each input model, we observed that
when the fractal-fractional order gets closer to the classical integer-order case, then
the effect of the pollution is increasing harmoniously for each lake model.
As a conclusion of these observations, it can be said that the non-integer order
operators have positive effects on the reduction of pollution in the lake pollution model.
As future work, we plan to investigate different real-world models
based on the techniques here developed.


\subsection*{CRediT authorship contribution statement}

\textbf{Tanzeela Kanwal:} Formal analysis, Methodology.
\textbf{Azhar Hussain:} Conceptualization, Formal analysis, Methodology.
\textbf{\.{I}brahim Avc{\i}:} Formal analysis, Methodology, Software.
\textbf{Sina Etemad:} Conceptualization, Methodology, Software.
\textbf{Shahram Rezapour:} Conceptualization, Methodology.
\textbf{Delfim F. M. Torres:} Formal analysis, Funding acquisition, Methodology.


\subsection*{Declaration of competing interest}

The authors declare that they have no known competing financial 
interests or personal relationships that could have appeared to
influence the work reported in this paper.


\subsection*{Data availability}

No data was used for the research described in the article.


\subsection*{Acknowledgments}

Torres was supported by the Portuguese Foundation
for Science and Technology (FCT), project UIDB/04106/2020
(\url{https://doi.org/10.54499/UIDB/04106/2020}).




\small

\begin{thebibliography}{99}
	
\bibitem{baz}
J. Biazar, L. Farrokhi, M. R. Islam,
Modeling the pollution of a system of lakes.
{\it Appl. Math. Comput.}, {\bf 178}(2) (2006) 423--430.
\url{https://doi.org/10.1016/j.amc.2005.11.056}

\bibitem{wv1}
\c{S}. Y\"{u}zba\c{s}i, N. \c{S}ahin, M. Sezer,
A collocation approach to solving the model of pollution for a system of lakes,
{\it Math. Computer Model.}, {\bf 55}(3-4) (2012) 330--341.
\url{https://doi.org/10.1016/j.mcm.2011.08.007}

\bibitem{wv2}
B. Benhammouda, H. Vazquez-Leal, L. Hernandez-Martinez,
Modified differential transform method for solving the model of pollution for a system of lakes,
{\it Discr. Dyn. Nat. Soc.}, {\bf 2014}  (2014) 645726.
\url{https://doi.org/10.1155/2014/645726}

\bibitem{khad}
M. M. Khader, T. S. El Danaf, A. S. Hendy,
A computational matrix method for solving systems of high order fractional differential equations,
{\it Appl. Math. Model.}, {\bf 37}(6) (2013) 4035--4050.
\url{https://doi.org/10.1016/j.apm.2012.08.009}

\bibitem{pp1}
N. Bildik, S. Deniz,
A new fractional analysis on the polluted lakes system,
{\it Chaos, Solitons \& Fractals}, {\bf 122} (2019) 17--24.
\url{https://doi.org/10.1016/j.chaos.2019.02.001}

\bibitem{pp2}
M. M. D. Ahmed, M. A. Khan,
Modeling and analysis of the polluted lakes system with various fractional approaches,
{\it Chaos, Solitons \& Fractals}, {\bf 134} (2020) 109720.
\url{https://doi.org/10.1016/j.chaos.2020.109720}

\bibitem{pp3}
D. G. Prakasha, P. Veeresha,
Analysis of Lakes pollution model with Mittag-Leffler kernel,
{\it J. Ocean Eng. Sci.}, {\bf 5}(4) (2020) 310--322.
\url{https://doi.org/10.1016/j.joes.2020.01.004}

\bibitem{pp4}
B. Shiri, D. Baleanu,
A General Fractional Pollution Model for Lakes,
{\it Commun. Appl. Math. Comput.}, {\bf 4} (2022) 1105--1130.
\url{https://doi.org/10.1007/s42967-021-00135-4}

\bibitem{z1}
A. Alsaedi, M. Alsulami, H. M. Srivastava, B. Ahmad, S. K. Ntouyas,
Existence theory for nonlinear third-order ordinary differential equations
with nonlocal multi-point and multi-strip boundary conditions,
{\it Symmetry}, {\bf 11}(2) (2019) 281.
\url{https://doi.org/10.3390/sym11020281}

\bibitem{MR4376325}
M. R. Sidi Ammi, M. Tahiri, D. F. M. Torres,
Necessary optimality conditions of a reaction-diffusion
SIR model with ABC fractional derivatives,
{\it Discrete Contin. Dyn. Syst. Ser. S} {\bf 15} (2022), no.~3, 621--637.
\url{https://doi.org/10.3934/dcdss.2021155}
{\tt arXiv:2106.15055}

\bibitem{z4}
M. Aslam, R. Murtaza, T. Abdeljawad, G. ur Rahman, A. Khan, H. Khan, H. Gulzar,
A fractional order HIV/AIDS epidemic model with Mittag-Leffler kernel,
{\it Adv. Differ. Equ.}, {\bf 2021} (2021), 107, 1--5.
\url{https://doi.org/10.1186/s13662-021-03264-5}

\bibitem{z5}
S. W. Ahmad, M. Sarwar, G. Rahmat, K. Shah, H. Ahmad, A. A. A. Mousa,
Fractional order model for the Coronavirus (COVID-19) in Wuhan, China,
{\it Fractals}, {\bf 30} (2022), 2240007.
\url{https://doi.org/10.1142/S0218348X22400072}

\bibitem{z7}
J. K. K. Asamoah, E. Okyere, E. Yankson, A. A. Opoku, A. Adom-Konadu, E. Acheampong, Y. D. Arthur,
Non-fractional and fractional mathematical analysis and simulations for Q fever,
{\it Chaos, Solitons \& Fractals}, {\bf 156} (2022), 111821.
\url{https://doi.org/10.1016/j.chaos.2022.111821}

\bibitem{z9}
H. Khan, Y. G. Li, W. Chen, D. Baleanu, A. Khan,
Existence theorems and Hyers-Ulam stability for a coupled system
of fractional differential equations with $p$-Laplacian operator,
\textit{Bound. Value Probl.}, {\bf 2017} (2017), 157, 1--16.
\url{https://doi.org/10.1186/s13661-017-0878-6}

\bibitem{MR4536746}
Z. Ali, F. Rabiei\ and\ K. Hosseini,
A fractal-fractional-order modified Predator-Prey mathematical model with immigrations,
Math. Comput. Simulation {\bf 207} (2023), 466--481.
\url{https://doi.org/10.1016/j.matcom.2023.01.006}

\bibitem{ab}
A. Atangana,
Fractal-fractional differentiation and integration: connecting fractal calculus
and fractional calculus to predict complex system,
{\it Chaos, Solitons \& Fractals}, {\bf 102} (2017) 396--406.
\url{https://doi.org/10.1016/j.chaos.2017.04.027}

\bibitem{FF01}
K. Shah, M. Arfan, I. Mahariq, A. Ahmadian, S. Salahshour, M. Ferrara,
Fractal-fractional mathematical model addressing the situation of Corona virus in Pakistan,
{\it Res. Phys.}, {\bf 19} (2020) 103560.
\url{https://doi.org/10.1016/j.rinp.2020.103560}

\bibitem{FF02}
J. F. Gomez-Aguilar, T. Cordova-Fraga, T. Abdeljawad, A. Khan, H. Khan,
Analysis of fractal-fractional Malaria transmission model,
{\it Fractals}, {\bf 28}(08) (2020) 2040041.
\url{https://doi.org/10.1142/S0218348X20400411}

\bibitem{FF03}
Z. Ali, F. Rabiei, K. Shah, T. Khodadadi,
Qualitative analysis of fractal-fractional order COVID-19 mathematical
model with case study of Wuhan,
{\it Alex. Eng. J.}, {\bf 60}(1) (2021) 477--489.
\url{https://doi.org/10.1016/j.aej.2020.09.020}

\bibitem{z11b}
S. Etemad, I. Avci, P. Kumar, D. Baleanu, S. Rezapour,
Some novel mathematical analysis on the fractal-fractional
model of the AH1N1/09 virus and its generalized Caputo-type version,
{\it Chaos, Solitons \& Fractals}, {\bf 162} (2022) 112511.
\url{https://doi.org/10.1016/j.chaos.2022.112511}

\bibitem{FF04}
J. K. K. Asamoah,
Fractal-fractional model and numerical scheme based on Newton
polynomial for Q fever disease under Atangana-Baleanu derivative,
{\it Res. Phys.}, {\bf 34} (2022) 105189.
\url{https://doi.org/10.1016/j.rinp.2022.105189}

\bibitem{FF}
S. Ahmad, A. Ullah, A. Akgul, M. De la Sen,
Study of HIV disease and its association with immune cells
under nonsingular and nonlocal fractal-fractional operator,
{\it Complexity}, {\bf 2021} (2021) 1904067.
\url{https://doi.org/10.1155/2021/1904067}

\bibitem{z12}
H. Najafi, S. Etemad, N. Patanarapeelert, J. K. K. Asamoah, S. Rezapour, T. Sitthiwirattham,
A study on dynamics of CD4$^+$ T-cells under the effect of HIV-1 infection based
on a mathematical fractal-fractional model via the Adams-Bashforth scheme and Newton polynomials,
{\it Mathematics}, {\bf 10} (2022) 1366.
\url{https://doi.org/10.3390/math10091366}

\bibitem{z13}
H. Khan, K. Alam, H. Gulzar, S. Etemad, S. Rezapour,
A case study of fractal-fractional tuberculosis model in China:
Existence and stability theories along with numerical simulations,
{\it Math. Comput. Simul.}, {\bf 198} (2022) 455--473.
\url{https://doi.org/10.1016/j.matcom.2022.03.009}

\bibitem{29}
A. Granas, J. Dugundji,
{\it Fixed Point Theory}, Springer-Verlag, New York (2003).

\bibitem{aa}
A. Atangana, S. \.{I} Araz,
{\it New numerical scheme with Newton polynomial---theory, methods, and applications},
Elsevier/Academic Press, London, 2021.

\end{thebibliography}
\end{document}